\pdfoutput=1
\documentclass[11pt,reqno,final]{amsart}
\usepackage[margin=27mm]{geometry}

\let\oldmarginpar\marginpar
\renewcommand\marginpar[1]{\-\oldmarginpar[\raggedleft\footnotesize #1]{\raggedright\footnotesize #1}}

\usepackage[T1]{fontenc}
\usepackage[utf8]{inputenc}

\usepackage{setspace}
\usepackage{indentfirst}

\RequirePackage{lineno}

\usepackage{amsmath,amssymb,stmaryrd,mathrsfs,bm}
\usepackage{tikz-cd}

\newtheorem{theorem}{Theorem}[section]
\newtheorem{lemma}[theorem]{Lemma}
\newtheorem{proposition}[theorem]{Proposition}
\newtheorem{corollary}[theorem]{Corollary}

\newtheorem*{Aa}{Addendum to Theorem~\ref{Ohio}}
\newtheorem*{Ab}{Addendum to Theorem~\ref{Montana}}

\theoremstyle{definition}
\newtheorem{definition}[theorem]{Definition}
\newtheorem{example}[theorem]{Example}

\theoremstyle{remark}
\newtheorem{remark}[theorem]{Remark}

\newtheoremstyle{intro}
{12pt}
{12pt}
{\itshape}
{}
{\bfseries}
{.}
{.5em}
{}

\theoremstyle{intro}
\newtheorem{introthm}{Theorem}

\usepackage[backend=biber,bibencoding=utf8,style=numeric,
url=false,doi=true,eprint=true,isbn=false,
hyperref=auto,backref=false]{biblatex}
\usepackage[pagebackref=false,unicode=true,bookmarks=true,
colorlinks=true,linktoc=page,linkcolor=black,citecolor=black,
pdfstartview={FitH},final=true]{hyperref}

\usepackage[color,notref,notcite]{showkeys}
\definecolor{refkey}{rgb}{1,0,0}
\definecolor{labelkey}{rgb}{0,0,1}
\definecolor{red}{rgb}{1,0,0}

\binoppenalty=\maxdimen
\relpenalty=\maxdimen
\DeclareMathOperator{\Hom}{Hom}
\DeclareMathOperator{\End}{End}
\DeclareMathOperator{\Der}{Der}
\DeclareMathOperator{\id}{id}
\DeclareMathOperator{\ad}{ad}
\DeclareMathOperator{\sym}{sym}
\DeclareMathOperator{\pbw}{pbw}
\DeclareMathOperator{\Bott}{Bott}
\DeclareMathOperator{\sgn}{sgn}
\DeclareMathOperator{\LC}{LC}
\DeclareMathOperator{\CE}{CE}
\DeclareMathOperator{\Gr}{Gr}

\newcommand{\Koszul}[1]{\sgn(#1)}
\newcommand{\argument}{\mathord{\color{black!25}-}}
\newcommand{\abs}[1]{\left| #1 \right|}
\newcommand{\into}{\hookrightarrow}
\newcommand{\onto}{\twoheadrightarrow}
\newcommand{\xto}[1]{\xrightarrow{#1}}
\newcommand{\NN}{\mathbb{N}}
\newcommand{\ZZ}{\mathbb{Z}}
\newcommand{\RR}{\mathbb{R}}
\newcommand{\CC}{\mathbb{C}}
\newcommand{\KK}{\Bbbk}
\newcommand{\cinf}[1]{C^{\infty}(#1)}
\newcommand{\shuffle}[2]{\mathfrak{S}_{#1}^{#2}}
\newcommand{\sections}[1]{\Gamma(#1)}
\newcommand{\enveloping}[1]{\mathcal{U}(#1)}
\newcommand{\zo}{^{0,1}}
\newcommand{\oz}{^{1,0}}
\newcommand{\XX}{\mathfrak{X}}
\newcommand{\OO}{\Omega}
\newcommand{\groupoid}[1]{\mathscr{#1}}
\newcommand{\group}[1]{\mathscr{#1}}
\newcommand{\duality}[2]{\left\langle #1 \middle| #2\right\rangle}
\newcommand{\lie}[2]{[#1,#2]}
\newcommand{\frakg}{\mathfrak{g}}
\newcommand{\frakh}{\mathfrak{h}}
\newcommand{\fraka}{\mathfrak{a}}
\newcommand{\DO}{\mathcal{D}}
\newcommand{\source}{s}
\newcommand{\target}{t}
\newcommand{\livf}[1]{\overset{\shortrightarrow}{#1}}
\newcommand{\rivf}[1]{\overset{\shortleftarrow}{#1}}
\newcommand{\anchor}{\rho}
\newcommand{\cR}{\mathcal{R}}
\newcommand{\Qd}{D}
\newcommand{\torsion}{T^\nabla}
\newcommand{\curvature}{R^\nabla}
\newcommand{\XXfv}{\XX^{\operatorname{fv}}}
\newcommand{\cocycle}{\mathcal{Z}}
\newcommand{\scindement}{\kappa_j}
\newcommand{\nablanatural}{\nabla}

\title[PBW isomorphisms and Kapranov dg-manifolds]
{Poincaré--Birkhoff--Witt isomorphisms \\ and Kapranov dg-manifolds}

\thanks{Research partially supported by NSF grants DMS-2001599, DMS-1707545, DMS-1406668, DMS-1101827, and NSA grant H98230-14-1-0153.}

\author{Camille Laurent-Gengoux}
\address{Institut Élie Cartan de Lorraine, UMR 7502, Université de Lorraine}
\email{camille.laurent-gengoux@univ-lorraine.fr}
\author{Mathieu Stiénon}
\address{Department of Mathematics, Pennsylvania State University}
\email{stienon@psu.edu}
\author{Ping Xu}
\address{Department of Mathematics, Pennsylvania State University}
\email{ping@math.psu.edu}

\begin{document}

\begin{abstract}
We prove that to every inclusion $A\into L$ of Lie algebroids over the same base manifold $M$
corresponds a Kapranov dg-manifold structure on $A[1]\oplus L/A$,
which is canonical up to isomorphism.
As a consequence, $\sections{\Lambda^\bullet A^\vee\otimes L/A}$
carries a canonical $L_\infty[1]$ algebra structure
whose unary bracket is the Chevalley--Eilenberg differential
$d_A^{\nabla^{\Bott}}$ corresponding to the Bott representation of $A$ on $L/A$
and whose binary bracket is a cocycle representative of the Atiyah class of the Lie pair $(L,A)$.
To this end, we construct explicit isomorphisms of $C^\infty(M)$-coalgebras
$\sections{S(L/A)}\xto{\sim}\frac{\mathcal{U}(L)}{\mathcal{U}(L)\sections{A}}$,
which we elect to call Poincaré--Birkhoff--Witt maps.
These maps admit a recursive characterization that allows for explicit computations.	
They generalize both the classical symmetrization map
$S(\mathfrak{g})\to\enveloping{\mathfrak{g}}$ of Lie theory
and (the inverse of) the complete symbol map for differential operators.
Finally, we prove that the Kapranov dg-manifold $A[1]\oplus L/A$ is linearizable
if and only if the Atiyah class of the Lie pair $(L,A)$ vanishes.
\end{abstract}

\maketitle

\tableofcontents


\section{Introduction}
\label{Biloxi}

In his work~\cite{MR1671737} on Rozansky--Witten invariants,
Kapranov discovered a natural $L_\infty[1]$ algebra structure
on the Dolbeault complex $\OO^{0,\bullet}(T^{1,0}_X)$ of an arbitrary Kähler manifold $X$.
This $L_\infty[1]$ algebra plays an important role
in derived geometry~\cite{MR2657369,MR2472137,MR2431634}.
Its unary bracket $\lambda_1$ is the Dolbeault operator $\overline{\partial}$,
its binary bracket $\lambda_2$ is the composition
\[ \OO^{0,k}(T^{1,0}_X)\otimes\OO^{0,l}(T^{1,0}_X)\xto{\wedge}\OO^{0,k+l}(T^{1,0}_X\otimes T^{1,0}_X)
\xto{\mathcal{R}_2}\OO^{0,k+l+1}(T^{1,0}_X) \]
of the wedge product and the contraction with the Atiyah cocycle
\[ \mathcal{R}_2\in\OO^{0,1}\big(\Hom(T^{1,0}_X\otimes T^{1,0}_X,T^{1,0}_X)\big) ,\]
and its higher multi-brackets are all obtained similarly as
compositions of the wedge product and some explicit contractions.

The Atiyah class \cite{MR86359} of the holomorphic tangent bundle $T_X$ of a complex manifold $X$
is the cohomology class of $\mathcal{R}_2$ in 
$H^{1,1}(X;\End(T_X))\cong
H^1\big(X;\OO_X^1\otimes\End(T_X)\big)$.
It is the obstruction to the existence of a holomorphic connection on $T_X$.
Together with Chen, two of the authors defined a similar Atiyah class
in the much wider context of \emph{pairs of Lie algebroids} in~\cite{MR3439229}.
We say that $(L,A)$ is a pair of Lie algebroids or \emph{Lie pair} if
$A$ is a subalgebroid of a Lie algebroid $L$ over the same base manifold.
The Atiyah class $\alpha_{L/A}$ of the Lie pair
captures the obstruction to the existence of an \emph{`$A$-compatible'} $L$-connection on $L/A$.
It generalizes both the classical Atiyah class of holomorphic tangent bundles~\cite{MR86359}
and the Molino class of foliations~\cite{MR0281224,MR0353332}.
It is thus natural to ask whether Kapranov's construction of an
$L_\infty[1]$ algebra can be extended to all Lie pairs.
In this paper, we give an affirmative answer to this question.
As a byproduct, we show that 
Kapranov's explicit construction
for Kähler manifolds extends to 
all complex manifolds.

It is well known that a finite-dimensional $L_\infty[1]$ algebra
is equivalent to a finite-dimensional $\ZZ$-graded vector space endowed with
an additional structure of pointed dg-manifold~\cite{MR2062626,MR1183483},
i.e.\ a dg-vector space whose homological vector field vanishes at the origin.

Kapranov dg-manifolds --- the core notion of the present paper ---
are analogues of pointed dg-manifolds in the context
of $\ZZ$-graded vector bundles.
More precisely, a \emph{Kapranov dg-manifold} consists of a pair of smooth vector bundles $A$ and $E$
over a common base manifold together with a pair of homological vector fields on the
$\ZZ$-graded manifolds $A[1]\oplus E$ and $A[1]$ such that both the inclusion
$A[1]\into A[1]\oplus E$ and the projection $A[1]\oplus E\onto E$ are morphisms of dg-manifolds.
Given a Kapranov dg-manifold, the vector bundle $A$ is necessarily a Lie algebroid,
the vector bundle $E$ is necessarily an $A$-module for some representation $\nablanatural$,
and the graded vector space $\bigoplus_{n\geqslant 0}\sections{\Lambda^n A^\vee\otimes E}$ necessarily
inherits a structure of $L_\infty[1]$ algebra encoded by a sequence $(\lambda_k)_{k\in\NN}$
of multi-brackets
$\lambda_k:S^k\big(\sections{\Lambda A^\vee\otimes E}\big)\to\sections{\Lambda A^\vee\otimes E}[1]$
such that each $\lambda_k$ with $k\geqslant 2$ is $\sections{\Lambda A^\vee}$-multilinear
while $\lambda_1$, which is not $\sections{\Lambda A^\vee}$-linear,
is the Chevalley--Eilenberg differential $d_A^{\nablanatural}$
associated with the representation $\nablanatural$ of the Lie algebroid $A$ on the vector bundle $E$.

We prove:
\begin{introthm}\label{Amsterdam}
Given a Lie pair $(L,A)$, one can endow $A[1]\oplus L/A$ with a structure of Kapranov dg-manifold.
Its homological vector field is the sum $\Qd=d_A^{\nabla^{\Bott}}+\sum_{k=2}^{\infty}\cR_k$
of the Chevalley--Eilenberg differential $d_A^{\nabla^{\Bott}}$
induced by the Bott representation of $A$ on $L/A$,
a 1-cocycle $\cR_2\in\sections{A^\vee\otimes S^2(L/A)^\vee\otimes L/A}$
representative of the Atiyah class of the Lie pair $(L,A)$,
and higher order terms $\cR_k\in\sections{A^\vee\otimes S^k(L/A)^\vee\otimes L/A}$
that can be computed iteratively.
Though the construction of the homological vector field $\Qd$ depends on some choice,
different choices yield canonically isomorphic Kapranov dg-manifold structures.
\end{introthm}

Conceptually, one can distinguish two key steps in the proof of Theorem~\ref{Amsterdam}.
First, working within the realm of Lie algebroids over $\RR$
where an analogue of Lie's third theorem \cite{MR0231414} holds (at least locally),
one constructs the Kapranov dg-manifold associated with a real Lie pair in a purely geometric way.
Next, one translates the geometric construction into purely algebraic formulae
that make sense for Lie algebroids both real and complex,
which shows that the restriction to real Lie algebroids is unnecessary.
Thus, Theorem~\ref{Amsterdam} holds for complex Lie algebroids as well.

Given a real Lie pair $(L,A)$ over a base manifold $M$, an analogue of
Lie's third fundamental theorem for Lie groupoids \cite{MR0231414}
asserts the existence of a pair $(\groupoid{L},\groupoid{A})$ of local
Lie groupoids \cite{MR2157566},
which the Lie functor transforms into $L$ and $A$ respectively.
Here $\groupoid{A}$ is a closed local Lie subgroupoid of $\groupoid{L}$.
Consider the differential graded commutative algebra
$\big(\sections{\Lambda^{\bullet} A^\vee}\otimes_R
C^\infty_\text{formal}(\groupoid{L}/\groupoid{A}), d_{A} \big )$,
where $C^\infty_\text{formal}(\groupoid{L}/\groupoid{A})$
denotes the algebra of functions on the formal
neighborhood of the unit section in the fiber bundle
$s:\groupoid{L}/\groupoid{A}\to M$
--- more precisely, the algebra of $s$-fiberwise $\infty$-jets along the unit space $M$
of functions on $\groupoid{L}/\groupoid{A}$ ---
and $d_{A}$ denotes the Chevalley--Eilenberg differential
for the $A$-module structure on $C^\infty_\text{formal}(\groupoid{L}/\groupoid{A})$
stemming from the natural left $A$-action on $\groupoid{L}/\groupoid{A}$.
The map $s:\groupoid{L}/\groupoid{A}\to M$ denotes
the factorization of the source map $s: \groupoid{L}\to M$
through the projection $\groupoid{L}\onto\groupoid{L}/\groupoid{A}$.
Since $L/A$ may be regarded as the normal bundle of $M$ in $\groupoid{L}/\groupoid{A}$,
the algebra $C^\infty_\text{formal}(\groupoid{L}/\groupoid{A})$
can be identified with $\sections{\hat{S}(L/A)^\vee}$,
the algebra of functions on the formal neighborhood of the zero section of
the vector bundle $\pi:L/A\to M$ --- more precisely,
the algebra of $\pi$-fiberwise $\infty$-jets along the zero section of functions on $L/A$.
Substituting $\sections{\hat{S}(L/A)^\vee}$ for $C^\infty_\text{formal}(\groupoid{L}/\groupoid{A})$,
we obtain the differential graded commutative algebra
$\big(\sections{\Lambda^\bullet A^\vee\otimes\hat{S}(L/A)^\vee},\Qd\big)$.
Its differential $\Qd$ is the homological vector field
on $A[1]\oplus L/A$ appearing in Theorem~\ref{Amsterdam}.

In order to obtain an explicit formula for the homological vector field $\Qd$,
an explicit isomorphism between the fiber bundles $\pi: L/A\to M$
and $s: \groupoid{L}/\groupoid{A} \to M$ is needed.
For this purpose, we construct exponential maps
from $L/A$ to $\groupoid{L}/\groupoid{A}$ in Section~\ref{Hungary}.
Indeed, every choice of (i) a splitting $j:L/A\to L$ of the short exact sequence of vector bundles
$0\to A\to L\to L/A\to 0$ and (ii) an $L$-connection $\nabla$ on $L/A$ determines an exponential map
$\exp^{\nabla,j}:L/A\to\groupoid{L}/\groupoid{A}$.

When the splitting $j$ identifies $L/A$ with a Lie subalgebroid of $L$ --- say $B$ ---
the homogeneous space $\groupoid{L}/\groupoid{A}$ gets identified locally
with the corresponding Lie subgroupoid of $\groupoid{L}$
--- say $\groupoid{B}$. In that case, $A$ and $B$ are said to form
a matched pair of Lie algebroids \cite{MR1460632}.
With these identifications, each $L$-connection $\nabla$ on $L/A$
determines a $B$-connection $\ddot{\nabla}$ on $B$,
and the exponential map $\exp^{\nabla,j}:L/A\to\groupoid{L}/\groupoid{A}$,
arising from $j$ and $\nabla$,
is precisely the exponential map
$\exp^{\ddot{\nabla}}:B\to\groupoid{B}$ associated with the connection $\ddot{\nabla}$
investigated in~\cite{MR1722129,MR1687747}.
In particular, when the base manifold $M$ is the one-point space $\{*\}$,
the Lie algebroid $B$ is simply a Lie algebra $\mathfrak{g}$,
the Lie groupoid $\groupoid{B}$ is the corresponding Lie group $G$,
and the $B$-connection on $B$ is simply a linear map $\mathfrak{g}\otimes\mathfrak{g}\to\mathfrak{g}$.
The Lie-theoretic exponential map $\exp:\mathfrak{g}\to G$ corresponds
to the choice of the trivial linear map
$\mathfrak{g}\otimes\mathfrak{g}\to\mathfrak{g}$.

Although exponential maps $\exp^{\nabla,j}:L/A\to\groupoid{L}/\groupoid{A}$
can be defined geometrically on a small neighborhood of the zero section of $L/A$,
what we really need in the present paper are \emph{formal exponential maps}
identifying the formal neighborhood of the zero section of $L/A$
to the formal neighborhood of $M$ in $\groupoid{L}/\groupoid{A}$.
In the real case, these formal exponential maps
arise as the $\infty$-jet of $\exp^{\nabla, j}:L/A\to\groupoid{L}/\groupoid{A}$
in the direction of the fibers and along the zero section.
We prove that they admit a purely algebraic definition,
which still makes sense for pairs of complex Lie algebroids
despite the loss of the groupoid picture (see~\cite{MR2285039}) --- this is the second step 
in the proof of Theorem~\ref{Amsterdam}.

At this point, it is useful to recall the relation between the Lie-theoretic exponential map
$\exp:\mathfrak{g}\to G$ and its $\infty$-jet,
the symmetrization map $S(\mathfrak{g})\to\enveloping{\mathfrak{g}}$.
The coalgebra $D'_0(\frakg)$ of distributions on $\frakg$ with support $\{0\}$
is canonically isomorphic, as a coalgebra, to $S(\frakg)$,
while the coalgebra $D'_e(G)$ of distributions on $G$ with support $\{e\}$
is canonically isomorphic, as a coalgebra, to $\enveloping{\frakg}$.
Since $\exp:\frakg\to G$ is a local diffeomorphism near $0$,
it can be used to push distributions forward from the Lie algebra to the Lie group.
The induced isomorphism of coalgebras $S(\frakg)\xto{\simeq}\enveloping{\frakg}$
is precisely the symmetrization map
\begin{equation}\label{eq:sym}
X_1\odot\cdots\odot X_n\mapsto\frac{1}{n!}\sum_{\sigma\in S_n}X_{\sigma (1)}\cdots X_{\sigma (n)}
.\end{equation}
The above construction can be extended to the exponential map
$\exp^{\nabla, j}:L/A\to\groupoid{L}/\groupoid{A}$.
The space of fiberwise distributions on the vector bundle
$\pi: L/A\to M$ with support the zero section
can be identified, as a $C^\infty(M)$-coalgebra, to $\sections{S(L/A)}$,
while the space of fiberwise distributions on the fiber bundle $s:\groupoid{L}/\groupoid{A}\to M$
with support the unit space $M$ can be identified, as a $C^\infty(M)$-coalgebra, to
$\tfrac{\mathcal{U}(L)}{\mathcal{U}(L)\sections{A}}$.
Pushing distributions forward through the exponential map
$\exp^{\nabla, j}:L/A\to\groupoid{L}/\groupoid{A}$,
which identifies tubular neighborhoods of $M$ in $L/A$ and $\groupoid{L}/\groupoid{A}$,
we obtain an isomorphism of $C^\infty(M)$-coalgebras
$\sections{S(L/A)}\to \tfrac{\mathcal{U}(L)}{\mathcal{U}(L)\sections{A}}$,
which we elect to call \emph{Poincaré--Birkhoff--Witt map}.
It generalizes the symmetrization map \eqref{eq:sym} in
Lie theory to the context of Lie pairs.

Indeed we prove the following
\begin{introthm}\label{Copenhagen}
Let $(L,A)$ be a Lie pair, either real or complex, over a base manifold $M$.
Each choice of (i) a splitting $j: L/A\to L$ of the short exact sequence
of vector bundles $0\to A\to L\to L/A\to 0$
and (ii) an $L$-connection $\nabla$ on $L/A$ determines an isomorphism of $C^\infty(M)$-coalgebras
$\pbw^{\nabla,j}:\sections{S(L/A)}\to\tfrac{\mathcal{U}(L)}{\mathcal{U}(L)\sections{A}}$.
This map $\pbw^{\nabla,j}$ is characterized purely algebraically
by a recursive relation -- see Theorem~\ref{Nairobi}.
\end{introthm}

Theorem~\ref{Amsterdam} is proved using the recursive relation
defining the map $\pbw^{\nabla,j}$ appearing in Theorem~\ref{Copenhagen}.
This is discussed in Sections~\ref{Paris} and~\ref{Olten}.

For a Lie algebroid $A$ and an $A$-module $E$,
the Chevalley--Eilenberg differentials on $\sections{\Lambda^\bullet A^\vee}$
and $\sections{\Lambda^\bullet A^\vee\otimes E}$
can be reinterpreted as a pair of homological vector fields
on the pair of graded manifolds $A[1]$ and $A[1]\oplus E$
making the latter into a Kapranov dg-manifold,
which we call the \emph{linear} Kapranov dg-manifold associated with the $A$-module $E$.
A Kapranov dg-manifold $A[1]\oplus E$ is said to be \emph{linearizable}
if it is isomorphic to the linear Kapranov dg-manifold
associated with some $A$-module structure on $E$.

One might wonder whether the Kapranov dg-manifold $A[1]\oplus L/A$
arising from a Lie pair $(L,A)$ as in Theorem~\ref{Amsterdam} is linearizable or not.
Our next result gives a complete answer to this natural question.

\pagebreak[2]
\begin{introthm}\label{Dublin}
\strut
\begin{enumerate}
\item For a Lie pair $(L,A)$, the Kapranov dg-manifold structure on
$A[1]\oplus L/A$ of Theorem~\ref{Amsterdam} is linearizable
if and only if the Atiyah class $\alpha_{L/A}$ vanishes.
\item
In the case of Lie pairs over $\RR$,
the Atiyah class $\alpha_{L/A}$ vanishes if and only if
there exists a fiber-preserving diffeomorphism
$\phi: L/A\to \groupoid{L}/\groupoid{A}$
from a tubular neighborhood of the zero section of $L/A$
to a tubular neighborhood of the unit section of
$\groupoid{L}/\groupoid{A}$,
which fixes the submanifold $M$,
and whose fiberwise $\infty$-jet along $M$ intertwines
the fiberwise $\infty$-jets along $M$ of the $A$-actions on $L/A$
and on $\groupoid{L}/\groupoid{A}$, respectively.
\end{enumerate}
\end{introthm}

As an application, we consider the Lie pair
$(L=T_X\otimes \CC,A=T^{0, 1}_X)$
arising from a complex manifold $X$.
Theorem~\ref{Amsterdam} yields a Kapranov dg-manifold structure
on $T^{0,1}_X[1]\oplus T^{1,0}_X$,
whose homological vector field involves the Atiyah cocycle as its second-order term.
When $X$ is a Kähler manifold, we recover the classical result of Kapranov~\cite{MR1671737}.
Furthermore, it follows from Theorem~\ref{Dublin}
that the Kapranov dg-manifold $T^{0,1}_X[1]\oplus T^{1, 0}_X$
is linearizable if and only if the Atiyah class of $X$ vanishes.

A few remarks are in order regarding the relations
between the present paper and the existing literature.
Our Kapranov dg-manifolds should not be confused with
those dg-manifolds studied by Ciocan-Fontanine 
and Kapranov~\cite{MR1801413,MR1839580,MR2496057} in derived geometry,
where (unlike here) algebras are required to be negatively graded.
A Kapranov dg-manifold $A[1]\oplus E$, when considered as a vector bundle over $A[1]$,
is \emph{not} a $Q$-bundle in the sense of Kotov--Strobl~\cite{MR3293862},
for the dg-structure is not compatible with the linear structure.
After~\cite{MR2989383} appeared, we learned from Bressler~\cite{Bressler_private_communication}
that he had obtained results similar to some of ours.
Recently, Blom--Posthuma studied the dual of the $\pbw$ map
and found beautiful applications to index theory \cite{arXiv:1512.07863}.
Vitagliano studied various homotopy algebra structures associated with foliations
(a special case of Lie algebroid pairs) \cite{MR3277952,MR3313214,MR3300319}
using PBW maps.
We would also like to point out that the present paper is related to Calaque's work~\cite{MR3217749}.
However, we warn the reader that `PBW' does not have the same meaning here as in~\cite{MR3217749}.
In the present paper, `PBW isomorphism' means an explicit $R$-linear isomorphism
whose construction depends only on the choice of a connection (and a splitting), and always exists
just as is the case for (the inverse of) the complete symbol map \cite{MR1687747}.
On the other hand, `PBW theorem' in~\cite{MR3217749} means an isomorphism of filtered ${A}$-modules,
which exists only when the Atiyah class vanishes.
The `PBW theorem' in~\cite{MR3217749} is equivalent
to linearizability of $A[1]\oplus L/A$ in the present paper
--- compare~\cite[Theorem~5.4]{MR3217749} and Theorem~\ref{Dublin}.

We conclude with a brief summary of a number of further developments posterior to
the posting of the first e-print version of the present paper on arXiv.org.
In~\cite{MR4059952}, Laurent-Gengoux and Voglaire introduced
the notion of Atiyah class of a Lie groupoid pair,
which is related to the Atiyah class of the corresponding Lie algebroid pair by the van~Est functor.
As an application, they sharpened the second part of Theorem~\ref{Dublin}
by showing that, when the Atiyah class vanishes, linearization holds
locally in a neighborhood of the zero section rather than merely `formally' along the zero section.
Their result for Lie pairs over $\RR$ extends to Lie groupoid pairs
a theorem of Bordemann pertaining to Lie group pairs,
namely a consequence of the existence of a connection as stated in~\cite[Proposition~2.5]{MR3014184}.
In~\cite{MR4150934}, Stiénon and Xu constructed a homological vector field
on the graded manifold $L[1]\oplus L/A$ arising from a Lie pair $(L,A)$
using the PBW isomorphism for Lie pairs
constructed in the present paper
(Theorem~\ref{Copenhagen}).
The resulting dg manifold, which Stiénon and Xu elected to call Fedosov dg manifold,
played a fundamental role in establishing a `Gelfand--Kazhdan formal geometry' for Lie pairs
in the sense of Fedosov and, subsequently, in establishing a formality theorem
and a Kontsevich-Duflo type theorem for Lie pairs \cite{arXiv:1901.04602,MR3964152,MR3650387}.
The Kapranov dg manifold $A[1]\oplus L/A$ of Theorem~\ref{Amsterdam}
is a dg submanifold of the Fedosov dg manifold $L[1]\oplus L/A$
constructed in~\cite{MR4150934}.
In~\cite{arXiv:2106.00812}, the $L_\infty[1]$ 
algebra arising from the 
Lie pair encoding an integrable distribution 
$\mathcal{F}$ on a manifold $M$ was compared 
with the Kapranov $L_\infty[1]$ algebra 
associated with the dg manifold 
$(\mathcal{F}[1],d_\mathcal{F})$ .
Finally, Liao and Stiénon \cite{MR3910470} adapted
the construction of the PBW isomorphism for Lie pairs (Theorem~\ref{Copenhagen})
to the context of $\ZZ$-graded manifolds
and then used it to construct `Fedosov resolutions' of $\ZZ$-graded manifolds.
The latter played a fundamental role in proving a formality theorem and
a Kontsevich-Duflo type theorem for dg-manifolds \cite{MR3754617,PolishSurvey}.

\subsection*{Terminology and notations}

We use the symbol $\KK$ to denote either of the fields $\RR$ and $\CC$.
Given a Lie $\KK$-algebroid or a Lie pair,
we use the symbol $R$ to denote the algebra of smooth $\KK$-valued
functions on its base manifold.
Unless specified otherwise, Lie groupoid means \emph{local} Lie groupoid.

\subsubsection*{Shuffles}
A $(p,q)$-shuffle is a permutation $\sigma$ of the set $\{1,2,\cdots,p+q\}$
such that \[ \sigma(1)<\sigma(2)<\cdots<\sigma(p)
\qquad\text{and}\qquad \sigma(p+1)<\sigma(p+2)<\cdots<\sigma(p+q) .\]
The symbol $\shuffle{p}{q}$ denotes the set of all $(p,q)$-shuffles.

\subsubsection*{Grading shift}
Given a graded vector space $V=\bigoplus_{k\in\ZZ}V^{(k)}$,
the notation $V[i]$ denotes the graded
vector space obtained by shifting the grading on $V$ according to the rule
$(V[i])^{(k)}=V^{(i+k)}$. Accordingly, if $E=\bigoplus_{k\in\ZZ}E^{(k)}$
is a graded vector bundle over $M$, then $E[i]$ denotes
the graded vector bundle obtained by shifting the degree in the fibers
of $E$ according to the above rule.

\subsubsection*{Koszul sign}
The Koszul sign $\Koszul{\sigma; v_1, \cdots, v_{n}}$
of a permutation $\sigma$ of homogeneous vectors $v_1,v_2,\dots,v_n$
of a $\ZZ$-graded vector space $V=\bigoplus_{n\in\ZZ}V_n$ is determined by
the equality \[ v_{\sigma(1)}\odot v_{\sigma(2)}\odot\cdots\odot v_{\sigma(n)}
= \Koszul{\sigma; v_1, \cdots, v_n} \ v_1\odot v_2\odot\cdots\odot v_n \]
in the graded commutative algebra $S(V)$.

\subsubsection*{$L_\infty$ algebra}
An $L_\infty[1]$ algebra~\cite{MR1235010,MR1183483,MR2440258}
is a $\ZZ$-graded vector space $V=\bigoplus_{n\in\ZZ}V_n$ endowed with
a sequence $(\lambda_k)_{k=1}^\infty$ of linear maps
$\lambda_k: S^k(V)\to V[1]$ satisfying the generalized Jacobi identities
\[ \sum_{p+q=n}\sum_{\sigma\in\shuffle{p}{q}}
\Koszul{\sigma;v_1,\cdots,v_n}\
\lambda_{1+q}\big(\lambda_p(v_{\sigma(1)},\cdots,v_{\sigma(p)}),
v_{\sigma(p+1)},\cdots,v_{n}\big)=0 \]
for each $n\in\NN$ and for all homogeneous vectors $v_1,v_2,\dots,v_n\in V$.
\newline
A $\ZZ$-graded vector space $V$ is an $L_\infty$ algebra
if and only if $V[1]$ is an $L_\infty[1]$ algebra.

\subsubsection*{Lie--Rinehart algebra}
A Lie--Rinehart algebra $(R,\mathfrak{L})$
consists of
(1) a commutative unital $\KK$-algebra $R$,
(2) a Lie algebra $\mathfrak{L}$ over $\KK$
endowed with an additional structure of left $R$-module,
and (3) an action $\rho:\mathfrak{L}\to\Der(R)$
of $\mathfrak{L}$ from the left on $R$ by derivations.
These various structures must satisfy the compatibility conditions
\[ \lie{X}{fY}=\rho_X(f) Y+f\lie{X}{Y}
\qquad\text{and}\qquad
\rho_{fX}(g)=f\rho_X(g) ,\]
for all $f,g\in R$ and $X,Y\in\mathfrak{L}$.

\subsubsection*{Multiplication in a groupoid}
We adopt the following convention for the multiplication
in a groupoid $\mathscr{L}\rightrightarrows M$
with source map $\source:\groupoid{L}\to M$ and target map $\target:\groupoid{L}\to M$:
given two elements $g$ and $h$ of $\mathscr{L}$,
their product $gh$ is defined only if the target of $g$ coincides with the source of $h$,
i.e.\ if $t(g)=s(h)$ in $M$.
With this convention, we have $s(gh)=s(g)$ and $t(gh)=t(h)$.
Hence, left translation by $g$ maps the source-fiber $\source^{-1}\big(\target(g)\big)$ to
the source-fiber $\source^{-1}\big(\source(g)\big)$.
Consequently, the left invariant vector fields on $\groupoid{L}$
are necessarily tangent to the fibers of the source map.

\subsection*{Acknowledgments}
We would like to thank several institutions for their hospitality,
which allowed the completion of this project:
Pennsylvania State University (Laurent-Gengoux), Université de Lorraine (Stiénon),
and Université Paris Diderot (Xu).
We also wish to thank Martin Bordemann, Paul Bressler, Hessel Posthuma,
Jim Stasheff, and Yannick Voglaire for useful discussions and comments.
We are grateful to the anonymous referee for many insightful comments and suggestions
which led to significant improvements in exposition.

\section{Poincaré--Birkhoff--Witt isomorphisms}

\subsection{Representations of Lie algebroids}
Let $M$ be a smooth manifold, let $A\to M$ be a Lie $\KK$-algebroid
with anchor map $\anchor:A\to T_M\otimes_{\RR}\KK$,
and let $E\to M$ be a vector bundle over $\KK$.
The algebra of smooth functions on $M$ with values in $\KK$ will be denoted $R$.
A \emph{representation of the Lie algebroid} $A$ on the vector bundle $E\to M$ is
a flat (Lie algebroid) $A$-connection $\nabla$ on $E$.
A vector bundle endowed with a representation of the Lie algebroid $A$ is called an \emph{$A$-module}.
By extension, any left $R$-module endowed
with a representation of the Lie--Rinehart algebra (see~\cite{MR154906})
$\big(R,\sections{A}\big)$ will also be called an $A$-module.

Let $(L,A)$ be a Lie pair, i.e.\ an inclusion $A\into L$ of Lie algebroids
over the same base manifold.
The \emph{Bott representation} of $A$ on the quotient $L/A$ is the flat $A$-connection
on $L/A$ defined by
\begin{equation}\label{Istambul}
\nabla^{\Bott}_a \big(q(l)\big)=q\big(\lie{a}{l}\big),
\quad\forall a\in\sections{A},l\in\sections{L} ,\end{equation}
where $q$ denotes the canonical projection $L\onto L/A$.
Therefore, the quotient $L/A$ is an $A$-module.

\subsection{Universal enveloping algebra of a Lie algebroid}

Let $L$ be a Lie $\KK$-algebroid over a smooth manifold $M$
and let $R$ denote the algebra of smooth functions on $M$ taking values in $\KK$.
By $\enveloping{L}$, we denote its universal enveloping algebra \cite{MR2653938,MR1687747}.
Essentially, $\enveloping{L}$ is the quotient of the (reduced) tensor algebra
$\bigoplus_{n=1}^{\infty}\big(R\oplus\sections{L}\big)^{\otimes n}$
of the $\KK$-module $R\oplus\sections{L}$ by the two-sided ideal
generated by the elements of the following four types:
\begin{align*}
& X\otimes Y-Y\otimes X-\lie{X}{Y} && f\otimes X-fX \\
& X\otimes g-g\otimes X-\anchor(X)(g) && f\otimes g-fg
\end{align*}
for all $X,Y\in\sections{L}$ and $f,g\in C^\infty(M)$.

There is a natural filtration \cite{MR154906}:
\begin{equation}
\label{Jakarta}
R\into \mathcal{U}^{\leqslant 1}(L) \into \mathcal{U}^{\leqslant 2}(L)
\into \mathcal{U}^{\leqslant 3}(L) \into \cdots
\end{equation}

When the base $M$ is the one-point space $\{*\}$ so that $L$ is a Lie algebra,
$\enveloping{L}$ is the usual enveloping algebra of the Lie algebra ${L}$.
When the Lie algebroid $L$ is the tangent bundle $T_M$,
its universal enveloping algebra $\enveloping{L}$ is the algebra of differential operators on $M$.
In general, the universal enveloping algebra $\enveloping{L}$ of the Lie algebroid $L$
associated with a Lie groupoid $\groupoid{L}$ is canonically identified with the associative algebra
of source-fiberwise differential operators on $\cinf{\groupoid{L}}$
invariant under left translations\footnote{See `Terminology and Notations'
at the end of the Introduction for our groupoid multiplication
convention.}~\cite{MR2653938,MR1687747}.

The universal enveloping algebra $\enveloping{L}$
of the Lie algebroid $L\to M$ is an $R$-coalgebra~\cite{MR1815717}.
Its comultiplication
\[ \Delta:\enveloping{L}\to\enveloping{L}\otimes_R\enveloping{L} \]
is compatible with its filtration~\eqref{Jakarta} and characterized by the identities
\begin{gather*}
\Delta(1)=1\otimes 1; \\
\Delta(x)=1\otimes x+x\otimes 1, \quad \forall x\in \sections{L}; \\
\Delta(u\cdot v)=\Delta(u)\cdot\Delta(v), \quad \forall u,v\in\enveloping{L} ,
\end{gather*}
where $1\in R$ denotes the constant function on $M$ with value $1$
while the symbol~$\cdot$ denotes the multiplications in $\enveloping{L}$ 
and $\enveloping{L}\otimes_R\enveloping{L}$.
We refer the reader to~\cite[Equation~(15) and the remark following Definition~3.1]{MR1815717}
for the precise meaning of the last equation above.
Explicitly, we have
\begin{multline*}
\Delta(b_1\cdot b_2\cdot\cdots\cdot b_n)=
1\otimes(b_1\cdot b_2\cdot\cdots\cdot b_n)
+ \sum_{\substack{p+q=n \\ p,q\in\NN}}\sum_{\sigma\in\shuffle{p}{q}}
(b_{\sigma(1)}\cdot\cdots\cdot b_{\sigma(p)}) \otimes
(b_{\sigma(p+1)}\cdot\cdots\cdot b_{\sigma(n)}) \\
+ (b_1\cdot b_2\cdot\cdots\cdot b_n)\otimes 1
,\end{multline*}
for all $b_1,\dots,b_n\in\sections{L}$.

\subsection{The coalgebras \texorpdfstring{$\sections{S(L/A)}$}{Γ(S(L/A))}
and \texorpdfstring{$\frac{\enveloping{L}}{\enveloping{L}\sections{A}}$}{U(L)/U(L)Γ(A)}}

Let $(L,A)$ be a Lie pair, i.e.\ an inclusion $A\into L$
of Lie $\KK$-algebroids over the same base manifold.

Writing $\enveloping{L}\sections{A}$ for the left ideal of $\enveloping{L}$
generated by $\sections{A}$, the quotient
$\frac{\enveloping{L}}{\enveloping{L}\sections{A}}$
is automatically a filtered $R$-coalgebra since
\[ \Delta\big(\enveloping{L}\sections{A}\big)\subseteq \enveloping{L}\otimes_R
\big(\enveloping{L}\sections{A}\big) + \big(\enveloping{L}\sections{A}\big)
\otimes_R \enveloping{L} \]
and the filtration~\eqref{Jakarta} on $\enveloping{L}$ descends
to a filtration of $\frac{\enveloping{L}}{\enveloping{L}\sections{A}}$:
\begin{equation} \label{Kolkata}
R\into \left(\frac{\enveloping{L}}{\enveloping{L}\sections{A}}\right)^{\leqslant 1}
\into \left(\frac{\enveloping{L}}{\enveloping{L}\sections{A}}\right)^{\leqslant 2}
\into \left(\frac{\enveloping{L}}{\enveloping{L}\sections{A}}\right)^{\leqslant 3} \into \cdots
\end{equation}

Likewise, deconcatenation defines a graded $R$-coalgebra structure on 
\[ \sections{S(L/A)}=\bigoplus_{n\geqslant 0}\sections{S^n(L/A)} .\]
The comultiplication
\[ \Delta:\sections{S(L/A)}\to\sections{S(L/A)}\otimes_R\sections{S(L/A)} \] is given by
\begin{multline}\label{London}
\Delta(b_1\odot b_2\odot\cdots\odot b_n) \\
= 1\otimes(b_1\odot b_2\odot\cdots\odot b_n)
+ \sum_{\substack{p+q=n \\ p,q\in\NN}}\sum_{\sigma\in\shuffle{p}{q}}
(b_{\sigma(1)}\odot\cdots\odot b_{\sigma(p)})
\otimes (b_{\sigma(p+1)}\odot\cdots\odot b_{\sigma(n)}) \\
+ (b_1\odot b_2\odot\cdots\odot b_n)\otimes 1
,\end{multline}
for all $b_1,\dots,b_n\in\sections{L/A}$.
Here, the symbol~$\odot$ denotes the symmetric product in $\sections{S(L/A)}$.

Note that there is an obvious filtration on $\sections{S(L/A)}$:
\begin{equation} \label{Pune}
R\into \sections{S^{\leqslant 1}(L/A)}
\into \sections{S^{\leqslant 2}(L/A)}
\into \sections{S^{\leqslant 3}(L/A)} \into \cdots
\end{equation}

\subsection{Poincaré--Birkhoff--Witt isomorphism}
\label{Madrid}

For every Lie algebra $\mathfrak{g}$,
the classical Poincaré--Birkhoff--Witt theorem asserts that,
for every $n\in\NN$, the vector space
$\Gr^n\big(\enveloping{\mathfrak{g}}\big):=
\frac{\mathcal{U}^{\leqslant n}(\mathfrak{g})}{\mathcal{U}^{\leqslant n-1}(\mathfrak{g})}$
is canonically isomorphic to $S^n(\mathfrak{g})$.
In fact, we have a stronger result:
the symmetrization map \eqref{eq:sym}
is an isomorphism of filtered coalgebras
from $S(\mathfrak{g})$ to $\enveloping{\mathfrak{g}}$,
which lifts the canonical isomorphism of
graded vector spaces $S(\mathfrak{g})\to
\Gr\big(\enveloping{\mathfrak{g}}\big)$.

Given a Lie--Rinehart algebra $(R,\mathfrak{L})$,
in general, there does not exist a symmetrization map from $S_R(\mathfrak{L})$,
the symmetric algebra of $\mathfrak{L}$ (over $R$), to $\enveloping{R,\mathfrak{L}}$,
the universal enveloping algebra of the Lie--Rinehart algebra $(R,\mathfrak{L})$.
However, there is a well-defined symmetrization map
\begin{equation}\label{Rinehart_sym}
S_R(\mathfrak{L})\to\Gr\big(\enveloping{R,\mathfrak{L}}\big)
\end{equation}
on the level of the associated graded vector spaces.
Rinehart proved that, provided $\mathfrak{L}$ is $R$-projective,
the symmetrization map \eqref{Rinehart_sym}
is an isomorphism of graded $R$-algebras \cite[Theorem~3.1]{MR154906}.

The next theorem provides a replacement for the missing symmetrization map
between filtered $R$-coalgebras in the context of Lie algebroids
--- more precisely, in the even more general context of Lie pairs.

Let $(L,A)$ be a Lie pair.
We will use the symbol $\bm{1}$ to denote
the image in $\frac{\enveloping{L}}{\enveloping{L}\sections{A}}$
of the constant function $1\in R$ under the canonical map
$R\into\enveloping{L}\onto\frac{\enveloping{L}}{\enveloping{L}\sections{A}}$.
We note that $\frac{\enveloping{L}}{\enveloping{L}\sections{A}}$
is naturally a left $\enveloping{L}$-module.
An $L$-connection $\nabla$ on $L/A$ induces
an $L$-connection on $S(L/A)$,
which we still denote $\nabla$ by abuse of notation:
\begin{equation}\label{Edmonton}
\nabla_l(b_1\odot\cdots\odot b_n)
=\sum_{i=1}^{n} b_1\odot\cdots\odot \nabla_l b_i\odot
\cdots\odot b_n ,\end{equation}
for all $l\in\sections{L}$, $n\in\NN$
and $b_1,\dots,b_n\in\sections{L/A}$.
It is understood that $\nabla_l 1=0$, for all $l\in\sections{L}$.

\begin{theorem}\label{Nairobi}
Let $(L,A)$ be a Lie pair.
Given a splitting $j:L/A\to L$ of the short exact sequence
\begin{equation}\label{Sydney}
\begin{tikzcd} 0 \arrow[r] & A \arrow[r, "i"] & L \arrow[r, "q"] & L/A \arrow[r] & 0 \end{tikzcd}
\end{equation}
and an $L$-connection $\nabla$ on $L/A$,
there exists a unique isomorphism of filtered $R$-coalgebras
\[ \pbw^{\nabla,j}:\sections{S(L/A)}\to\tfrac{\enveloping{L}}{\enveloping{L}\sections{A}} \]
satisfying
\begin{gather}
\pbw^{\nabla,j}(1)=\bm{1}; \label{Ottawa} \\
\pbw^{\nabla,j}(b)=j(b)\cdot\bm{1}, \quad\forall b\in\sections{L/A}; \label{Pittsburgh}
\end{gather}
and, for all $n\in\NN$ and $b\in\sections{L/A}$,
\begin{equation}
\pbw^{\nabla,j}(b^{n+1})=j(b)\cdot\pbw^{\nabla,j}(b^n)
-\pbw^{\nabla,j}\big(\nabla_{j(b)}(b^n)\big) \label{Quito}
.\end{equation}
\end{theorem}

Equation~\eqref{Quito} is equivalent to
\begin{multline}\label{Rome}
\pbw^{\nabla,j}(b_0 \odot \cdots \odot b_n) =\tfrac{1}{n+1}\sum_{i=0}^{n} \Big( j(b_i) \cdot
\pbw^{\nabla,j}(b_0 \odot \cdots \odot \widehat{b_i} \odot \cdots \odot b_n ) \\ -
\pbw^{\nabla,j}\big(\nabla_{j(b_i)} ( b_0 \odot \cdots \odot \widehat{b_i} \odot
\cdots \odot b_n ) \big) \Big)
\end{multline}
for all $b_0,\dots,b_n \in \sections{L/A}$.

\begin{remark}
Equation~\eqref{Rome} is similar to an equation obtained by Bordemann in~\cite[page~78]{MR3014184}
when $L$ is a Lie algebra and $A$ is a Lie subalgebra.
\end{remark}

The remainder of this section is devoted to the proof of Theorem~\ref{Nairobi}.

It is immediate that Equations~\eqref{Ottawa}, \eqref{Pittsburgh}, and \eqref{Rome}
together define inductively a unique $R$-linear map $\pbw^{\nabla,j}$.
Furthermore, it is simple to check that this map $\pbw^{\nabla,j}$ respects the filtrations
\eqref{Pune} and~\eqref{Kolkata}.

\begin{proposition}
The filtered $R$-linear map $\pbw^{\nabla,j}$ is a morphism of coalgebras:
\begin{equation}\label{Riga}
\Delta\circ\pbw^{\nabla,j} = (\pbw^{\nabla,j}\otimes\pbw^{\nabla,j})\circ\Delta
.\end{equation}
\end{proposition}

\begin{proof}
By the definition of the comultiplication $\Delta$, we have
\begin{equation}\label{arctic} \Delta(b^n)=
\sum_{\substack{p+q=n \\ p\geqslant 0 \\ q\geqslant 0}}
\frac{n!}{p!\; q!}b^p\otimes b^q ,\end{equation}
for all $b\in\sections{L/A}$ and $n\in\NN$.

It is clear that for all $l\in\sections{L}$,
the operator $\nabla_l$ is a coderivation of the $R$-coalgebra $\sections{S(L/A)}$.
Consequently, we have
\begin{equation}\label{tern} \begin{split}
\Delta\big(\nabla_{j(b)}(b^n)\big) &=
(\nabla_{j(b)}\otimes\id+\id\otimes\nabla_{j(b)})\circ\Delta(b^n) \\
&= (\nabla_{j(b)}\otimes\id+\id\otimes\nabla_{j(b)})
\Big( \sum_{p=0}^n
\binom{n}{p}b^p\otimes b^{n-p} \Big) \\
&= \sum_{p=0}^n \binom{n}{p} \big(\nabla_{j(b)}(b^p)\otimes b^{n-p}
+ b^p\otimes\nabla_{j(b)}(b^{n-p})\big)
.\end{split} \end{equation}

We reason by induction on $n$
to prove that Equation~\eqref{Riga} holds on $\sections{S^n(L/A)}$, for all $n\in\NN$.

Obviously, the claim holds for $n=0$ and $n=1$ since
\[ \Delta\circ\pbw^{\nabla,j}(1)=\Delta(\bm{1})=\bm{1}\otimes\bm{1}
=\pbw^{\nabla,j}(1)\otimes\pbw^{\nabla,j}(1)
=(\pbw^{\nabla,j}\otimes\pbw^{\nabla,j})\circ\Delta(1) \]
and
\begin{multline*} \Delta\circ\pbw^{\nabla,j}(b)=\Delta\big(j(b)\cdot\bm{1}\big)
=j(b)\cdot\bm{1}\otimes\bm{1}+\bm{1}\otimes j(b)\cdot\bm{1} \\
=\pbw^{\nabla,j}(b)\otimes\pbw^{\nabla,j}(1)+\pbw^{\nabla,j}(1)\otimes\pbw^{\nabla,j}(b) \\
=(\pbw^{\nabla,j}\otimes\pbw^{\nabla,j})(b\otimes 1+1\otimes b)
=(\pbw^{\nabla,j}\otimes\pbw^{\nabla,j})\circ\Delta(b) ,\end{multline*}
for all $b\in\sections{L/A}$.
Assuming that Equation~\eqref{Riga} holds on $\sections{S^n(L/A)}$,
it suffices to show that
\[ \Delta\circ\pbw^{\nabla,j}(b^{n+1})=
(\pbw^{\nabla,j}\otimes\pbw^{\nabla,j})\circ\Delta(b^{n+1}),
\qquad\forall b\in\sections{L/A}, \]
to complete the induction.

Using Equation~\eqref{Quito}; the induction hypothesis;
Equation~\eqref{arctic}; and Equation~\eqref{tern}, we obtain
\begin{align*}
& \Delta\circ\pbw^{\nabla,j}(b^{n+1}) \\
&= \Delta\Big( j(b)\cdot\pbw^{\nabla,j}(b^n) - \pbw^{\nabla,j}\big(\nabla_{j(b)}(b^n)\big) \Big) \\
&= \Delta\big(j(b)\big)\cdot\Delta\big(\pbw^{\nabla,j}(b^n)\big)
- \Delta\circ\pbw^{\nabla,j}\big(\nabla_{j(b)}(b^n)\big) \\
&= \Delta\big(j(b)\big)\cdot(\pbw^{\nabla,j}\otimes\pbw^{\nabla,j})\big(\Delta(b^n)\big)
- (\pbw^{\nabla,j}\otimes\pbw^{\nabla,j})\circ\Delta\big(\nabla_{j(b)}(b^n)\big) \\
&= \big(j(b)\otimes 1+1\otimes j(b)\big)\cdot
\Big(\sum_{p=0}^n\binom{n}{p}
\pbw^{\nabla,j}(b^p)\otimes \pbw^{\nabla,j}(b^{n-p})\Big) \\
&\quad - \sum_{p=0}^n
\binom{n}{p} \pbw^{\nabla,j}\big(\nabla_{j(b)}(b^p)\big) \otimes \pbw^{\nabla,j}(b^{n-p}) \\
&\quad - \sum_{p=0}^n
\binom{n}{p} \pbw^{\nabla,j}(b^p)\otimes\pbw^{\nabla,j}\big(\nabla_{j(b)}(b^{n-p})\big) \\
&= \sum_{p=0}^n\binom{n}{p}
\Big(j(b)\cdot\pbw^{\nabla,j}(b^p)\otimes\pbw^{\nabla,j}(b^{n-p})
+\pbw^{\nabla,j}(b^p)\otimes j(b)\cdot\pbw^{\nabla,j}(b^{n-p}) \\
&\quad\qquad -\pbw^{\nabla,j}\big(\nabla_{j(b)}(b^p)\big) \otimes \pbw^{\nabla,j}(b^{n-p})
-\pbw^{\nabla,j}(b^p) \otimes \pbw^{\nabla,j}\big(\nabla_{j(b)}(b^{n-p})\big) \Big) \\
&= \sum_{p=0}^n \binom{n}{p}
\Big( \pbw^{\nabla,j}(b^{p+1})\otimes \pbw^{\nabla,j}(b^{n-p})
+\pbw^{\nabla,j}(b^p)\otimes \pbw^{\nabla,j}(b^{n-p+1}) \Big) \\
&= (\pbw^{\nabla,j}\otimes\pbw^{\nabla,j})
\Big( \sum_{p=0}^n
\binom{n}{p}(b^{p+1}\otimes b^{n-p}+b^p\otimes b^{n-p+1}) \Big) \\
&= (\pbw^{\nabla,j}\otimes\pbw^{\nabla,j})
\Big( \sum_{p=0}^{n+1}
\binom{n+1}{p} b^p\otimes b^{n+1-p} \Big) \\
&= (\pbw^{\nabla,j}\otimes\pbw^{\nabla,j})\circ\Delta(b^{n+1})
.\qedhere\end{align*}
\end{proof}

The remainder of this section is devoted to the proof of Theorem~\ref{Nairobi},
more precisely the claim that $\pbw^{\nabla,j}$
is an \emph{isomorphism} of filtered $R$-coalgebras.

\begin{lemma}\label{Stockholm}
For all $n\in\NN$ and $b_1,\dots,b_n\in\sections{L/A}$,
\[ \pbw^{\nabla,j}(b_1\odot\cdots\odot b_n)
-\left(\frac{1}{n!}\sum_{\sigma\in S_n}
j(b_{\sigma(1)})\cdot j(b_{\sigma(2)})\cdot\cdots\cdot
j(b_{\sigma(n)})\right)\cdot\bm{1} \]
is an element of
$\left(\frac{\enveloping{L}}{\enveloping{L}\sections{A}}\right)^{\leqslant n-1}$.
\end{lemma}

\begin{proof}
It follows from the inductive relation \eqref{Rome} that
\begin{multline*} \pbw^{\nabla,j}(b_1\odot\cdots\odot b_n)
-\frac{1}{n}\sum_{k=1}^n j(b_k)\cdot\pbw^{\nabla,j}(b_1\odot\cdots
\odot\widehat{b_k}\odot\cdots\odot b_n) \\
=-\frac{1}{n}\sum_{k=1}^n \pbw^{\nabla,j}\big(\nabla_{j(b_k)}(b_1\odot\cdots
\odot\widehat{b_k}\odot\cdots\odot b_n)\big) \end{multline*}
belongs to $\left(\frac{\enveloping{L}}{\enveloping{L}\sections{A}}\right)^{\leqslant n-1}$ as
$ \nabla_{j(b_k)}(b_1\odot\cdots\odot\widehat{b_k}\odot\cdots\odot b_n)\in \sections{S^{n-1}(L/A)}$
and $\pbw^{\nabla,j}$ respects the filtrations.
The result follows by induction on $n$.
\end{proof}

\begin{proof}[Proof of Theorem~\ref{Nairobi}]
It remains to prove that $\pbw^{\nabla,j}$ is an isomorphism.
For this purpose, since the filtration is bounded below,
it suffices to prove that the induced map
between the associated graded vector spaces
\[ \Gr(\pbw^{\nabla,j}):\sections{S(L/A)}\to
\Gr\left(\frac{\enveloping{L}}{\enveloping{L}\sections{A}}\right) \]
is an isomorphism.

Since $L$ is a Lie algebroid, the pair $(R,\sections{L})$ is a Lie--Rinehart algebra
whose universal enveloping algebra $\enveloping{R,\sections{L}}$ is precisely
the universal enveloping algebra $\enveloping{L}$ of the Lie algebroid $L$.
The symmetrization map \eqref{Rinehart_sym} arising from the Lie--Rinehart algebra $(R,\sections{L})$
yields an isomorphism of graded associative algebras \cite{MR154906}:
\[ \sym:\sections{S(L)}\to\Gr\big(\enveloping{L}\big) .\]

Consider the canonical projection
\[ \chi:\enveloping{L}\to\tfrac{\enveloping{L}}{\enveloping{L}\sections{A}} \]
and the injection
\[ j:\sections{S(L/A)}\to\sections{S(L)} \]
induced by the splitting $j$.
Lemma~\ref{Stockholm} asserts the commutativity of the diagram
\begin{equation}\label{eq:RinehartComm} \begin{tikzcd}[column sep=huge]
\sections{S(L)} \arrow[r, "\sym"] & \Gr\big(\enveloping{L}\big) \arrow[d, two heads, "\Gr(\chi)"] \\
\sections{S(L/A)} \arrow[u, hook, "j"] \arrow[r, "\Gr(\pbw^{\nabla,j})"'] &
\Gr\left(\frac{\enveloping{L}}{\enveloping{L}\sections{A}}\right)
.\end{tikzcd} \end{equation}
The cokernel of $j:\sections{S^n(L/A)}\to\sections{S^n(L)}$ is clearly isomorphic to
the subspace $\sections{S^{n-1}(L)\odot A}$ of $\sections{S^n(L)}$.
Moreover, the identifications $\Gr^n\big(\enveloping{L}\big) = \frac{\mathcal{U}^{\leqslant n}(L)}
{\mathcal{U}^{\leqslant n-1}(L)}$ and
\[ \Gr^n\left(\frac{\enveloping{L}}{\enveloping{L}\cdot\sections{A}}\right) =
\frac{\mathcal{U}^{\leqslant n}(L)}{\mathcal{U}^{\leqslant n-1}(L)\cdot\sections{A}
+\mathcal{U}^{\leqslant n-1}(L)} ,\]
imply that
\[ \ker\big(\Gr^n(\chi)\big) =
\frac{\mathcal{U}^{\leqslant n-1}(L)\cdot\sections{A}
+\mathcal{U}^{\leqslant n-1}(L)}{\mathcal{U}^{\leqslant n-1}(L)} .\]

Since the symmetrization map $\sym:\sections{S(L)}\to\Gr\big(\enveloping{L}\big)$
is a morphism of algebras,
it maps the subspace $\sections{S^{n-1}(L)\odot A}$ of $\sections{S^n(L)}$ to the subspace
$\frac{\mathcal{U}^{\leqslant n-1}(L)\cdot\sections{A}
+\mathcal{U}^{\leqslant n-1}(L)}{\mathcal{U}^{\leqslant n-1}(L)}$
of $\Gr^n\big(\enveloping{L}\big)$.
Thus, the symmetrization map $\sym:\sections{S(L)}\to\Gr\big(\enveloping{L}\big)$
identifies the cokernel of $j$ with the kernel of $\Gr(\chi)$.
Since $j$ is injective, $\sym$ is bijective, $\Gr(\chi)$ is surjective,
it follows from the commutative diagram \eqref{eq:RinehartComm}
that $\Gr(\pbw^{\nabla,j})$ is an isomorphism.
\end{proof}

\begin{remark}
In the special case of Lie pairs over $\RR$,
Theorem~\ref{Madagascar} below will shed light
on the origin of Equations~\eqref{Ottawa},
\eqref{Pittsburgh}, and~\eqref{Quito}.
Indeed, Theorem~\ref{Madagascar} will give a more enlightening proof of Theorem~\ref{Nairobi}
exploiting the close relation between $\pbw^{\nabla,j}$ and exponential maps.
\end{remark}

\section{Exponential maps}
\label{Urumqi}

Throughout this section, we work over the field $\RR$ of real numbers exclusively.
Thus we restrict ourselves to real Lie algebroids.

For any Lie group, there is a canonical exponential map
from the associated Lie algebra to the Lie group itself.
The construction of the exponential map generalizes to Lie groupoids
though in a noncanonical way as a choice of connection is needed~\cite{MR1722129}.
In the particular case of a Lie group $\group{G}$,
one of these connection-induced exponential maps stands out:
the one associated with the trivial connection defined by $\nabla_X Y =0$
for every left-invariant vector fields $X$ and $Y$ on $\group{G}$.

Here, given an inclusion of Lie groupoids $\groupoid{A}\into\groupoid{L}$
(and the corresponding inclusion of Lie algebroids $A\into L$),
we define an exponential map which takes a neighborhood of the zero section of $L/A$
to a neighborhood of the unit section of $\groupoid{L}/\groupoid{A}$.
The definition requires two choices:
(i) a splitting of the short exact sequence of vector bundles \eqref{Sydney}
and (ii) an $L$-connection on $L/A$.

\subsection{\texorpdfstring{$\source$}{s}- and \texorpdfstring{$\anchor$}{ρ}-paths}

We briefly recall two important related classes of paths
in Lie algebroids and groupoids \cite{MR2795150}.
See `Terminology and Notations' at the end of the Introduction 
for our groupoid multiplication convention.
In what follows, the symbol $I$ always denotes some open interval of the real line containing $0$.

Let $\groupoid{L}$ be a Lie groupoid over a manifold $M$ with source $\source:\groupoid{L}\to M$,
target $\target:\groupoid{L}\to M$, and unit $1:M\to\groupoid{L}$.
The subbundle $L=\{ X \in 1^* (T_{\groupoid{L}}) | s_* X=0 \}$
of the pullback of $T_{\groupoid{L}}$ via $1:M\to\groupoid{L}$
carries a Lie algebroid structure whose anchor $\anchor:L\to T_M$ is the (restriction of)
the differential $\target_*:T_{\groupoid{L}}\to T_M$ of the target~\cite{MR2157566}.
We write $\pi$ to denote the bundle projection $L\to M$.

An \emph{$\source$-path} is a smooth curve $\gamma:I\to\groupoid{L}$
originating from a point of the unit submanifold of $\groupoid{L}$
(i.e.\ $\gamma(0)=1_m$ for some $m\in M$) and fully contained in one of the $\source$-fibers
(i.e.\ $\source\circ\gamma(t)=m$ for all $t\in I$).
A \emph{$\anchor$-path} is a smooth curve $\beta:I\to L$ satisfying
\[ \pi_*\big(\beta'(t)\big)=\anchor\big(\beta(t)\big), \quad\forall t\in I .\]
Every $\source$-path $\gamma$ determines a unique $\anchor$-path $\beta$
and vice versa through the initial value problem
\[ \left\{ \begin{aligned}
\beta(t)&=\left.\frac{d}{d\tau} \big(\gamma(t)\big)^{-1}\gamma(\tau) \right|_t ,\\
\gamma(0)&=1_{\pi(\beta(0))}
.\end{aligned} \right. \]

\subsection{Connections as horizontal liftings}
\label{SantaFe}

Let $M$ be a smooth manifold, let $L\to M$ be a Lie algebroid with anchor map $\anchor$,
and let $E\xto{\varpi}M$ be a vector bundle.
A linear $L$-connection on $E$ can be described not only as
a covariant derivative $\nabla:\sections{L}\times\sections{E}\to\sections{E}$
but also as a horizontal lifting.

\begin{definition}
A (linear) $L$-connection on $E$ is a map $L\times_M E\xto{h} T_E$,
called \emph{horizontal lifting}, such that the diagram
\begin{equation}\label{Warsaw}
\begin{tikzcd}[row sep=small] & L \arrow[rr, "\anchor"]\arrow[dd] & & T_M \arrow[dd] \\
L\times_M E \arrow[rr, "h", dashrightarrow, near end, crossing over] \arrow[ru] \arrow[dd]
& & T_E \arrow[ru, "\varpi_*"'] & \\
& M \arrow[rr, "\id", near start] & & M \\
E \arrow[ru, "\varpi"] \arrow[rr, "\id"] & &
E \arrow[from=uu, crossing over] \arrow[ru, "\varpi"'] &
\end{tikzcd}
\end{equation}
commutes and its front and top faces
\[ \begin{tikzcd} L \times_M E \arrow[r, "h", dashrightarrow] \arrow[d] & T_E \arrow[d] \\
E \arrow[r, "\id"'] & E \end{tikzcd}
\qquad\text{and}\qquad
\begin{tikzcd} L \times_M E \arrow[r, "h", dashrightarrow] \arrow[d]
& T_E \arrow[d, "\varpi_*"] \\ L \arrow[r, "\anchor"'] & T_M \end{tikzcd} \]
are vector bundle maps.
In other words, the diagram \eqref{Warsaw} is a morphism of double vector bundles
--- see \cite{MR2157566} --- from the left face of the cube to the right face of the cube.
\end{definition}

The covariant derivative and horizontal lift describing a given connection are
related to one another by the identity
\begin{equation}\label{XiAn}
e_*\big(\anchor(l_m)\big)-h(l_m,e_m)=\tau_{e_m} \big((\nabla_l e)_m\big) ,
\end{equation}
which holds for all $m\in M$, $l\in\sections{L}$, and $e\in\sections{E}$
--- see~\cite[Theorem~12.32]{MR2572292} or~\cite[page~114]{MR0152974}.
Here, $\tau_{e_m}$ denotes the canonical linear isomorphism
between the fiber $E_m$ and its tangent space at the point $e_m$.

We note that the horizontal lifting $h$ induces a morphism
$\sections{L}\ni l\mapsto l^\diamond\in\XX(E)$ of $C^\infty(M)$-modules
taking values in the Lie subalgebra $\XX_{\varpi}(E)$ of
$\varpi$-projectable vector fields making the following diagram commute:
\[ \begin{tikzcd}[row sep=tiny]
& \XX_{\varpi}(E) \arrow[dd, "\varpi_*"] \\
\sections{L} \arrow[ru, "\diamond"] \arrow[rd, "\anchor"'] & \\
& \XX(M)
.\end{tikzcd} \]
For all $l\in\sections{L}$ and $e\in E$, we have
\begin{equation}
\label{gazpacho}
l^\diamond_e=h(l_{\pi(e)},e)
.\end{equation}

When the $L$-connection is flat,
$l\mapsto l^\diamond$ is a Lie algebra morphism.
This point of view will be used in Section~\ref{Patagonia}.

\subsection{Geodesic spray}
Suppose a splitting
\begin{equation}\label{Murmansk}
\begin{tikzcd} 0 \arrow[r] & A \arrow[r, "i"] & L
\arrow[l, bend left, "p"] \arrow[r, "q"] &
L/A \arrow[l, bend left, "j"] \arrow[r] & 0 \end{tikzcd} \end{equation}
of the short exact sequence \eqref{Sydney} and an $L$-connection $\nabla$ on $L/A$ have been chosen.
Together, the connection and the splitting determine
a vector field $\Xi$ on $L/A$:
the value of $\Xi$ at a point $x\in L/A$ is the horizontal lift of the vector $j(x)\in L$.
More precisely,
\begin{equation} \Xi_x=h\big(j(x),x\big), \quad\forall x\in L/A ,\label{Lancaster}\end{equation}
where $h:L\times_M (L/A) \to T_{L/A}$ denotes the horizontal lifting
associated with the $L$-connection on $L/A$ as in~\eqref{Warsaw}.

Let $\pi$ denote the projection $L/A\to M$. Since
\begin{equation} \pi_*(\Xi_x)=\anchor\big(j(x)\big), \quad\forall x\in L/A, \label{York}\end{equation}
the splitting $j:L/A\to L$ maps the integral curves of $\Xi$ in $L/A$
to $\anchor$-paths in the Lie algebroid $L$.
This very property generalizes the classical notion of spray to the context of Lie algebroids.
Furthermore, the integral curves $t\mapsto b(t)$ of $\Xi$
satisfy the differential equation \[ b'(t)=h\big(j(b(t)),b(t)\big) ,\]
which is similar to the usual geodesic equation.
Therefore, we call the vector field $\Xi$
the \emph{geodesic spray} associated with the connection $\nabla$ and the splitting $j$.

\begin{remark}\label{Yerevan}
The flow $t\mapsto\Phi^\Xi_t$ of the geodesic spray satisfies the following important property:
$\Phi^\Xi_t(rx)=r\Phi^\Xi_{rt}(x)$, for all $x\in L/A$ and all real numbers $t,r$ in a sufficiently small open interval around $0$, 
which depends on $x$.
\end{remark}

Let $\sigma$ be a smooth section of $L/A\xto{\pi}M$.
Since the flow of the geodesic spray $\Xi$ does not respect
the fibers of the vector bundle $L/A\xto{\pi}M$,
the image of $\sigma(M)$ under $\Phi^\Xi_t$ (for a fixed numerical value of the time parameter $t$)
is not guaranteed to be the image of another section of $L/A\xto{\pi}M$.
Nevertheless, we proceed with showing that, given any point $m$ of $M$,
there exists an open neighborhood $U_m$ of $m$
and a small open interval $I_m$ of $\RR$ containing $0$
such that, for each $t\in I_m$,
the image of $\sigma(U_m)$ under the flow $\Phi^{\Xi}_t$
is the image of some local section of $L/A\xto{\pi}M$, which we denote $\sigma_t$.
(Note that the domain of definition of the local section $\sigma_t$ varies with $t$.
Indeed, the domain of definition of $\sigma_t$ is $\pi\circ\Phi^{\Xi}_t\circ\sigma(U_m)$.)

Consider the map $\Psi^{\Xi,\sigma}:\RR\times M\to\RR\times M$ defined by
\[ \Psi^{\Xi,\sigma}(t,m)=\big(t,\pi\circ\Phi^{\Xi}_t\circ\sigma(m)\big) ,\]
which fixes all points of the submanifold $\{0\}\times M$.
Actually, unless the geodesic spray $\Xi$ is complete,
$\Psi^{\Xi,\sigma}$ is only defined on a neighborhood of $\{0\}\times M$.
The differential of $\Psi^{\Xi,\sigma}$ is clearly nonsingular at every point of $\{0\}\times M$.
Therefore $\Psi^{\Xi,\sigma}$ is a diffeomorphism from an open neighborhood $V$ of $\{0\}\times M$
in $\RR\times M$ to some other open neighborhood $W$ of $\{0\}\times M$ in $\RR\times M$.
The inverse of $\Psi^{\Xi,\sigma}$ is of the form
\[ (\Psi^{\Xi,\sigma}\big)^{-1}(t,m)=\big(t,(\psi^{\Xi,\sigma}_t\big)^{-1}(m)\big) ,\]
where, for each fixed $t$, the symbol $\psi^{\Xi,\sigma}_t$ denotes the bijective map
\[ \{ m\in M \mathrel{|} (t,m)\in V \}
\xto{\pi\circ\Phi^{\Xi}_t\circ\sigma}
\{ m\in M \mathrel{|} (t,m)\in W \} .\]

We define the one-parameter family of maps $\sigma_t$ as follows:
\begin{equation}\label{Canada} \sigma_t(m)=
\Phi^{\Xi}_t\circ\sigma\circ(\psi^{\Xi,\sigma}_t\big)^{-1}(m),
\quad\forall(t,m)\in W
.\end{equation}
By construction, $\sigma_0=\sigma$ and $\sigma_t$ is
a section of $L/A\xto{\pi}M$ over $\{ m\in M \mathrel{|} (t,m)\in W \}$.

The following lemma will be used later in the proof of Lemma~\ref{Oman}.

\begin{lemma}\label{Denmark}
For every section $\sigma$ of the vector bundle $L/A\xto{\pi}M$, we have
\[ \left.\frac{d\sigma_t}{dt}\right|_0=-\nabla_{j(\sigma)}\sigma .\]
\end{lemma}

\begin{proof}
Differentiating $\Phi^{\Xi}_t\circ\sigma(m)=\sigma_t\circ\psi^{\Xi,\sigma}_t(m)$
with respect to the parameter $t$, we obtain
\[ \left.\frac{d}{dt}\Phi^{\Xi}_t\big(\sigma(m)\big)\right|_0
= \left.\frac{d}{dt}\sigma_t(m)\right|_0
+\left.\frac{d}{dt}\sigma\circ\pi\circ\Phi^{\Xi}_t\circ\sigma(m)\right|_0 \]
and hence, making use of Equations~\eqref{Lancaster}, \eqref{York} and~\eqref{XiAn},
\begin{align*} \left.\frac{d}{dt}\sigma_t(m)\right|_0 =&\
\Xi_{\sigma(m)}-\sigma_*\pi_*\Xi_{\sigma(m)} \\
=&\ h\big(j(\sigma(m)),\sigma(m)\big)-\sigma_*\anchor\big(j(\sigma(m))\big) \\
=&\ -\tau_{\sigma(m)}\big((\nabla_{j(\sigma)}\sigma)(m)\big)
.\end{align*}
Therefore, we have established the equality
$\left.\frac{d\sigma_t}{dt}\right|_0=-\nabla_{j(\sigma)}\sigma$
in $\sections{L/A}$.
\end{proof}

\subsection{Exponential map}
\label{Hungary}
We can now outline the definition of the exponential map
$\exp^{\nabla,j}$ from \emph{a neighborhood of the zero section of} $L/A$ to
\emph{a neighborhood of the unit section of} $\groupoid{L}/\groupoid{A}$.
This map depends on the splitting $j$ and the connection $\nabla$ we have chosen;
different choices yielding different exponential maps.
\begin{enumerate}
\item Given an element $x\in L/A$, consider the integral curve $t\mapsto b_x(t)$ in $L/A$
of the geodesic spray $\Xi$ originating from $b_x(0)=x$.
According to Remark~\ref{Yerevan}, provided $x$ lies sufficiently close to the zero section of $L/A$,
this curve is defined up to time $t=1$.
\item The lifted curve $t\mapsto j\big(b_x(t)\big)$ is a $\anchor$-path in the Lie algebroid $L$,
for the `image' of the geodesic equation $b'_x(t) = h\big(j(b_x(t)),b_x(t)\big)$
under the projection $\pi_*:T_{L/A}\to T_M$ is precisely the $\anchor$-path equation
$\pi_* \big((j\circ b_x)'(t)\big) = \anchor\big( j\circ b_x(t) \big)$
as $\pi_*\big(h(l,e)\big)=\anchor(l)$ per Diagram~\eqref{Warsaw}.
\item The $\anchor$-path $t\mapsto j\circ b_x(t)$ determines a unique $\source$-path
$t\mapsto g^x(t)$ in the Lie groupoid $\groupoid{L}$.
\item Composing the path $g^x$ with the canonical projection
$\groupoid{L}\to\groupoid{L}/\groupoid{A}$,
we obtain a path in $\groupoid{L}/\groupoid{A}$ whose value at $t=1$
is taken to be the image of $x$ by the exponential map $\exp^{\nabla,j}$.
\end{enumerate}

Now, we are ready to introduce:

\begin{definition}\label{Italy}
The exponential map \[ \exp^{\nabla,j}: L/A\to\groupoid{L}/\groupoid{A} \]
associated with an $L$-connection $\nabla$ on $L/A$ and a splitting $j:L/A\to L$
is the map which takes a point $x$ of $L/A$
to the projection of $g^x(1)$ in $\groupoid{L}/\groupoid{A}$,
where $t\mapsto g^x(t)$ is the unique path in $\groupoid{L}$ satisfying
\[ \left.\frac{d}{d\tau}\big(g^x(t)\big)^{-1}g^x(\tau)\right|_t = j\big(\Phi^{\Xi}_t(x)\big)
\qquad\text{and}\qquad g^x(0)=1_{\pi(x)} .\]
\end{definition}

The exponential map is well-defined on a neighborhood of the zero section of $L/A$.

\pagebreak[2]
\begin{lemma}\label{Japan}
\strut
\begin{enumerate}
\item For any $x\in L/A$ and any real numbers $r,t$
in a sufficiently small open interval around $0$,
we have $g^{rx}(t)=g^{x}(rt)$.
\item For any $x\in L/A$ and any real numbers $t_1,t_2$
in a sufficiently small open interval around $0$, we have
$g^{x}(t_1+t_2)=g^{x}(t_1)\cdot g^{\Phi^\Xi_{t_1}(x)}(t_2)$.
\end{enumerate}
\end{lemma}

\begin{proof}
The first assertion is a straightforward consequence of Remark~\ref{Yerevan}
while the second one follows immediately from the flow property
$\Phi^\Xi_{t_1+t_2}=\Phi^\Xi_{t_2}\circ\Phi^\Xi_{t_1}$.
\end{proof}

\begin{proposition}\label{Koweit}
The exponential map $\exp^{\nabla,j}:L/A\to\groupoid{L}/\groupoid{A}$
is a fiber bundle map over $\id:M\to M$,
which maps the zero section of $L/A$ to the unit section of $\groupoid{L}/\groupoid{A}$,
and identifies diffeomorphically a neighborhood of the zero section of $L/A$
to a neighborhood of the unit section of $\groupoid{L}/\groupoid{A}$.
\end{proposition}

\begin{proof}The result follows from the following two observations:
\textbf{(1)} $\exp^{\nabla,j}(0_m)=1_m$ for every $m\in M$ since the geodesic spray $\Xi$
vanishes along the zero section of $L/A$;
\textbf{(2)} along the zero section, the differential of $\exp^{\nabla,j}$
in the direction of the fibers is the identity map since, by Lemma~\ref{Japan},
\begin{multline*}
\left.\frac{d}{dt}\exp^{\nabla,j}(tx)\right|_0
=\left.\frac{d}{dt}\bm{q}\big(g^{tx}(1)\big)\right|_0
=\bm{q}_*\left(\left.\frac{d}{dt} g^x(t)\right|_0\right) \\
=\bm{q}_*\left(\left.\frac{d}{dt} \big(g^x(0)\big)^{-1}\cdot g^x(t)\right|_0\right)
=q\big(j\circ \Phi^\Xi_{0}(x)\big)=\Phi^\Xi_{0}(x)=x,
\end{multline*}
where $\bm{q}:\groupoid{L}\to\groupoid{L}/\groupoid{A}$
and $q:L\to L/A$ are the canonical projections.
\end{proof}

\begin{example}\label{ex:matched}
When the subbundle $B=j(L/A)$ of $L$ to which $L/A$ is identified by the
splitting \eqref{Murmansk} happens to be a Lie subalgebroid of $L$,
we say that the complementary Lie subalgebroids $A$ and $B$ of $L$
form a \emph{matched pair} \cite{MR1460632}.
The vector bundle $L$ being the Whitney sum $A\oplus B$,
we use the notation $L=A\bowtie B$ to denote that $A$ and $B$ form a matched pair of Lie algebroids.
In this situation, a choice of $B$-connection $\ddot{\nabla}$ on $B$
determines an $L$-connection $\nabla$ on $L/A\cong B$
through the relation
\begin{equation}
\label{eq:matched}
\nabla_{a+b}\; c = \nabla^{\Bott}_a c+\ddot{\nabla}_{b}\; c
,\quad\forall a\in\sections{A},\forall b,c\in\sections{B} .
\end{equation}
The resulting map $\exp^{\nabla,j}$ is the Landsman exponential map
\cite{MR1722129,MR1687747}
corresponding to the $B$-connection $\ddot{\nabla}$ on $B$.
A particular case of this situation arises when $A$ is the trivial subbundle of rank $0$
of the Lie algebroid $L$ and the splitting $j$ is (necessarily) the identity map on $L$.
\end{example}

\begin{remark}
The reader might wonder whether there exists an $L$-connection $\widetilde{\nabla}$ on $L$
whose associated Landsman exponential map $\exp^{\widetilde{\nabla}}$ makes the diagram
\begin{equation}\label{eq:Landsmanexp}
\begin{tikzcd}
L \arrow[r, "\exp^{\widetilde{\nabla}}"] & \mathscr{L} \arrow[d, two heads, "\bm{q}"] \\
L/A \arrow[u, hook, "j"] \arrow[r, "\exp^{\nabla,j}"'] & \mathscr{L}/\mathscr{A}
\end{tikzcd}
\end{equation}
commute.
One can prove that, if an $L$-connection $\widetilde{\nabla}$ on $L$ satisfies
\begin{equation}\label{La_Chaux-de-Fonds}
\widetilde{\nabla}_{l}\big(j(b)\big)=j\big(\nabla_{l}b\big)
,\quad\forall l\in\sections{L},\forall b\in\sections{L/A}
,\end{equation}
then the associated Landsman exponential map $\exp^{\widetilde{\nabla}}$
does indeed make the diagram \eqref{eq:Landsmanexp} commute.
Moreover, given a splitting $j$ and an $L$-connection $\nabla$ on $L/A$,
one can always construct an $L$-connection $\widetilde{\nabla}$ on $L$ satisfying
\eqref{La_Chaux-de-Fonds}.
Therefore, we could also have defined our exponential map
$\exp^{\nabla,j}:L/A\to\mathscr{L}/\mathscr{A}$
as the composition $\bm{q}\circ\exp^{\widetilde{\nabla}}\circ j$.
However, working with $\widetilde{\nabla}$ rather than $\nabla$ introduces
extraneous information in the construction:
knowledge of $\widetilde{\nabla}_l i(a)$ for $a\in\sections{A}$ is superfluous.
Finally, we would like to stress the following important point:
it is crucial that, for all $l\in\sections{L}$,
the subspace $j\big(\sections{L/A}\big)$ of $\sections{L}$ be stable
under the operation $x\mapsto\widetilde{\nabla}_l x$,
which is guaranteed by condition \eqref{La_Chaux-de-Fonds}.
Otherwise, the composition $\bm{q}\circ\exp^{\widetilde{\nabla}}\circ j$
might not be a local diffeomorphism near the zero section of $L/A$.
\end{remark}

\subsection{From the exponential map to the PBW isomorphism}

Given a fiber bundle $\pi:P\to M$ and a section $\epsilon:M\to P$,
consider the space $\DO(P,M)$ of all maps $C^\infty(P)\to C^\infty(M)$
obtained by composition of a $\pi$-fiberwise differential operator on $P$
(seen as an endomorphism of $C^\infty(P)$)
with the restriction $\epsilon^* : C^\infty(P) \to C^\infty(M)$.
In other words, $\DO(P,M)$ is the space of $\pi$-fiberwise distributions on $P$
supported on $\epsilon(M)$.
The space $\DO(P,M)$ is an $R$-coalgebra filtered by the order of the differential operators.
Here $R = C^\infty(M)$.

Given an isomorphism of fiber bundles
\[ \begin{tikzcd} P_1 \arrow[r, "\Psi"] \arrow[d, "\pi_1"] & P_2 \arrow[d, "\pi_2"'] \\
M \arrow[u, dashed, bend left, "\epsilon_1"] \arrow[r, "\id"'] & M
\arrow[u, dashed, bend right, "\epsilon_2"']\end{tikzcd} \]
and a pair of sections $\epsilon_1:M\to P_1$ and $\epsilon_2:M\to P_2$
such that $\Psi\circ\epsilon_1=\epsilon_2$,
the algebra isomorphism $\Psi^*:C^\infty(P_2)\to C^\infty(P_1)$
induces an isomorphism of filtered $R$-coalgebras
\[ \Psi_*:\DO(P_1,M)\to\DO(P_2,M) .\]
In other words, $\Psi_*$ is the push-forward of fiberwise distributions through $\Psi$.

For instance, given an inclusion of Lie groupoids $\groupoid{A}\into\groupoid{L}$
and the corresponding inclusion of Lie algebroids $A\into L$, the exponential map
$\exp^{\nabla,j}:L/A\to\groupoid{L}/\groupoid{A}$ of Definition~\ref{Italy}
induces an isomorphism of $R$-coalgebras
\begin{equation}\label{Laos} \exp^{\nabla,j}_*:\DO(L/A,M)\to\DO(\groupoid{L}/\groupoid{A},M)
.\end{equation}
Indeed, in view of Proposition~\ref{Koweit}, it suffices to take
\begin{align*}
\pi_1&=(L/A\to M) & \pi_2&=(\groupoid{L}/\groupoid{A}\xto{s}M) \\
\epsilon_1&=(M\xto{0}L/A) & \epsilon_2&=(M\xto{1}\groupoid{L}/\groupoid{A}) \\
\Psi&=(L/A\xto{\exp^{\nabla,j}}\groupoid{L}/\groupoid{A}) & &
\end{align*}
in the construction above.

\begin{lemma}\label{Krakow}
Let $\groupoid{A}\hookrightarrow\groupoid{L}$ be an inclusion of Lie groupoids
having the same space of units and let $A\hookrightarrow L$ denote
the corresponding inclusion of Lie algebroids.
\begin{enumerate}
\item The filtered $R$-coalgebras $\DO(L/A,M)$ and $\sections{S(L/A)}$ are canonically isomorphic.
\item The filtered $R$-coalgebras $\DO(\groupoid{L}/\groupoid{A},M)$ and
$\frac{\enveloping{L}}{\enveloping{L}\sections{A}}$ are canonically isomorphic.
\end{enumerate}
\end{lemma}

\begin{proof}[Sketch of proof]
\textbf{(1)} To the symmetric product $x_1\odot x_2\odot\cdots\odot x_n$ in $\sections{S(L/A)}$ of
sections $x_1,x_2,\dots,x_n$ of $L/A\to M$, we associate the fiberwise differential operator
in $\DO(L/A,M)$ which takes a function $f\in C^\infty(L/A)$ to the function
\[ 0^*\big(\underline{x_1}\circ\underline{x_2}\circ\dots\circ\underline{x_n}(f)\big)
\in C^\infty(M) .\]
Here $0$ denotes the zero section of the vector bundle $L/A\to M$
while $\underline{x}$ stands for the vertical and fiberwise constant vector field
on $L/A$ with flow $\Phi^{\underline{x}}_t(b)=b+tx_{\pi(b)}$ for all $b\in L/A$.
\textbf{(2)} To the image in $\frac{\enveloping{L}}{\enveloping{L}\sections{A}}$ of the product
$x_1\cdot x_2\cdot \cdots \cdot x_n\in\enveloping{L}$
of sections $x_1,x_2,\dots,x_n$ of $L$, we associate the fiberwise differential operator
in $\DO(\groupoid{L}/\groupoid{A},M)$ which takes a function
$f\in C^\infty(\groupoid{L}/\groupoid{A})$ to the function
\[ 1^*\big(\overrightarrow{x_1}\circ\overrightarrow{x_2}
\circ\dots\circ\overrightarrow{x_n}(\bm{q}^*f)\big)\in C^\infty(M) .\]
Here $\overrightarrow{x}$ stands for the left-invariant vector field on $\groupoid{L}$
corresponding to the section $x$ of $L$.
As previously, $\bm{q}$ denotes the canonical projection
$\bm{q}:\groupoid{L}\to\groupoid{L}/\groupoid{A}$
and $1$ denotes the unit map $1:M\to\groupoid{L}$.
\end{proof}

It turns out that the isomorphism~\eqref{Laos} coincides with the PBW isomorphism
defined in Section~\ref{Madrid}. This is the result that was announced as Theorem~\ref{Copenhagen}
in Section~\ref{Biloxi}.

\begin{theorem}\label{Madagascar}
Let $\groupoid{A}\hookrightarrow\groupoid{L}$ be an inclusion of Lie groupoids
having the same space of units
and let $A\hookrightarrow L$ denote the corresponding inclusion of Lie algebroids.
Given a splitting $j:L/A \to L$ of the short exact sequence of vector bundles
\eqref{Sydney} and an $L$-connection $\nabla$ on $L/A$,
the filtered $R$-coalgebra isomorphism
\[ \exp^{\nabla,j}_*:\DO(L/A,M)\to\DO(\groupoid{L}/\groupoid{A},M) \]
induced by the exponential map coincides with the Poincaré--Birkhoff--Witt isomorphism
\[ \pbw^{\nabla,j}:\sections{S(L/A)}\to\frac{\enveloping{L}}{\enveloping{L}\sections{A}} \]
in Theorem~\ref{Nairobi}.
\end{theorem}

The remainder of this section is devoted to the proof of Theorem~\ref{Madagascar}.

We will make use of the bundle map
\[ \begin{tikzcd} L/A \arrow[r, "\bm{E}"] \arrow[d, "\pi"'] & \groupoid{L} \arrow[d, "s"] \\
M \arrow[r, "\id"'] & M \end{tikzcd} \]
defined by $\bm{E}(x)=g^x(1)$ for all $x\in L/A$,
and the associated map
$\mathscr{E}:\sections{L/A}\to\sections{\groupoid{L}}$.
We remind the reader that the exponential map $\exp^{\nabla,j}$
is simply the composition of $\bm{E}:L/A\to\groupoid{L}$
with the canonical projection $\bm{q}:\groupoid{L}\to\groupoid{L}/\groupoid{A}$.

Given a point $l\in\groupoid{L}$ and a section 
$y\in\sections{\groupoid{L}}$, 
we write $l\cdot y$ to denote the product $l\cdot y_{t(l)}\in\groupoid{L}$ 
of $l$ and the value of the section $y$ at $t(l)$.

\pagebreak[2]
\begin{lemma}\label{Niger}
\strut
\begin{enumerate}
\item For any $l\in\groupoid{L}$ and $x\in\sections{L/A}$,
we have \[ \left. \frac{d}{dt} l\cdot\mathscr{E}(t x) \right|_0=\overrightarrow{j(x)}|_l ,\]
where $\overrightarrow{j(x)}$ denotes the left-invariant vector field
on the Lie groupoid $\groupoid{L}$
(tangent to the s-fibers) associated with the section $j(x)$ of the Lie algebroid $L$.
\item For any $x\in\sections{L/A}$ and any real numbers $t_1,t_2$
in a sufficiently small open interval around $0$,
we have \[ \mathscr{E}\big((t_1+t_2)x\big)=\mathscr{E}(t_1x)\cdot\mathscr{E}\big(t_2 x_{t_1}\big) ,\]
where the r.h.s.\ is a product of $s$-sections of $\groupoid{L}$ and
\[ x_{t_1}=\Phi^{\Xi}_{t_1}\circ x\circ (\psi^{\Xi,x}_{t_1})^{-1} \] as in Equation~\eqref{Canada}.
\end{enumerate}
\end{lemma}

\begin{proof}
Let $f$ denote an arbitrary smooth function on $\groupoid{L}$.
The first assertion follows from
\[ \left.\frac{d}{dt} f\big(l\cdot\mathscr{E}(t x)\big)\right|_0
=\left.\frac{d}{dt}f\big(l\cdot g^{t x}(1)\big)\right|_0
=\left.\frac{d}{dt}f\big(l\cdot g^{x}(t)\big)\right|_0
=\overrightarrow{j(x)}|_l (f) ,\]
in which we have made use of the first part of Lemma~\ref{Japan}.
The second part of Lemma~\ref{Japan} can be rewritten as
\[ \bm{E}\big((t_1+t_2)x\big)=
\bm{E}(t_1 x)\cdot\bm{E}\big(t_2\Phi^{\Xi}_{t_1}(x)\big)
,\quad\forall x\in L/A.\]
The second assertion follows immediately since $\mathscr{E}$ is the map
on sections associated with the bundle map $\bm{E}$.
\end{proof}

Recall that the universal enveloping algebra $\enveloping{L}$ of the Lie algebroid $L$
can be regarded as the associative algebra of $s$-fiberwise differential operators
on $\cinf{\groupoid{L}}$ invariant under left translations.
The map $\mathscr{E}:\sections{L/A}\to\sections{\groupoid{L}}$
induces a map $\mathcal{E}:\sections{S(L/A)}\to\enveloping{L}$ as follows.
The element $\mathcal{E}[x_1\odot x_2\odot\cdots\odot x_n]$ in $\enveloping{L}$ corresponding to
$x_1\odot x_2\odot\cdots\odot x_n$ in $\sections{S^n(L/A)}$
is the left invariant differential operator which
takes a function $f\in\cinf{\groupoid{L}}$ to the function
\[ \groupoid{L}\ni l\mapsto \left.\frac{d}{dt_1}\right|_0
\left.\frac{d}{dt_2}\right|_0 \cdots \left.\frac{d}{dt_n}\right|_0
f\big( l\cdot\mathscr{E}(t_1 x_1+t_2 x_2 + \cdots + t_n x_n) \big) \in \RR .\]
In particular, for any $x\in\sections{L/A}$ and $n\in\NN$, we have
\[ \big(\mathcal{E}[x^n] f\big)(l)
=\left.\frac{d}{dt_1}\right|_0 \left.\frac{d}{dt_2}\right|_0 \cdots
\left.\frac{d}{dt_n}\right|_0 f\Big( l\cdot\mathscr{E}\big((t_1+t_2+\cdots+t_n)x\big)\Big)
=\left.\frac{d^n}{dt^n}f\big( l\cdot\mathscr{E}(tx)\big)\right|_0 .\]
We declare that $\mathcal{E}[r]=r$ for every $r\in R=C^\infty(M)$
seen as an element of degree 0 in $\sections{S(L/A)}$.

It is easy to check that $\mathcal{E}$ is a morphism of left $R$-modules.

\begin{lemma}\label{Oman}
For all $x\in\sections{L/A}$ and $n\in\NN$, we have $\mathcal{E}[x]=\overrightarrow{j(x)}$ and
\[ \mathcal{E}[x^{n+1}]=\overrightarrow{j(x)}\circ\mathcal{E}[x^n]-\mathcal{E}[\nabla_{j(x)}(x^n)] .\]
\end{lemma}

\begin{proof}
Indeed, it follows from Lemma~\ref{Niger} that, for all $f\in\cinf{\groupoid{L}}$,
\[ \big(\mathcal{E}[x] f\big)(l)=\overrightarrow{j(x)}|_l (f) \]
and
\begin{align*}
\big(\mathcal{E}[x^{n+1}] f\big)(l)
=&\left.\frac{d}{dt_0}\right|_0 \left.\frac{d}{dt_1}\right|_0 \cdots \left.\frac{d}{dt_n}\right|_0
f\Big( l\cdot\mathscr{E}\big((t_0+t_1+\cdots+t_n)x\big)\Big) \\
=&\left.\frac{d}{dt_0}\right|_0 \left.\frac{d}{dt_1}\right|_0 \cdots \left.\frac{d}{dt_n}\right|_0
f\Big( l\cdot\mathscr{E}(t_0 x)\cdot\mathscr{E}\big((t_1+t_2+\cdots+t_n)x_{t_0}\big)\Big) \\
=&\left.\frac{d}{dt_0}\right|_0 \Big(\mathcal{E}\big[(x_{t_0})^n\big]f\Big)
\big(l\cdot\mathscr{E}(t_0 x)\big) \\
=&\left.\frac{d}{dt_0}\right|_0 \Big(\mathcal{E}\big[x^n\big]f\Big)
\big(l\cdot\mathscr{E}(t_0 x)\big)
+\left.\frac{d}{dt_0}\right|_0 \Big(\mathcal{E}\big[(x_{t_0})^n\big]f\Big) \big(l\big) \\
=&\Big(\overrightarrow{j(x)}\big(\mathcal{E}[x^n]f\big)\Big)(l)
+\Big(\mathcal{E}\Big[\left.\frac{d}{dt_0}(x_{t_0})^n\right|_0\Big] f\Big)(l) \\
=&\Big(\overrightarrow{j(x)}\big(\mathcal{E}[x^n]f\big)\Big)(l)
-\big(\mathcal{E}\big[\nabla_{j(x)}(x^n)\big]f\big)(l)
.\end{align*}
The last equality above is a consequence of Lemma~\ref{Denmark}.
\end{proof}

\begin{proof}[Proof of Theorem~\ref{Madagascar}]
One easily observes that, by Lemma~\ref{Krakow},
\[ \sections{S(L/A)}\cong\DO(L/A,M)\xto{\exp^{\nabla,j}_*}
\DO(\groupoid{L}/\groupoid{A},M)\cong\frac{\enveloping{L}}{\enveloping{L}\sections{A}} \]
is the composition of $\mathcal{E}$ with the projection
$\enveloping{L}\to\frac{\enveloping{L}}{\enveloping{L}\sections{A}}$ the same way
$\exp^{\nabla,j}$ is the composition of $\bm{E}$ with the projection
$\bm{q}:\groupoid{L}\to\groupoid{L}/\groupoid{A}$.
Then Lemma~\ref{Oman} asserts that
\[ \exp^{\nabla,j}_*:\sections{S(L/A)}\to\frac{\enveloping{L}}{\enveloping{L}\sections{A}} \]
is precisely the map $\pbw^{\nabla,j}$ inductively defined in Theorem~\ref{Nairobi}.
\end{proof}

\begin{remark}
When $L=T_M$ and $A$ is the trivial Lie subalgebroid of $L$ of rank 0,
the $\pbw^{\nabla,j}$ map of Theorem~\ref{Madagascar}
is the inverse of the so called `normal order complete symbol map,'
which is an isomorphism from the space $\enveloping{T_M}$ of differential operators on $M$
to the space $\sections{S(T_M)}$ of fiberwise polynomial functions on $T^\vee_M$.
The complete symbol map was generalized to arbitrary Lie algebroids over $\RR$
by Nistor, Weinstein, and one of the authors~\cite{MR1687747}.
It played an important role in quantization theory~\cite{MR706215,MR1687747}.
\end{remark}

\section{Kapranov dg-manifolds}
\label{Yamoussoukro}

As an important application of the 
PBW map of Theorem~\ref{Madagascar} for a Lie pair $(L,A)$,
we construct, in Sections~\ref{Yamoussoukro} and~\ref{Abidjan},
a homological vector field on the graded manifold $A[1]\oplus L/A$
turning it into what we call a Kapranov dg-manifold.

\subsection{Definition}

Let $A$ be a Lie algebroid over a smooth manifold $M$
and $R$ be the algebra of smooth functions on $M$ valued in $\KK$.
The Chevalley--Eilenberg differential
\[ d_A:\sections{\Lambda^k A^\vee}\to\sections{\Lambda^{k+1} A^\vee} \]
defined by
\begin{multline*}
\label{eqCE}
\big(d_A \alpha\big)(a_0,a_1,\cdots,a_k)=
\sum_{i=0}^{k} (-1)^i \anchor(a_i)\big(\alpha(a_0,\cdots,\widehat{a_i},\cdots,a_k)\big) \\
+\sum_{i<j}(-1)^{i+j}\alpha(\lie{a_i}{a_j},a_0,\cdots,\widehat{a_i},\cdots,\widehat{a_j},\cdots,a_k)
\end{multline*}
and the wedge product
make $\bigoplus_{k\geqslant 0}\sections{\Lambda^k A^\vee}$
into a differential graded commutative $\KK$-algebra~\cite{MR2157566}.

The Chevalley--Eilenberg covariant differential
associated with a representation $\nablanatural$ of a Lie algebroid $A \to M$
of rank $n$ on a vector bundle $E\to M$ is the $\KK$-linear operator
\[ d_A^{\nablanatural}: \sections{\Lambda^k A^\vee\otimes E}
\to\sections{\Lambda^{k+1} A^\vee\otimes E} \]
that takes a section $\alpha\otimes e$ of $\Lambda^k A^\vee\otimes E$ to
\[ d_A^{\nablanatural}(\alpha\otimes e)=(d_A\alpha)\otimes e
+(-1)^{k}\sum_{j=1}^{n}(\alpha\wedge\beta_j)\otimes \nablanatural_{b_j}e ,\]
where $b_1,b_2,\dots,b_n$ and $\beta_1,\beta_2,\dots,\beta_n$
are any pair of dual local frames for the vector bundles $A$ and $A^\vee$.
The connection $\nablanatural$ being flat, $d_A^{\nablanatural}$ is a coboundary operator:
$d_A^{\nablanatural}\circ d_A^{\nablanatural}=0$.

Recall that a \emph{dg-manifold} is a $\ZZ$-graded manifold
endowed with a vector field $Q$ of degree $+1$ satisfying $\lie{Q}{Q}=0$.
Such a vector field $Q$ is said to be homological~\cite{MR1480150}.
In the literature, dg-manifolds are also called $Q$-manifolds~\cite{MR1432574,MR2062626}.
We refer the reader to~\cite{MR3293862,MR1958834} for details.
Below are several standard examples of dg-manifolds.

\begin{example}
Let $A\to M$ be a vector bundle.
According to Va{\u\i}ntrob~\cite{MR1480150},
a Lie algebroid structure on $A$
is equivalent to a dg-manifold structure on $A[1]$.
The homological vector field on $A[1]$ is the Chevalley--Eilenberg differential $d_{A}$.
\end{example}

\begin{example}[\cite{MR2163405}]
Suppose $\frakg=\bigoplus_{i\in \ZZ}\frakg_i$ is a finite-dimensional $\ZZ$-graded vector space.
Then a dg-manifold structure on $\frakg$ is equivalent
to a curved $L_\infty[1]$ algebra structure on $\frakg$.
\end{example}

We are now ready to introduce the main object of the paper.

\begin{definition}
A \emph{Kapranov dg-manifold} consists of a pair of vector bundles
$A$ and $E$ over a common base manifold
together with a homological vector field on each one of the graded manifolds $A[1]$ and $A[1]\oplus E$
such that both the inclusion $A[1]\into A[1]\oplus E$ and the projection $A[1]\oplus E\onto A[1]$
are morphisms of dg-manifolds.
\end{definition}

Note that the requirement that $A[1]$ be a dg-manifold implies that $A$ must be a Lie algebroid.

\begin{example}\label{Yellowknife}
Given a representation $\nablanatural$ of a Lie algebroid $A$ over a vector bundle $E$,
the Chevalley--Eilenberg differentials $d_A$ and $d_A^{\nablanatural}$ provide
a pair of homological vector fields on $A[1]$ and $A[1]\oplus E$
compatible with the inclusion and projection maps.
The resulting Kapranov dg-manifold $(A[1]\oplus E,d_A^{\nablanatural})$
is said to be \emph{linear}.
\end{example}

\begin{definition}
A \emph{morphism of Kapranov dg-manifolds} from $A[1]\oplus E$ to $A'[1]\oplus E'$
consists of a pair $(\Phi,\phi)$ of morphisms of dg-manifolds
$\Phi:A[1]\oplus E \to A'[1]\oplus E'$ and $\phi:A[1]\to A'[1]$
such that the following two diagrams commute:
\[ \begin{array}{cc}
\begin{tikzcd} A[1]\oplus E \arrow[r, "\Phi"] & A'[1]\oplus E' \\
A[1] \arrow[r, "\phi"'] \arrow[u, hook] & A'[1] \arrow[u, hook] \end{tikzcd} &
\begin{tikzcd} A[1]\oplus E \arrow[r, "\Phi"] \arrow[d, two heads] &
A'[1]\oplus E' \arrow[d, two heads] \\
A[1] \arrow[r, "\phi"'] & A'[1] .\end{tikzcd}
\end{array} \]
Such a morphism is said to \emph{fix the dg-submanifold $A[1]$}
if $A=A'$ and $\phi=\id_{A[1]}$, in which case
it is simply denoted by $\Phi$.
\newline
An \emph{isomorphism of Kapranov dg-manifolds} is a morphism $(\Phi,\phi)$ of Kapranov dg-manifolds
where both $\Phi$ and $\phi$ are dg-manifold isomorphisms.
\end{definition}

\begin{definition}
A Kapranov dg-manifold $(A[1]\oplus E,D)$ is said to be \emph{linearizable}
if there exists a representation $\nablanatural$ of the Lie algebroid $A$ on the vector bundle $E$
and an isomorphism of Kapranov dg-manifolds from $(A[1]\oplus E,D)$
to the linear Kapranov dg-manifold $(A[1]\oplus E,d_A^{\nablanatural})$
fixing the dg-submanifold $(A[1],d_A)$.
\end{definition}

\subsection{Alternative descriptions}

Given a vector bundle $E$ over the manifold $M$, recall
that deconcatenation (see Equation~\eqref{London})
defines a comultiplication on $\sections{S(E)}$ while
the symmetric tensor product
defines a multiplication on $\sections{S(E^\vee)}$.
Let $\mathfrak{e}$ denote the ideal of the graded algebra
$\sections{S(E^\vee)}$ generated by $\sections{E^\vee}$.
The algebra $\Hom_R\big(\sections{S(E)},R\big)$
dual to the $R$-coalgebra $\sections{S(E)}$
is the $\mathfrak{e}$-adic completion of $\sections{S(E^\vee)}$.
It will be denoted by $\sections{\hat{S}(E^\vee)}$.
Equivalently, one can think of the completion $\hat{S}(E^\vee)$ of $S(E^\vee)$
as a bundle of algebras over $M$.
Therefore we have a pairing
$\duality{\varepsilon}{e}\in R$ for any
$\varepsilon\in\sections{\hat{S}(E^\vee)}$
and $e\in \sections{S(E)}$, which can be obtained as follows.
Consider $e$ as a fiberwise polynomial on $E^\vee$ and $\varepsilon$
a fiberwise formal differential operator on $E^\vee$ with constant coefficients.
Then $\duality{\varepsilon}{e}$ is the function on $M$ obtained
by applying $\varepsilon$ on $e$ fiberwise.

\pagebreak[2]
\begin{remark}\label{Portugal}
It will be useful to keep the following obvious
facts in mind.
\begin{enumerate}
\item Let $\varepsilon\in\sections{\hat{S}(E^\vee)}$.
Then $\varepsilon\in\sections{\hat{S}^{\geqslant n}(E^\vee)}$
if and only if $\duality{\varepsilon}{e}=0$ for all $e\in\sections{S^{<n}(E)}$.
\item Let $e \in \sections{S(E)}$.
Then $e\in\sections{S^{\leqslant n}(E)}$ if and only if $\duality{\varepsilon}{e}=0$ for all
$\varepsilon\in\sections{\hat{S}^{>n}(E^\vee)}$.
\end{enumerate}
\end{remark}

Let $\nablanatural$ be a representation of $A$ on $E$.
Consider the space $\XXfv(E)$ of formal vertical vector fields on $E$ along the zero section,
which are, by definition, $R$-linear derivations of the algebra
of fiberwise formal functions on $E$ along the zero section.
Since the latter can be identified with $\sections{\hat{S}(E^\vee)}$,
the space $\XXfv(E)$ is naturally identified with $\sections{\hat{S}(E^\vee)\otimes E}$.
The Lie algebra structure on $\XXfv(E)$
extends to a graded Lie algebra structure on $\Gamma(\Lambda A^\vee)\otimes_R\XXfv(E)$,
which is isomorphic to $\bigoplus_{n\geqslant 0}\sections{\Lambda^n(A^\vee)
\otimes\hat{S}(E^\vee)\otimes E}$.
Since $E$ is an $A$-module, so is $\hat{S}(E^\vee)\otimes E$.
Therefore $\bigoplus_{n\geqslant 0}\sections{\Lambda^n(A^\vee)
\otimes\hat{S}(E^\vee)\otimes E}$ is equipped
with the corresponding Chevalley--Eilenberg differential
\begin{equation}\label{Rwanda}
d_A^{\nablanatural}:\sections{\Lambda^\bullet(A^\vee)\otimes\hat{S}(E^\vee)\otimes E} \to
\sections{\Lambda^{\bullet+1}(A^\vee)\otimes\hat{S}(E^\vee)\otimes E}.
\end{equation}

\begin{lemma}\label{Sweden}
Given a representation $\nablanatural$ of $A$ on $E$, the graded vector space
\[ \bigoplus_{n\geqslant 0}\sections{\Lambda^n(A^\vee)\otimes\hat{S}(E^\vee)\otimes E} ,\]
together with the differential \eqref{Rwanda} and the graded Lie bracket described above,
is a differential graded Lie algebra.
\end{lemma}

Here is the main result of this section.

\begin{proposition}
\label{Tasmania}
Given a Lie algebroid $A$ and a vector bundle $E$ over $M$,
the following supplemental data are equivalent.
\begin{enumerate}
\item\label{uno} A homological vector field $D$ on
$A[1]\oplus E$ making it a Kapranov dg-manifold.
\item\label{quattro} A degree $+1$ derivation $D$ of the graded algebra
$\bigoplus_{n\geqslant 0}\sections{\Lambda^n A^\vee\otimes\hat{S}(E^\vee)}$,
which preserves the filtration
\begin{equation}\label{Uruguay}
\cdots\into\sections{\Lambda A^\vee\otimes\hat{S}^{\geqslant 2}(E^\vee)}
\into\sections{\Lambda A^\vee\otimes\hat{S}^{\geqslant 1}(E^\vee)}\into
\sections{\Lambda A^\vee\otimes\hat{S}(E^\vee)}
\end{equation}
and satisfies $D^2=0$ and $D(\alpha\otimes 1)=d_A(\alpha)\otimes 1$,
for all $\alpha\in\sections{\Lambda A^\vee}$.
\item\label{due} An $A$-action $\delta:\sections{A}\times\sections{S(E)}\to
\sections{S(E)}$
by coderivations such
that $\delta_a(1)=0$, for all $a\in\sections{A}$.
\item\label{tre} An $A$-action
$\delta^*:\sections{A}\times\sections{\hat{S}(E^\vee)}
\to\sections{\hat{S}(E^\vee)}$
by derivations such that
$\sections{\hat{S}^{\geqslant 1}(E^\vee)}$ is stable under
$\delta^*_a$, for all $a\in\sections{A}$.
\item\label{cinque} A representation
$\nablanatural:\sections{A}\times\sections{E}\to\sections{E}$ of $A$ on $E$
together with a solution $\cR\in\sections{A^\vee\otimes\hat{S}^{\geqslant 2}(E^\vee)\otimes E}$
of the Maurer--Cartan equation
\begin{equation}\label{eq:MC}
d_A^{\nablanatural}\cR+\tfrac{1}{2}\lie{\cR}{\cR}=0
\end{equation}
of the dgla
$\bigoplus_{n\geqslant 0}\sections{\Lambda^n A^\vee\otimes\hat{S}(E^\vee)\otimes E}$
of Lemma~\ref{Sweden}.
\item\label{sei}
An $L_\infty[1]$ algebra structure on the graded vector space
$ \bigoplus_{n\geqslant 0}\sections{\Lambda^n A^\vee\otimes E} $
defined by a sequence $(\lambda_k)_{k\in\NN}$ of multibrackets
\[ \lambda_k:S^k\big(\sections{\Lambda A^\vee\otimes E}\big)
\to\sections{\Lambda A^\vee\otimes E}[1] \]
such that each $\lambda_k$, with $k\geqslant 2$,
is $\sections{\Lambda A^\vee}$-multilinear, and $\lambda_1$
is the Chevalley--Eilenberg differential $d_A^{\nabla}$
associated with a representation $\nabla$ of $A$
on $E$.
\end{enumerate}
The operators $D$, $\delta$, $\delta^*$,
the multibrackets $(\lambda_k)_{k\in\NN}$,
the representation $\nablanatural$,
and the Maurer--Cartan element $\cR$
are related by Equations~\eqref{Poland} to~\eqref{Switzerland} below.
\end{proposition}

\begin{proof}
\textbf{(\ref{uno})$\Rightarrow$(\ref{quattro})}\quad
The projection $A[1]\oplus E\onto A[1]$ is
a morphism of dg-manifolds if and only if
\[ D(\alpha\otimes 1)=d_A \alpha \otimes 1, \quad\forall\alpha\in\sections{\Lambda A^\vee} .\]
The inclusion $A[1]\into A[1]\oplus E$ is a morphism
of dg-manifolds if and only if
\begin{equation}\label{eq:meaningInj} \pi\circ D=d_A\circ\pi ,\end{equation} where
$\pi:\sections{\Lambda A^\vee\otimes \hat{S}(E^\vee)}
\to \sections{\Lambda A^\vee}$ is the map
\[ \pi(\alpha\otimes\sigma)= \begin{cases}
\sigma \alpha & \text{if }\sigma\in\sections{S^0(E^\vee)}=R \\
0 & \text{if } \sigma\in\sections{\hat{S}^{\geqslant 1}(E^\vee)}
.\end{cases} \]
Therefore, for every $\alpha\in\sections{\Lambda A^\vee}$
and $\varepsilon\in\sections{E^\vee}$, we have
\[ \pi\big(D(\alpha\otimes\varepsilon)\big)
=d_A\big(\pi(\alpha\otimes\varepsilon)\big)
=d_A(0)=0 .\] Hence,
$D(\alpha\otimes\varepsilon)\in\sections{\Lambda A^\vee
\otimes \hat{S}^{\geqslant 1}(E^\vee)}$.
By derivation property, $D$ preserves
the filtration~\eqref{Uruguay}.

\textbf{(\ref{quattro})$\Rightarrow$(\ref{uno})}\quad
By assumption, $D$ is a homological vector field on $A[1] \oplus E$ such that the
projection $A[1]\oplus E\onto A[1]$ is
a morphism of dg-manifolds.
To prove that the inclusion $A[1]\into A[1]\oplus E$ is a morphism
of dg-manifolds, it suffices to prove that
Equation~\eqref{eq:meaningInj} holds.
For all $\alpha\in\sections{\Lambda A^\vee}$ and
$\sigma\in\sections{\hat{S}^{\geqslant 1}(E^\vee)}$, we have
$D(\alpha\otimes\sigma)\in\sections{
\Lambda A^\vee\otimes\hat{S}^{\geqslant 1}(E^\vee)}$
since $D$ preserves the filtration~\eqref{Uruguay}.
Therefore,
$ \pi\big(D(\alpha\otimes\sigma)\big)=0=
d_A\big(\pi(\alpha\otimes\sigma)\big) $.
Moreover, for all
$f\in\sections{S^0(E^\vee)}=R$, we have
\[ \pi\big(D(\alpha\otimes f)\big)
=\pi\big(D(f\alpha\otimes 1)\big)
=\pi\big(d_A(f\alpha)\otimes 1\big)
=d_A\big(\pi(\alpha\otimes f)\big) .\]
Hence Equation~\eqref{eq:meaningInj} holds.

\textbf{(\ref{due})$\Leftrightarrow$(\ref{tre})}\quad
The $A$-actions $\delta$ and $\delta^*$ determine each other through the relation
\begin{equation}
\label{Poland}
\anchor(a)\duality{\varepsilon}{e}=\duality{\delta^*_a(\varepsilon)}{e}+
\duality{\varepsilon}{\delta_a(e)} ,\quad\forall a\in\sections{A},
\varepsilon\in\sections{\hat{S}(E^\vee)}, e\in\sections{S(E)} .
\end{equation}
Obviously, $\delta_a$ is a coderivation if and only if $\delta^*_a$ is
a derivation.
Assume that $\delta_a$ is a coderivation.
According to Proposition~\ref{shark},
$\delta_a(1)=0$
if and only if $\sections{S^{< k}(E)}$ is $\delta_a$-stable
for all $k\in\NN$.
As stated in Remark~\ref{Portugal},
$\sections{\hat{S}^{\geqslant k}(E^\vee)}$ is the annihilator of
$\sections{S^{< k}(E)}$.
Hence $\sections{S^{< k}(E)}$ is $\delta_a$-stable
if and only if
$\sections{\hat{S}^{\geqslant k}(E^\vee)}$ is $\delta_a^*$-stable.
Since $\delta^*_a$ is a derivation, the latter is equivalent to $\delta^*_a(\varepsilon)\in\sections{
\hat{S}^{\geqslant 1}(E^\vee)}$
for all
$\varepsilon\in\sections{E^\vee}$.

\textbf{(\ref{tre})$\Rightarrow$(\ref{quattro})}\quad
Given an $A$-action
$\delta^*:\sections{A}\times\sections{\hat{S}(E^\vee)}\to\sections{\hat{S}(E^\vee)}$,
set \[ D(\alpha\otimes\sigma)=(d_A\alpha)\otimes\sigma
+\sum_j(\nu_j\wedge\alpha)\otimes(\delta^*_{v_j}\sigma) \]
for all $\alpha\in\sections{\Lambda A^\vee}$ and $\sigma\in\sections{\hat{S}(E^\vee)}$.
Here $v_1,\dots,v_l$ and $\nu_1,\dots,\nu_l$ are any pair
of dual local frames of the vector bundles $A$ and $A^\vee$.
The flatness of the action $\delta^*$ implies that $D^2=0$,
while the other properties of $D$ follow immediately from those of $\delta^*$.

\textbf{(\ref{quattro})$\Rightarrow$(\ref{tre})}\quad
Given $D$, the relation
\begin{equation}\label{Mali} \delta^*_a(\sigma)=i_a D(1\otimes\sigma), \quad\forall a\in\sections{A},
\forall\sigma\in\sections{\hat{S}(E^\vee)} \end{equation}
defines a map $\delta^*:\sections{A}\times\sections{\hat{S}(E^\vee)}\to\sections{\hat{S}(E^\vee)}$.
It is indeed an action, since
\[ \begin{split}
\delta^*_a (f\sigma) =& i_a D(1\otimes f\sigma) \\
=& i_a \big\{ D(f\otimes 1)\cdot(1\otimes\sigma)
+ (f\otimes 1)\cdot D(1\otimes\sigma) \big\} \\
=& i_a \Big\{ (d_A f)\otimes\sigma + f D(1\otimes\sigma) \Big\} \\
=&( \anchor(a)f) \,\sigma +f\delta^*_a(\sigma)
,\end{split} \]
for all $a\in\sections{A}$, $f\in C^\infty(M)$, and $\sigma\in\sections{\hat{S}(E^\vee)}$.

\textbf{(\ref{quattro})$\Rightarrow$(\ref{cinque})}\quad
Given $D$, set
\begin{equation}\label{Romania} \nablanatural_a \varepsilon
= \varpi\big(i_a D(1\otimes\varepsilon)\big) ,\quad \forall a\in\sections{A},
\forall\varepsilon\in\sections{E^\vee}, \end{equation}
where $\varpi$ denotes the canonical projection $\varpi:\hat{S}(E^\vee)\to S^1(E^\vee)$.

For all $f\in C^\infty(M)$, we have
\[ \begin{split}
D(1\otimes f\varepsilon)
=& D\big((f\otimes 1)\cdot(1\otimes\varepsilon)\big) \\
=& \big((d_A f)\otimes 1\big)\cdot (1\otimes\varepsilon)
+f D(1\otimes\varepsilon) \\
=& (d_A f)\otimes\varepsilon + f D(1\otimes\varepsilon)
.\end{split} \]
It follows that
\[ i_a D(1\otimes f\varepsilon)=(\anchor(a)f)\otimes\varepsilon
+f i_a D(1\otimes\varepsilon). \]
Applying $\varpi$ to both sides, we obtain
\[ \nablanatural_a(f\varepsilon)=(\anchor(a)f) \varepsilon
+f \nablanatural_a\varepsilon .\]
Thus the map $\nablanatural$ is an $A$-connection on $E^\vee$.
Since $D^2=0$ and $D$ preserves the filtration~\eqref{Uruguay}, we have
\[ d_A^{\nablanatural}(d_A^{\nablanatural}\varepsilon)
= \varpi\bigg(D\Big(\varpi\big(D
(1\otimes\varepsilon)\big)\Big)\bigg)
=\varpi\big(D^2(1\otimes\varepsilon)\big)
=0 \] for all $\varepsilon\in\sections{E^\vee}$.
Therefore, the connection $\nablanatural$ is flat.

We use the same symbol $\nablanatural$
to denote the induced flat connection on the dual bundle $E$.
It follows immediately from the identity
\[ \anchor(a)\duality{\varepsilon}{e}=
\duality{\nablanatural_a\varepsilon}{e}
+\duality{\varepsilon}{\nablanatural_a e}
,\quad\forall a\in\sections{A},\forall e\in\sections{E},
\forall\varepsilon\in\sections{E^\vee},\]
together with Equations~\eqref{Poland} to~\eqref{Romania} that
the diagram
\[ \begin{tikzcd}
\sections{S(E)} \arrow[r, "\delta_a"] &
\sections{S(E)} \arrow[d, two heads, "\varpi"] \\
\sections{E} \arrow[u, hook] \arrow[r, "\nabla_a"'] & \sections{E}
\end{tikzcd} \]
commutes.

Since both $D$ and $d_A^{\nablanatural}$ are derivations of
the algebra $\sections{\Lambda A^\vee\otimes \hat{S}(E^\vee)}$
and \[ D(\alpha\otimes 1)=(d_A\alpha)\otimes 1=
d_A^{\nablanatural}(\alpha\otimes 1) ,\]
it follows that
\[ \big(D-d_A^{\nablanatural}\big)(\alpha\otimes\sigma)
= (-1)^{n} (\alpha\otimes 1)
\cdot \big(D-d_A^{\nablanatural}\big)(1\otimes\sigma) \]
for all $\alpha\in\sections{\Lambda^n A^\vee}$ and
$\sigma\in\sections{\hat{S}(E^\vee)}$.
Therefore, since the derivation $D-d_A^{\nablanatural}$ satisfies
\[ \big(D-d_A^{\nablanatural}\big)(1\otimes\sigma)
\in\sections{A^\vee\otimes \hat{S}^{\geqslant k+1}(E^\vee)} \]
for all $\sigma\in\sections{S^k(E^\vee)}$, there exists
$\cR\in\sections{A^\vee\otimes\hat{S}^{\geqslant 2}(E^\vee)\otimes E}$ such that
\[ \big(D-d_A^{\nablanatural}\big)(\alpha\otimes\sigma)
= \sum_r (\nu_r\wedge\alpha)\otimes i_{v_r}\cR(\sigma) .\]
Here $v_1,\dots,v_l$ and $\nu_1,\dots,\nu_l$ are any pair
of dual local frames of the vector bundles $A$ and $A^\vee$,
and $i_{v_r}\cR$ is seen as a derivation of
the algebra $\sections{\hat{S}(E^\vee)}$.
In other words:
\begin{equation}\label{Scandinavia} \cR=D-d_A^{\nablanatural} .\end{equation}
Since $D^2=0$ and $(d_A^{\nablanatural})^2=0$, the Maurer-Cartan equation \eqref{eq:MC} follows.

\textbf{(\ref{cinque})$\Rightarrow$(\ref{quattro})}\quad
Given a representation $\nablanatural$ of $A$ on $E$ and
a solution $\cR$ of the Maurer--Cartan equation \eqref{eq:MC},
the derivation $D=d_A^{\nablanatural}+\cR$ satisfies all
requirements.

\textbf{(\ref{cinque})$\Rightarrow$(\ref{sei})}\quad
Let $\Pi:\frakh\to\fraka$ denote the natural projection
of the graded Lie algebra
\[ \frakh=\bigoplus_{n\geqslant 0}\sections{\Lambda^n A^\vee\otimes\hat{S}(E^\vee)\otimes E} \]
of formal vertical vector fields on $A[1]\oplus E$ onto
its abelian Lie subalgebra
\[ \fraka=\bigoplus_{n\geqslant 0}\sections{\Lambda^n A^\vee\otimes E} .\]
Since
\[ \Pi\lie{x}{y}=\Pi\lie{\Pi x}{y}+\Pi\lie{x}{\Pi y}
,\quad\forall x,y\in\frakh ,\]
the triple $(\frakh,\fraka,\Pi)$ is a \emph{V-algebra}
in the sense of Cattaneo \& Schätz~\cite[Definition~2.3]{MR2440258}.
It is easy to check that the degree $+1$ derivation $d_A^{\nablanatural}+\ad_{\cR}$
is an \emph{adapted derivation} \cite[Definition~2.4]{MR2440258}.
According to~\cite[Theorem~2.5]{MR2440258} (see also \cite{MR2163405}),
the graded vector space $\fraka=\sections{\Lambda A^\vee\otimes E}$
is an $L_\infty[1]$ algebra whose multibrackets are given by
\[ \lambda_k(x_1,\cdots,x_k)=\Pi\lie{\lie{\cdots\lie{\lie{(d_A^{\nablanatural}
+\ad_{\cR}\big)x_1}{x_2}}{x_3}\cdots}{x_{k-1}}}{x_k} .\]

It follows from a straightforward computation that
$\lambda_1=d_A^\nabla$,
and each $\lambda_k$ with $k\geqslant 2$ is $\sections{\Lambda(A^\vee)}$-multilinear
and given by:
\begin{equation}\label{Switzerland}
\lambda_k(\xi_1\otimes e_1,\cdots,\xi_k\otimes e_k)
=(-1)^{\abs{\xi_1}+\cdots+\abs{\xi_k}}
\xi_1\wedge\cdots\wedge\xi_k\wedge\cR_k(e_1,\cdots,e_k)
,\end{equation}
for all $e_1,\dots,e_k\in\sections{E}$ and all homogeneous elements
$\xi_1,\dots,\xi_k$ of $\sections{\Lambda A^\vee}$.
Here $\cR_k$ denotes the component of $\cR$ in
$\sections{A^\vee\otimes S^k(E^\vee)\otimes E}$.

\textbf{(\ref{sei})$\Rightarrow$(\ref{cinque})}\quad
For $k\geqslant 2$, the multibracket $\lambda_k$ is $\sections{\Lambda A^\vee}$-multilinear
and hence satisfies Equation~\eqref{Switzerland} for some
$\cR_k\in\sections{A^\vee\otimes S^k(E^\vee)\otimes E}$.
It follows from the generalized Jacobi identity that the fiberwise vertical vector field
$\cR:=\sum_{k\geqslant 2}\cR_k\in\frakh$ satisfies
\[ \Pi\lie{\cdots\lie{\lie{d_A^{\nablanatural}+\ad_\cR}
{d_A^{\nablanatural}+\ad_\cR}(e_1)}{e_2}\cdots}{e_n}=0
,\quad\forall e_1,\cdots,e_n\in\sections{E} .\]
Therefore, we have
\[ 0=\lie{d_A^{\nablanatural}+\ad_\cR}{d_A^{\nablanatural}+\ad_\cR}
=2 \ad_{d_A^{\nablanatural}\cR+\frac{1}{2}\lie{\cR}{\cR}} .\]
It thus follows that $d_A^{\nablanatural}\cR+\frac{1}{2}\lie{\cR}{\cR}=0$.
\end{proof}

As a consequence, we have the following:

\begin{proposition}\label{prop:AtiyahDg}
Each Kapranov dg-manifold $(A[1]\oplus E,D)$ determines
a cohomology class in $H_{\CE}^1(A;S^2(E^\vee)\otimes E)$.
\end{proposition}

\begin{proof}
Given a Kapranov dg-manifold $(A[1]\oplus E,D)$, 
let $\nablanatural$ be the corresponding representation 
of the Lie algebroid $A$ on the vector bundle $E$,
and let $\cR=\sum_{k\geqslant 2}\cR_k$ be the corresponding 
Maurer--Cartan element in
$\sections{A^\vee\otimes S^{\geqslant 2}(E^\vee)\otimes E}$ 
--- see Proposition~\ref{Tasmania}.
Projecting the Maurer--Cartan equation~\eqref{eq:MC} onto the component
$\sections{\Lambda^2 A^\vee\otimes S^2(E^\vee)\otimes E}$,
we obtain $d_A^{\nablanatural}\cR_2=0$.
Thus $\cR_2\in\sections{A^\vee\otimes S^2(E^\vee)\otimes E}$
is a Chevalley--Eilenberg $1$-cocycle for the representation of $A$ on $S^2(E^\vee)\otimes E$
induced by the representation $\nablanatural$.
Hence the Kapranov dg-manifold $(A[1]\oplus E,D)$ yields a
cohomology class $[\cR_2]\in H_{\CE}^1(A;S^2(E^\vee)\otimes E)$.
\end{proof}

We will say that $[\cR_2]\in H_{\CE}^1(A;S^2(E^\vee)\otimes E)$
is the \emph{characteristic class} of the Kapranov dg-manifold
$(A[1]\oplus E,D=d_A^{\nablanatural}+\cR)$.
Its meaning will become apparent in Corollary~\ref{Cantonese}.

Next, we give a characterization of Kapranov dg-manifold isomorphisms.

\begin{proposition}\label{Tagalog}
Let $(A[1]\oplus E,D^\sharp)$ and $(A[1]\oplus E,D^\flat)$ 
be two Kapranov dg-manifolds
inducing the same dg-manifold structure on $A[1]$.
The following three types of morphisms are equivalent:
\begin{enumerate}
\item Isomorphisms of Kapranov dg-manifolds
from $(A[1]\oplus E,D^\sharp)$ to $(A[1]\oplus E,D^\flat)$
fixing the dg-submanifold $A[1]$.
\item Automorphisms of the $R$-coalgebra $\sections{S(E)}$
which respect the filtration
\[ R=\sections{S^{\leqslant 0}(E)} \into \sections{S^{\leqslant 1}(E)} \into
\sections{S^{\leqslant 2}(E)} \into \sections{S^{\leqslant 3}(E)} \into \cdots \]
and intertwine the $A$-actions $\delta_\sharp$ and $\delta_\flat$
corresponding to $D^\sharp$ and $D^\flat$ respectively.
\item Automorphisms of the $R$-algebra $\sections{\hat{S}(E^\vee)}$
which respect the filtration
\[ \cdots\into\sections{\hat{S}^{\geqslant 3}(E^\vee)}
\into\sections{\hat{S}^{\geqslant 2}(E^\vee)}
\into\sections{\hat{S}^{\geqslant 1}(E^\vee)}
\into\sections{\hat{S}(E^\vee)} \]
and intertwine the $A$-actions $\delta_\sharp^*$ and ${\delta_\flat^*}$
corresponding to $D^\sharp$ and $D^\flat$ respectively.
\end{enumerate}
\end{proposition}

\begin{proof}
Observe that every automorphism of the graded manifold $A[1]\oplus E$
such that the diagrams
\[ \begin{tikzcd} A[1]\oplus E \arrow[r, "\Phi"] & A[1]\oplus E \\
A[1] \arrow[r, "\id"] \arrow[u, hook] & A[1] \arrow[u, hook] \end{tikzcd}
\qquad\text{and}\qquad
\begin{tikzcd} A[1]\oplus E \arrow[r, "\Phi"] \arrow[d, two heads] &
A[1]\oplus E \arrow[d, two heads] \\
A[1] \arrow[r, "\id"] & A[1] \end{tikzcd} \]
commute must be of the form
$\alpha\otimes\sigma\mapsto\alpha\otimes\Phi(\sigma)$,
for all $\alpha\in\sections{\wedge A^\vee}$ and $\sigma\in\sections{\hat{S}(E^\vee)}$.
Here $\Phi$ is an automorphism of the $R$-algebra $\sections{\hat{S}(E^\vee)}$
preserving the filtration as in the third assertion.
Such automorphisms of $\sections{\hat{S}(E^\vee)}$ are in one-to-one correspondence
with the automorphisms of the $R$-coalgebra $\sections{{S}(E)}$
preserving the filtration as in the second assertion.
The conclusion follows immediately.
\end{proof} 	

Note that an automorphism $\Phi$ of the filtered $R$-coalgebra $\sections{{S}(E)}$
is determined by a sequence $(\phi_k)_{k\in\NN}$ of bundle maps
$\phi_k: S^k(E) \to E$ such that $\phi_1$ is invertible.
The map $\phi_k$ is called the $k$-th Taylor coefficient of $\Phi$.
Abusing notations, we write $\Phi=(\phi_k)_{k\in\NN}$.
As a consequence of Proposition~\ref{Tagalog},
every Kapranov dg-manifold isomorphism from $(A[1]\oplus E,D^\sharp)$ 
to $(A[1]\oplus E,D^\flat)$ fixing the dg-submanifold $A[1]$ 
is determined by such a sequence of Taylor coefficients $(\phi_k)_{k\in\NN}$.

\begin{proposition}\label{prop:class_well_defined}
Let $\Phi=(\phi_k)_{k\in\NN}$ be an isomorphism of Kapranov dg-manifolds
from $(A[1]\oplus E,D^\sharp=d_A^{\nabla^\sharp}+\sum_{k\geqslant 2}\cR^\sharp_k)$
to $(A[1]\oplus E,D^\flat=d_A^{\nabla^\flat}+\sum_{k\geqslant 2}\cR^\flat_k)$
fixing the dg-submanifold $A[1]$.
Then, the first Taylor coefficient $\phi_1$ of $\Phi$ is necessarily
an isomorphism of $A$-modules from $(E,\nabla^\sharp)$ to $(E,\nabla^\flat)$.
Moreover, the induced isomorphism in Chevalley--Eilenberg cohomology
identifies the characteristic classes $[\cR_2^\sharp]$ and $[\cR_2^\flat]$.
\end{proposition}

\begin{proof}
It follows from Equation~\eqref{Romania} that $\phi_1$ is an isomorphism of $A$-modules.
We may thus assume without loss of generality that $\phi_1=\id$
and $\nabla^\flat=\nabla^\sharp=\nabla$.
Therefore, it suffices to prove that $[\cR^\sharp_2]=[\cR^\flat_2]$
in $H_{\CE}^1(A;S^2(E^\vee)\otimes E)$.
For all $\varepsilon\in\sections{E^\vee}$, we have
\[ (\id_{\sections{A^\vee}}\otimes\Phi^\vee)\Big(\big(d_A^\nabla+\cR^\sharp_2
+\sum_{k\geqslant 3}\cR^\sharp_k\big)(\varepsilon)\Big)
=(d^{\nabla}_A+\cR^\flat_2+\sum_{k\geqslant 3}\cR^\flat_k)\big(\Phi^\vee(\varepsilon)\big) \]
in $\sections{A^\vee\otimes\hat{S}(E^\vee)}$.
Projecting onto $\sections{A^\vee\otimes S^2(E^\vee)}$, we obtain
\[ \phi^\vee _2(d_A^\nabla(\varepsilon))+\cR^\sharp_2(\varepsilon)
=d_{A}^{\nabla}(\phi^\vee_2(\varepsilon))+\cR^\flat_2(\varepsilon) \]
and therefore $\cR^\sharp_2-\cR^\flat_2=d_A^\nabla(\phi^\vee_2)$,
where $\phi^\vee_2$ is now regarded as
a section of $S^2(E^\vee)\otimes E$.
The conclusion thus follows.
\end{proof}

\begin{corollary}\label{Cantonese}
If a Kapranov dg-manifold $A[1]\oplus E$ is linearizable,
then its characteristic class vanishes.
\end{corollary}

We note that the converse may not be true.
However, as we will see in Theorem~\ref{Montana},
for those Kapranov dg-manifolds stemming from Lie pairs,
the converse of Corollary~\ref{Cantonese} is indeed true,
and the characteristic class becomes the Atiyah class of the Lie pair.

\subsection{Kapranov dg-manifolds arising from Lie algebroid actions}
\label{Patagonia}

\subsubsection{The general case}

In this section, we describe a simple construction of Kapranov dg-manifolds.
Given a surjective submersion $\pi:X\to M$,
let $R$ denote the algebra of smooth functions on $M$ and
let $\XX_{\pi}(X)$ denote the space of all vector fields on $X$ that are $\pi$-projectable.
Recall \cite[Definition~2.3]{MR1037400}
that an \emph{action of a Lie algebroid $A\to M$ on $\pi:X\to M$} is a map
$\sections{A}\ni a\to a^\diamond\in\XX_{\pi}(X)$
that is both a morphism of Lie algebras
and a morphism of $R$-modules
and makes the following diagram commute:
\[ \begin{tikzcd}[row sep=tiny]
& \XX_{\pi}(X) \arrow[dd, "\pi_*"] \\
\sections{A} \arrow[ru, "\diamond"] \arrow[rd, "\anchor"'] & \\
& \XX(M)
.\end{tikzcd} \]
Assume that $\sections{A}\xto{\diamond}\XX_{\pi}(X)$
is an action of a Lie algebroid $A\to M$
on a surjective submersion $\pi:X\to M$
admitting a global section $\sigma:M\to X$,
which is preserved under the action,
i.e.\ $a^\diamond$ is tangent to $\sigma(M)$ for all $a\in\sections{A}$.
For any $a\in\Gamma(A)$, 
the derivation $a^\diamond$ of $C^\infty(X)$ determines
a coderivation $\delta_a$ of the coalgebra
$\DO(X,\sigma(M))$ of distributions on $X$
with support $\sigma(M)$.
Thus we obtain an $A$-action on
$\DO(X,\sigma(M))$ by coderivations,
which we call the \emph{$A$-action on
$\DO(X,\sigma(M))$ induced by $\diamond$}.

Let $N_\sigma\to M$ denote the normal bundle to the submanifold $\sigma(M)$ inside $X$.
Then one can construct a Kapranov dg-manifold structure on $A[1]\oplus N_\sigma$ as follows.

According to the tubular neighborhood theorem,
there exists a fiber-preserving diffeomorphism $\psi:N_\sigma\to X$
from a tubular neighborhood of the zero section of $N_\sigma\to M$
to a tubular neighborhood of the submanifold $\sigma(M)$ inside $X$ such that
\[ \begin{tikzcd}
N_\sigma \arrow[r, "\psi"] \arrow[d] & X \arrow[d, "\pi"'] \\
M \arrow[r, "\id"'] \arrow[u, bend left, "0"] & M \arrow[u, bend right, "\sigma"']
\end{tikzcd} \]
commutes.
The $\infty$-jets of smooth functions on $X$ along $\sigma(M)$ form an $R$-algebra,
which can be identified with $\sections{\hat{S}(N_\sigma^\vee)}$ via the diffeomorphism $\psi$.
Likewise, the $R$-coalgebra $\DO(X,\sigma(M))$ of distributions on $X$ with support $\sigma(M)$
can be identified with $\sections{S(N_\sigma)}$.

For every $a\in\Gamma(A)$,
since $a^\diamond$ is tangent to $\sigma(M)$,
the derivation $a^\diamond$ of $C^\infty(X)$ stabilizes the ideal of functions vanishing on $\sigma(M)$
and thus induces a derivation $\delta^*_a$ of $\sections{\hat{S}(N_\sigma^\vee)}$,
which stabilizes $\sections{\hat{S}^{\geqslant 1}(N_\sigma^\vee)}$.
The derivation $\delta^*_a$ is the dual map of a coderivation
$\delta_a$ of the symmetric $R$-coalgebra $\sections{S(N_\sigma)}$
satisfying $\delta_a(1)=0$. Therefore, according to Proposition~\ref{Tasmania},
$A[1]\oplus N_\sigma$ is a Kapranov dg manifold.
Furthermore, the Kapranov dg-manifold structures on $A[1]\oplus N_\sigma$
resulting from different choices of $\psi$ are all isomorphic according to Proposition~\ref{Tagalog}.
We summarize the above discussion in the following:

\begin{proposition}
\label{prop:KapranovOfLieAlgAction}
Every action of a Lie algebroid $A$
on a surjective submersion $\pi:X\to M$ preserving a section $\sigma:M\to X$
gives rise to a Kapranov dg-manifold structure
--- unique up to isomorphism ---
on the graded manifold $A[1]\oplus N_\sigma$.
\end{proposition}

\begin{remark}\label{garedeLyon}
It follows from Propositions~\ref{prop:KapranovOfLieAlgAction},
\ref{prop:AtiyahDg} and~\ref{prop:class_well_defined} that
every action of a Lie algebroid $A$ on a surjective submersion $\pi:X\to M$
preserving a section $\sigma:M\to X$ determines a cohomology class
in $H^1(A,S^2(N_\sigma^\vee)\otimes N_\sigma)$,
which we call its \emph{characteristic class}.
\end{remark}

\subsubsection{Module over a Lie algebroid}

Now let $E$ be an $A$-module.
It is clear that the normal bundle $N_0$ to the zero section inside the vector bundle $\pi:E\to M$
is canonically isomorphic to $E$ as a vector bundle.
Therefore, the coalgebra $D(E,0(M))$ of distributions on $E$
with support the zero section of $E$ is canonically identified to $\sections{S(E)}$

As explained in Section~\ref{SantaFe}, the representation
$\nabla:\sections{A}\times\sections{E}\to\sections{E}$
determines, by way of the corresponding
horizontal lifting map
$h:A\times_M E\to T_E$,
an $A$-action
\[ \begin{tikzcd}[row sep=tiny]
& \XX_{\pi}(E) \arrow[dd, "\pi_*"] \\
\sections{A} \arrow[ru, "\diamond"] \arrow[rd, "\anchor"'] & \\
& \XX(M)
\end{tikzcd} \]
on $E$, which preserves the zero section of $E$.
For all $m\in M$; $a\in\sections{A}$; and $e\in\sections{E}$, the tangent vector
$e_*\big(\anchor(a_m)\big)-h(a_m,e_m)$ is `vertical'
(i.e.\ belongs to the subspace $\ker(\pi_{*e_m})$
of $(T_E)_{e_m}$) and the canonical identification
of $\ker(\pi_{*e_m})$ with $E_m$ maps it to
$(\nabla_a e)_m$ --- see Equation~\eqref{XiAn}.
Furthermore, for all $a\in\sections{A}$
and $e\in E$, we have $a^\diamond_e=h(a_{\pi(e)},e)$
--- see Equation~\eqref{gazpacho}.

Let $\delta$ denote the $A$-action
by coderivations on $D(E,0(M))=\sections{S(E)}$
induced by $\sections{A}\xto{\diamond}\XX_{\pi}(E)$.

\begin{proposition}
Let $\pi:E\to M$ be an $A$-module
with representation \[ \nabla:\sections{A}\times\sections{E}\to\sections{E} .\]
Then the induced $A$-action by coderivations on the coalgebra
$D(E,0(M))$ of distributions on $E$ supported on the zero section
coincides with the representation
\[ \delta:\sections{A}\times\sections{S(E)}\to\sections{S(E)} \]
of $A$ on $S(E)$ defined by
\[ \delta_a(e_1\odot\cdots\odot e_n)
=\sum_{k=1}^n e_1\odot\cdots\odot e_{k-1}\odot
\nabla_a e_k\odot e_{k+1}\odot\cdots\odot e_n
,\]
for all $a\in\sections{A}$ and
$e_1,e_2,\dots,e_n\in\sections{E}$.
\end{proposition}

\begin{proof}

Denoting by $\lambda_\varepsilon$ the fiberwise
linear function on $E$ associated with a section
$\varepsilon$ of the dual vector bundle
$E^\vee\to M$ in the natural way, we obtain
\[ \duality{(d\lambda_\varepsilon)\circ e}{e_*(\anchor_a)-h(a,e)}
=\duality{\varepsilon}{\nabla_a e}, \]
where the left hand side is the duality pairing
between a section of $e^*(T_E)$ --- the pullback
of the tangent bundle $T_E\to E$ through the
map $e:M\to E$ --- and a section of the dual
bundle $e^*(T^\vee_E)$, while the right hand side
is the duality pairing of a section of $E$
with a section of $E^\vee$.
The last equation can be rewritten as
\[ \anchor_a \duality{\varepsilon}{e}
-a^\diamond_e(\lambda_\varepsilon)
=\duality{\varepsilon}{\nabla_a e} \]
or
\[ a^\diamond_e(\lambda_\varepsilon)
=\duality{\nabla_a\varepsilon}{e} ,\]
where $\nabla$ denotes both the $A$-connection on $E$
and the dual $A$-connection on $E^\vee$.

Thus, we obtain the following full
characterization of the vector field $a^\diamond$:
\begin{gather*}
a^\diamond(\lambda_\varepsilon)
=\lambda_{\nabla_a\varepsilon},
\qquad\forall\varepsilon\in\sections{E^\vee}; \\
a^\diamond(f\circ\pi)
=\big(\anchor_a(f)\big)\circ\pi,
\qquad\forall f\in C^\infty(M)
.\end{gather*}
It follows immediately that the diagram
\begin{equation}\label{Buchnov} \begin{tikzcd}
C^\infty(E) \arrow[d, "0^*"']
\arrow[r, "a^\diamond"] &
C^\infty(E) \arrow[d, "0^*"] \\
C^\infty(M) \arrow[r, "\anchor_a"'] &
C^\infty(M)
\end{tikzcd} \end{equation}
commutes for all $a\in\sections{A}$.
Here $0$ denotes the zero section of
the vector bundle $\pi:E\to M$.

Given $x\in\sections{E}$, let $\underline{x}$
denote the corresponding vertical and fiberwise constant vector field on $E$.
Since $\underline{x}(\lambda_\varepsilon)=
\duality{\varepsilon}{x}\circ\pi$ and
$\underline{x}(f\circ\pi)=0$, it is easy to
check that \[ \lie{a^\diamond}{\underline{x}}(\lambda_\varepsilon)
=\underline{\nabla_a x}(\lambda_\varepsilon),
\qquad \lie{a^\diamond}{\underline{x}}(f\circ\pi)=0,
\qquad\text{and}\qquad
\underline{\nabla_a x}(f\circ\pi)=0 .\]
Therefore, we obtain
\begin{equation}\label{Praha}
\lie{a^\diamond}{\underline{x}}
= \underline{\nabla_a x}
,\qquad\forall a\in\sections{A},
x\in\sections{E}
.\end{equation}

Recall that, to the symmetric product $x_1\odot x_2\odot\cdots\odot x_n$ in $\sections{S(E)}$ of
sections $x_1,x_2,\dots,x_n$ of $E\to M$,
corresponds the fiberwise differential operator
\[ C^\infty(E)\ni f \mapsto
0^*\big(\underline{x_1}\circ\underline{x_2}\circ\dots\circ\underline{x_n}(f)\big)
\in C^\infty(M) \]
in $\DO\big(E,0(M)\big)$.

Then, the fiberwise differential operator
in $\DO\big(E,0(M)\big)$ corresponding to
\[ \delta_a(x_1\odot\cdots\odot x_n)\in\sections{S(E)} \] is
\[ C^\infty(E)\ni f \mapsto \anchor_a
\big(0^*\big(\underline{x_1}\circ\underline{x_2}
\circ\dots\circ\underline{x_n}(f)\big)\big)
-0^*\big(\underline{x_1}\circ\underline{x_2}\circ\dots\circ
\underline{x_n}\circ a^\diamond(f)\big) \in C^\infty(M) .\]
It follows from the commutativity of
Diagram~\eqref{Buchnov} and from
Equation~\eqref{Praha} that
\[ \begin{split}
& \anchor_a
\big(0^*\big(\underline{x_1}\circ\underline{x_2}
\circ\dots\circ\underline{x_n}(f)\big)\big)
-0^*\big(\underline{x_1}\circ\underline{x_2}\circ\dots\circ\underline{x_n}\circ a^\diamond(f)\big) \\
=\strut & 0^*\big(a^\diamond\circ\underline{x_1}\circ\underline{x_2}
\circ\dots\circ\underline{x_n}(f)\big)
-0^*\big(\underline{x_1}\circ\underline{x_2}
\circ\dots\circ\underline{x_n}\circ a^\diamond(f)\big) \\
=\strut & \sum_{k=1}^n
0^*\big(\underline{x_1}\circ\cdots\circ
\lie{a^\diamond}{\underline{x_k}}\circ\dots
\circ\underline{x_n}(f)\big) \\
=\strut & \sum_{k=1}^n
0^*\big(\underline{x_1}\circ\cdots\circ
\underline{\nabla_a x_k}\circ\dots
\circ\underline{x_n}(f)\big)
.\end{split} \]
Hence, we conclude that
\[ \delta_a(x_1\odot\cdots\odot x_n)
= \sum_{k=1}^n x_1\odot\cdots\odot
\nabla_a x_k\odot\cdots\odot x_n .\qedhere\]
\end{proof}

As a consequence, the Kapranov dg-manifold structure on $A[1]\oplus E$
encoded by the $A$-action by coderivations $\delta$ on $\sections{S(E)}$
is the linear Kapranov dg manifold structure of Example~\ref{Yellowknife}
and its characteristic class vanishes.

In particular, we have the following

\begin{proposition}
\label{Bott_vf}
Given a Lie pair $(L,A)$, the $A$-action by coderivations $\delta^\dagger$ on $\sections{S(L/A)}$
--- the coalgebra of distributions on $L/A$ supported on the zero section ---
induced by the Bott representation $\nabla^{\Bott}$ of $A$ on $L/A$ satisfies
\[ \delta^\dagger_a(b_1\odot\cdots\odot b_n)=\sum_{k=1}^n b_1\odot\cdots\odot
b_{k-1}\odot\nabla^{\Bott}_a b_k\odot b_{k+1}\odot\cdots\odot b_n \]
for all $a\in\sections{A}$ and $b_1,b_2,\dots,b_n\in\sections{L/A}$.
Therefore, the Kapranov dg-manifold structure on $A[1]\oplus L/A$
arising from the Bott representation of $A$ on $L/A$ is linear
--- see Example~\ref{Yellowknife} ---
and its characteristic class must vanish.
\end{proposition}

\subsubsection{Groupoid homogeneous space}

To obtain a much more interesting Kapranov dg-manifold structure on $A[1]\oplus L/A$,
let us temporarily restrict ourselves to \emph{real} Lie algebroid pairs $(L,A)$,
i.e.\ let us work over the field $\RR$.
In this case, the Lie algebroid/groupoid analogues of Lie's fundamental theorems assert
the existence of a pair $(\groupoid{L},\groupoid{A})$ of local Lie groupoids,
which the Lie functor transforms into $L$ and $A$ respectively \cite{MR2795150}.
The Lie groupoid $\groupoid{A}$ acts from the right on the groupoid $\groupoid{L}$
and the resulting homogeneous space $\groupoid{L}/\groupoid{A}$ obviously contains
the unit space $M$ of $(\groupoid{L},\groupoid{A})$
as a distinguished embedded submanifold~\cite{MR1612164}.
Furthermore, the Lie groupoid $\groupoid{A}$ acts from the left on
$\groupoid{L}/\groupoid{A}$ and the induced $A$-action
on $\groupoid{L}/\groupoid{A}$ stabilizes the submanifold $M$.
The normal bundle $N_M$ to $M$ inside $\groupoid{L}/\groupoid{A}$ is naturally isomorphic to $L/A$
and, by way of the exponential map developed in Section~\ref{Urumqi},
can be identified to a tubular neighborhood of $M$ inside $\groupoid{L}/\groupoid{A}$.
Therefore, applying the construction of Proposition~\ref{prop:KapranovOfLieAlgAction}
to the aforementioned $A$-action on $\groupoid{L}/\groupoid{A}$ stabilizing $M$,
we can endow $A[1]\oplus L/A$ with a Kapranov dg-manifold structure.

More precisely, consider the usual left action
of the Lie algebroid $L$ on the (local) Lie groupoid $\groupoid{L}$:
\[ \sections{L}\ni l\longmapsto
\rivf{l}\in\XX(\groupoid{L}) .\]
Recall that $\rivf{l}$ denotes the right-invariant vector field on
the Lie groupoid $\groupoid{L}$ corresponding to the section $l$ of the Lie algebroid $L$:
\[ \rivf{l}_g = \left. \frac{d}{d\tau} \big(\exp(\tau l_{\source(g)})\big)^{-1}\cdot g \right|_0
, \quad\forall g\in\groupoid{L} .\]
Since $s_*(\rivf{l}_g)=\anchor(l_{s(g)})$, the vector field $\rivf{l}$ is $s$-projectable onto $M$.
Hence we have an action of $L$ on $s:\groupoid{L}\to M$.
Since $A$ is a Lie subalgebroid of $L$
and the right-invariant vector fields on $\groupoid{L}$
are $\bm{q}$-projectable onto $\groupoid{L}/\groupoid{A}$,
there exists an induced action of $A$
on $s:\groupoid{L}/\groupoid{A}\to M$:
\begin{equation}\label{Fedosov_inf_action}
\begin{tikzcd}[row sep=tiny]
& \XX_{s}(\groupoid{L}/\groupoid{A}) \arrow[dd, "s_*"] \\
\sections{A} \arrow[ru, "\ddagger"] \arrow[rd, "\anchor"'] & \\
& \XX(M)
\end{tikzcd}\end{equation}

For all $a\in\sections{A}$, the $s$-projectable vector field
$a^\ddagger=\bm{q}_*\big(\overleftarrow{i(a)}\big)\in\XX_s(\groupoid{L}/\groupoid{A})$
is tangent to the unit section
of $s:\groupoid{L}/\groupoid{A}\to M$.
Therefore, the derivation $a^\ddagger$ of $C^\infty(\groupoid{L}/\groupoid{A})$
induces a coderivation (denoted $\delta^\ddagger_a$) of the coalgebra
$\DO(\groupoid{L}/\groupoid{A},1(M))\cong\frac{\enveloping{L}}{\enveloping{L}\sections{A}}$
of distributions on $\groupoid{L}/\groupoid{A}$ with support the unit section
--- see Lemma~\ref{Krakow}.
Thus, the action \eqref{Fedosov_inf_action} induces an
$A$-action $\delta^\ddagger$ on
$\frac{\enveloping{L}}{\enveloping{L}\sections{A}}$ by coderivations.

For every $l\in\sections{L}$, multiplication by $l$
from the left in the universal algebra $\enveloping{L}$
is a coderivation of (the coalgebra) $\enveloping{L}$
which stabilizes the left ideal $\enveloping{L}\sections{A}$ of (the algebra) $\enveloping{L}$
and thus determines a coderivation of the coalgebra
$\frac{\enveloping{L}}{\enveloping{L}\sections{A}}$.
Hence $\frac{\enveloping{L}}{\enveloping{L}\sections{A}}$ is naturally an $A$-module:
indeed, the multiplication by $i(a)$ from the left in $\enveloping{L}$
descends to a representation of $A$ on $\frac{\enveloping{L}}{\enveloping{L}\sections{A}}$.

\begin{proposition}
\label{Fedosov_vf}
Let $(L,A)$ be a pair of \emph{real} Lie algebroids
with associated (local) Lie groupoids $\groupoid{L}$ and $\groupoid{A}$.
The $A$-action by coderivations $\delta^\ddagger$
on $\DO(\groupoid{L}/\groupoid{A},1(M))$
induced by the $A$-action \eqref{Fedosov_inf_action}
on $s:\groupoid{L}/\groupoid{A}\to M$ is precisely
the aforementioned natural $A$-module structure
on $\frac{\enveloping{L}}{\enveloping{L}\sections{A}}$:
\[ \delta^\ddagger_a (l_1\cdot l_2\cdot\cdots \cdot l_n\cdot\bm{1})=
i(a)\cdot l_1\cdot l_2\cdot\cdots \cdot l_n\cdot\bm{1} ,\]
for all $a\in\sections{A}$ and $l_1,l_2,\dots,l_n\in\sections{L}$.
\end{proposition}

\begin{proof}
As previously, let $\bm{q}$ denote canonical projection
$\bm{q}:\groupoid{L}\to\groupoid{L}/\groupoid{A}$,
let $1$ denote the unit map $1:M\to\groupoid{L}$,
and let $\livf{l}$ denote the left-invariant vector field
on $\groupoid{L}$ corresponding to the section $l$ of $L$:
\[ \livf{l}_g = \left. \frac{d}{d\tau} g\cdot\exp(\tau l_{\target(g)}) \right|_0
,\quad\forall g\in\groupoid{L} .\]

Recall that, to the image
$l_1\cdot l_2\cdot\cdots\cdot l_n\cdot
\bm{1}$ in $\frac{\enveloping{L}}{\enveloping{L}\sections{A}}$ of the product
$l_1\cdot l_2\cdot\cdots\cdot
l_n\in\enveloping{L}$
of sections $l_1,l_2,\dots,l_n$ of $L$,
corresponds the fiberwise differential operator
\[ C^\infty(\groupoid{L}/\groupoid{A})\ni f
\mapsto 1^*\big(\overrightarrow{l_1}\circ\overrightarrow{l_2}\circ\dots\circ\overrightarrow{l_n}
(\bm{q}^*f)\big)\in C^\infty(M) .\]
in $\DO(\groupoid{L}/\groupoid{A},1(M))$.

Then, the fiberwise differential operator in $\DO(\groupoid{L}/\groupoid{A},1(M))$
corresponding to
\[ \delta^\ddagger_a(l_1\cdot l_2\cdot\cdots\cdot l_n\cdot\bm{1})
\in\frac{\enveloping{L}}{\enveloping{L}\sections{A}} \] is
\[ C^\infty(\groupoid{L}/\groupoid{A})\ni f
\mapsto
\rho_a\Big(1^*\big(\overrightarrow{l_1}\circ\overrightarrow{l_2}\circ\dots\circ\overrightarrow{l_n}
(\bm{q}^*f)\big)\Big)
-1^*\big(\overrightarrow{l_1}\circ\overrightarrow{l_2}\circ\dots\circ\overrightarrow{l_n}
(\bm{q}^*a^\ddagger(f))\big)
\in C^\infty(M) .\]

Differentiating the identity
\[ \big(\exp(\tau l_m)\big)^{-1}
\cdot\exp(\tau l_m)=1_{\target(\exp(\tau l_m))} \]
with respect to the parameter $\tau$ at $0$,
we obtain
\[ \rivf{l}_{1_m}+\livf{l}_{1_m}=1_*\big(\anchor(l_m)\big) .\]

It follows that
\[ \begin{split}
& \rho_a\Big(1^*\big(\overrightarrow{l_1}\circ\overrightarrow{l_2}\circ\dots\circ\overrightarrow{l_n}
(\bm{q}^*f)\big)\Big)
-1^*\big(\overrightarrow{l_1}\circ\overrightarrow{l_2}\circ\dots\circ\overrightarrow{l_n}
(\bm{q}^*a^\ddagger(f))\big)
\\ =\strut & 1^*\big(1_*(\anchor_a)\circ\overrightarrow{l_1}
\circ\overrightarrow{l_2}\circ\dots\circ\overrightarrow{l_n}
(\bm{q}^*f)\big)
-1^*\big(\overrightarrow{l_1}\circ\overrightarrow{l_2}\circ\dots\circ\overrightarrow{l_n}
\circ\overleftarrow{i(a)}
(\bm{q}^* f)\big)
\\ =\strut & 1^*\big(1_*(\anchor_a)\circ\overrightarrow{l_1}
\circ\overrightarrow{l_2}\circ\dots\circ\overrightarrow{l_n}
(\bm{q}^*f)\big) -1^*\big(\overleftarrow{i(a)}
\circ\overrightarrow{l_1}\circ\overrightarrow{l_2}
\circ\dots\circ\overrightarrow{l_n}(\bm{q}^* f)\big)
\\ =\strut &
1^*\big(\overrightarrow{i(a)}
\circ\overrightarrow{l_1}
\circ\overrightarrow{l_2}\circ\dots\circ\overrightarrow{l_n}
(\bm{q}^* f)\big)
.\end{split} \]
Hence, we conclude that
\[ \delta^\ddagger_a (l_1\cdot l_2\cdot\cdots \cdot l_n\cdot\bm{1})=
i(a)\cdot l_1\cdot l_2\cdot\cdots \cdot l_n\cdot\bm{1} .\qedhere\]
\end{proof}

The normal bundle to the unit section of $\groupoid{L}/\groupoid{A}$
is naturally isomorphic to the vector bundle $\pi:L/A\to M$.
The identification of a tubular neighborhood of the unit section of $\groupoid{L}/\groupoid{A}$
with a tubular neighborhood of the zero section of $L/A$ via
an exponential map $\exp^{\nabla,j}$ induces, on the distribution level,
an identification $\pbw^{\nabla,j}$ of the coalgebras
$\frac{\enveloping{L}}{\enveloping{L}\sections{A}}$ and $\sections{S(L/A)}$
--- see Theorem~\ref{Madagascar}.

Transferring $\delta^\ddagger$ through the coalgebra isomorphism
$\pbw^{\nabla,j}:\sections{S(L/A)}\to\frac{\enveloping{L}}{\enveloping{L}\sections{A}}$,
we obtain an $A$-action $\delta^\lightning$ by coderivations on $\sections{S(L/A)}$:
\begin{equation}\label{eq:Mathieu}
\delta^\lightning_a(s)=
(\pbw^{\nabla,j})^{-1}\circ\delta^\ddagger_a\circ\pbw^{\nabla,j}(s)=
(\pbw^{\nabla,j})^{-1}\big(i(a)\cdot\pbw^{\nabla,j}(s)\big)
\end{equation}
for all $a\in\sections{A}$ and $s\in\sections{L/A}$.
Obviously, $\delta^\lightning$ depends on the choice of $\nabla$ and $j$,
while $\delta^\ddagger$ does not.
We note that $\delta^\lightning_a(1)=0$.

\begin{proposition}\label{schist}
The Kapranov dg-manifold structure on $A[1]\oplus L/A$ arising
(as in Proposition~\ref{prop:KapranovOfLieAlgAction})
from the $A$-action \eqref{Fedosov_inf_action} on the fiber bundle
$s:\groupoid{L}/\groupoid{A}\to M$
coincides with the one characterized (as in Proposition~\ref{Tasmania})
by the $A$-action by coderivations $\delta^\lightning$ on $\sections{S(L/A)}$.
\end{proposition}

\section{Kapranov dg-manifolds stemming from Lie pairs}
\label{Abidjan}

Together, Proposition~\ref{schist} and Equation~\eqref{eq:Mathieu}
show that the Kapranov dg-manifold structure on $A[1]\oplus L/A$
arising from the $A$-action \eqref{Fedosov_inf_action}
on $\groupoid{L}/\groupoid{A}$
--- i.e.\ the $A$-action by coderivations $\delta^\lightning$ ---
can be characterized purely
algebraically (in terms of the coalgebra isomorphism
$\pbw^{\nabla,j}$ and the multiplication in the universal
enveloping algebra $\enveloping{L}$)
without resorting to the local Lie groupoids
$\groupoid{L}$ and $\groupoid{A}$. While the local Lie
groupoids $\groupoid{L}$ and $\groupoid{A}$ and thus
the exponential map $\exp^{\nabla,j}$ only exist for
\emph{real} Lie pairs, the coalgebra isomorphism $\pbw^{\nabla,j}$
and thus the $A$-action by coderivations $\delta^\lightning$
make sense for \emph{all} Lie pairs, \emph{whether real or complex}.
Therefore, the construction of the Kapranov dg-manifold
structure on $A[1]\oplus L/A$ outlined at the end
of Section~\ref{Patagonia} can be readily generalized to
complex Lie pairs notwithstanding (despite) the loss
of the Lie groupoid picture.
The purpose of the present section is to establish the two results
pertaining to this Kapranov dg-manifold structure on $A[1]\oplus L/A$
announced as Theorems~\ref{Amsterdam} and~\ref{Dublin}
in Section~\ref{Biloxi}.

Consider the short exact sequence of vector bundles \eqref{Sydney} associated with a Lie pair $(L,A)$.
A choice of a splitting \eqref{Murmansk}
determines a bundle map $\scindement:\Lambda^2(L/A)\to A$ by the relation
\begin{equation}\label{eq:kj}
\scindement(b_1,b_2)=p\big(\lie{j(b_1)}{j(b_2)}\big),
\quad\forall b_1,b_2\in\sections{L/A}
.\end{equation}

Furthermore, the splitting~\eqref{Murmansk}
determines a $\KK$-bilinear map
\[ \Delta^{\Bott}:\sections{L/A}\times\sections{A}\to\sections{A} \]
through the relation
\[ \Delta^{\Bott}_b a=p\big(\lie{j(b)}{i(a)}\big),
\quad\forall b\in\sections{L/A}, \forall a\in\sections{A} .\]
The map $\Delta^{\Bott}$ is $R$-linear in its first argument and satisfies
\[ \Delta^{\Bott}_b(f\cdot a)
=\anchor_{j(b)}(f)\cdot a+f\cdot\Delta^{\Bott}_b a \]
for all $b\in\sections{L/A}$, $a\in\sections{A}$, and $f\in R$.
It is a sort of `$L/A$-pseudo-connection on $A$.'
In particular, if the subbundle $j(L/A)$ is a subalgebroid
of $L$ (i.e.\ if $\scindement=0$), then $\Delta^{\Bott}$
is a true connection: the Bott $j(L/A)$-connection on $A$.

\subsection{Torsion-free connections}

Given an $L$-connection $\nabla$ on $L/A$, its \emph{torsion} is
the bundle map $\torsion:\Lambda^2 L\to L/A$ defined by
\begin{equation}\label{Tibet} \torsion(l_1,l_2)
=\nabla_{l_1}q(l_2)-\nabla_{l_2}q(l_1)-q\big(\lie{l_1}{l_2}\big),
\quad\forall l_1,l_2\in\sections{L}
.\end{equation}

\begin{definition}
\begin{enumerate}
\item An $L$-connection $\nabla$ on $L/A$ is said to be torsion-free if $T^\nabla=0$.
\item An $L$-connection $\nabla$ on $L/A$ is said to be an extension
of the Bott representation of $A$ on $L/A$ --- see Equation~\eqref{Istambul} ---
if $\nabla_{i(a)}b=\nabla^{\Bott}_a b$
for all $a\in\sections{A}$ and $b\in\sections{L/A}$.
\end{enumerate}
\end{definition}

\pagebreak[2]
\begin{lemma}\label{Mascate}
\strut
\begin{enumerate}
\item If the $L$-connection $\nabla$ on $L/A$ is torsion-free,
then it is an extension of the Bott representation of $A$ on $L/A$.
\item If the $L$-connection $\nabla$ on $L/A$
is an extension of the Bott representation of $A$ on $L/A$,
then there exists a unique bundle map $\beta^\nabla:\Lambda^2(L/A)\to L/A$ making the diagram
\[ \begin{tikzcd} \Lambda^2 L \arrow[d, two heads, "q"'] \arrow[r, "\torsion"] &
L/A \\ \Lambda^2 (L/A) \arrow[ru, "\beta^\nabla"'] & \end{tikzcd} \]
commute.
\end{enumerate}
\end{lemma}

\begin{proof}
It suffices to notice that, for all $a\in\sections{A}$ and $l\in\sections{L}$,
\[ \torsion(i(a),l)=\nabla_{i(a)}q(l)-q\big(\lie{i(a)}{l}\big)
=\nabla_{i(a)}q(l)-\nabla^{\Bott}_a q(l) .\qedhere\]
\end{proof}

\begin{proposition}\label{Vilnius}
Given any Lie pair $(L,A)$, there exist torsion-free $L$-connections on $L/A$.
\end{proposition}

\begin{proof}[Sketch of proof]
First, construct an $L$-connection $\nabla$ on $L/A$ using the usual partition of unity argument.
Then, tweak $\nabla$ so as to obtain an extension
$\nabla':\sections{L}\times\sections{L/A}\to\sections{L/A}$
of the Bott $A$-connection: choose a splitting \eqref{Murmansk}
of the short exact sequence~\eqref{Sydney} and set
\[ \nabla'_l b=\nabla^{\Bott}_{p(l)}b+\nabla_{j\circ q(l)}b .\]
Finally, obtain a torsion-free connection $\nabla'':\sections{L}\times\sections{L/A}\to\sections{L/A}$
from $\nabla'$ by setting
\[ \nabla''_l b=\nabla'_l b-\frac{1}{2}\beta^{\nabla'}\big(q(l),b\big) .\qedhere\]
\end{proof}

\begin{remark}
The splitting $j$ and the three connections $\nabla$, $\nabla'$, and $\nabla''$
determine the same Poincaré--Birkhoff--Witt isomorphism:
\[ \pbw^{\nabla,j}=\pbw^{\nabla',j}=\pbw^{\nabla'',j} .\]
\end{remark}

\subsection{Atiyah class of a Lie pair}

Let $(L,A)$ be a Lie pair.
Choose a splitting \eqref{Murmansk} of the short exact sequence \eqref{Sydney}
and a \emph{torsion-free} $L$-connection $\nabla$ on $L/A$.

The $L$-connection $\nabla$ on $L/A$ extends by derivations to an $L$-connection on $S(L/A)$,
which we also denote by $\nabla$ by abuse of notation.
Its curvature is the bundle map $\curvature:\Lambda^2 L\to\Der\big(S(L/A)\big)$ defined by
\begin{equation}\label{courbure}
\curvature(l_1,l_2)=\lie{\nabla_{l_1}}{\nabla_{l_2}}
-\nabla_{\lie{l_1}{l_2}}
,\quad\forall l_1,l_2\in\sections{L} 
\end{equation}
--- the sections of the vector bundle $\Der\big(S(L/A)\big)$ are the derivations
of the $R$-algebra $\sections{S(L/A)}$.

Consider the bundle map $\cocycle^{\nabla}:A\otimes S^2(L/A)\to L/A$ defined by
\begin{equation}\label{ZZZ}
\cocycle^{\nabla}(a;b_1\odot b_2)=\tfrac{1}{2}\big\{\curvature\big(i(a),j(b_1)\big)b_2
+\curvature\big(i(a),j(b_2)\big)b_1\big\}
\end{equation}
for all $a\in\sections{A}$ and $b_1,b_2\in\sections{L/A}$.
In particular, we have
\begin{equation}\label{Boussu}
\cocycle^{\nabla}(a;b^2)=\curvature\big(i(a),j(b)\big)b,
\quad \forall a\in\sections{A},b\in\sections{L/A}
.\end{equation}
Since $\nabla$ is an extension of the (flat) Bott $A$-connection on $L/A$,
the map $\cocycle^{\nabla}$ does not actually depend on the chosen splitting $j$.

\pagebreak[2]
\begin{lemma}\label{Norilsk}
\strut
\begin{enumerate}
\item The element $\cocycle^{\nabla}\in\sections{A^\vee\otimes S^2(L/A)^\vee\otimes L/A}$
is a Chevalley--Eilenberg 1-cocycle of $A$ with values in $A$-module
$S^2(L/A)^\vee\otimes L/A$ induced by the Bott connection.
\item Given two torsion-free $L$-connections $\nabla^\sharp$ and $\nabla^\flat$,
the difference $\cocycle^{\nabla^\sharp}-\cocycle^{\nabla^\flat}$
is a Chevalley--Eilenberg 1-coboundary with respect to the representation of $A$
on $S^2(L/A)^\vee\otimes L/A$ induced by the Bott connection.
\end{enumerate}
\end{lemma}

Lemma~\ref{Norilsk} is analogous to \cite[Theorem~2.5]{MR3439229} and its proof is similar.

The cohomology class $[\cocycle^{\nabla}]\in H_{\CE}^1\big(A;S^2(L/A)^\vee\otimes L/A\big)$,
which is independent of the chosen connection $\nabla$,
is called the \emph{Atiyah class} of the Lie pair $(L,A)$
and is denoted $\alpha_{L/A}$.

\begin{remark}
The definition of the Atiyah class given above agrees with the definition of~\cite{MR3439229}
even though the Atiyah cocycle $\cocycle^\nabla$ defined here is slightly different from
(but nevertheless cohomologous to) the Atiyah cocycle defined in~\cite{MR3439229}.
\end{remark}

\subsection{Main theorems}
\label{Paris}

We are now ready to state the results that we announced
as Theorems~\ref{Amsterdam} and~\ref{Dublin} in Section~\ref{Biloxi}.
We defer their proofs to the end of Section~\ref{Olten}.
\begin{theorem}
\label{Ohio}
Let $(L,A)$ be a Lie pair. The choice of (i) a splitting $j:L/A\to L$
of the short exact sequence of vector bundles
\eqref{Sydney} and (ii) a torsion-free $L$-connection $\nabla$ on $L/A$
determines a Kapranov dg-manifold structure on $A[1]\oplus L/A$
whose homological vector field
$\Qd\in\XX(A[1]\oplus L/A)$ is a
degree $+1$ derivation \[ \sections{\Lambda^{\bullet} A^\vee\otimes\hat{S}((L/A)^\vee)}
\xto{\Qd} \sections{\Lambda^{\bullet+1} A^\vee\otimes\hat{S}((L/A)^\vee)} \]
of the algebra of functions on $A[1]\oplus L/A$ of the form
\[ \Qd=d_{A}^{\nabla^{\Bott}}+\sum_{k=2}^{\infty}\cR_k ,\]
where
\begin{enumerate}
\item $d_{A}^{\nabla^{\Bott}}$ is the Chevalley--Eilenberg
differential corresponding to the Bott representation of $A$
on $\hat{S}\big((L/A)^\vee\big)$;
\item $\cR_k$ is a section of the vector bundle
$A^\vee\otimes S^k(L/A)^\vee\otimes L/A$
seen as an operator on
$\sections{\Lambda^{\bullet}A^\vee\otimes\hat{S}((L/A)^\vee)}$
(acting by contraction);
\item $\cR_2$ is the Atiyah 1-cocycle $\cocycle^{\nabla}$
associated with the connection $\nabla$; and
\item $\cR_k$, for each $k\geqslant 3$,
can be computed explicitly as an algebraic function
of $\scindement$, $\cocycle^{\nabla}$, $\curvature$, 
(see Equations~\eqref{eq:kj}, \eqref{ZZZ}, and~\eqref{courbure})
their higher covariant derivatives, and compositions thereof.
\end{enumerate}
Moreover, the various Kapranov dg-manifold structures
on $A[1]\oplus L/A$ resulting from all possible
choices of splitting and connection are all isomorphic.
\end{theorem}

\begin{Aa}
In the case of a pair $(L,A)$ of \emph{real} Lie algebroids,
the Kapranov dg-manifold structure on $A[1]\oplus L/A$
of Theorem~\ref{Ohio} is
the Kapranov dg-manifold arising,
as in Proposition~\ref{schist}, from the left action of $A$
on $s:\groupoid{L}/\groupoid{A}\to M$.
In particular, the characteristic class of the left action
of $A$ on $s:\groupoid{L}/\groupoid{A}\to M$
--- see Remark~\ref{garedeLyon} ---
is the Atiyah class of the Lie pair $(L,A)$.
\end{Aa}

In general, the algebraic expression for $\cR_k$, with $k\geqslant 3$,
in Theorem~\ref{Ohio} can be very complicated.
However, in the special case of matched pairs satisfying
an additional assumption, we have a very clean description
of $\cR_k$ as indicated in the following proposition,
which will be needed in Section~\ref{Amiens}.

\begin{proposition}
\label{zebra}
Let $L=A\bowtie B$ be a matched pair of Lie algebroids
--- see Example~\ref{ex:matched}.
Assume that there exists a torsion-free flat $B$-connection
$\ddot{\nabla}$ on $B$ and let $\nabla$ be the induced torsion-free
$L$-connection on $B$ defined by Equation~\eqref{eq:matched}.
Then, for the Lie pair $(L,A)$ and the natural splitting of
the short exact sequence \eqref{Sydney} identifying $L/A$ to $B$,
the sections $\cR_{k}$ of the vector bundles
$A^\vee\otimes B\otimes S^k(B^\vee)$ satisfy the simple
recursive relation
\[ \cR_{k+1}=(\id_{A^\vee\otimes B}\otimes\sym)
\big(d_B^{\widehat{\nabla}}\cR_k\big) ,\]
for all $k\geqslant 2$.
Here $\widehat{\nabla}$ denotes the $B$-connection on
$A^\vee\otimes B\otimes S^k(B^\vee)$ induced by the
$B$-connections $\Delta^{\Bott}$ on $A$ and $\ddot{\nabla}$ on $B$,
while \[ \sym:B^\vee\otimes S^k(B^\vee)\to S^{k+1}(B^\vee) \]
is the symmetrization map
sending $b\otimes(b_1\odot b_2\odot\cdots\odot b_k)$
to $b\odot b_1\odot b_2\odot\cdots\odot b_k$.
\end{proposition}

The Bott representation of $A$ on $L/A$ induces an $A$-action (by coderivations)
$\delta^\dagger$ on the $R$-coalgebra $\sections{S(L/A)}$,
which respects the grading --- see Proposition~\ref{Bott_vf}.

Multiplication by elements of $\sections{A}$ and $C^\infty(M)$
from the left in $\enveloping{L}$ induces an $A$-action 
(by coderivations) $\delta^\ddagger$ on the $R$-coalgebra
$\frac{\enveloping{L}}{\enveloping{L}\sections{A}}$
--- see Proposition~\ref{Fedosov_vf}.

\begin{lemma}\label{Nebraska}
The $A$-action on the $R$-coalgebra
$\frac{\enveloping{L}}{\enveloping{L}\sections{A}}$
induced by the multiplication by elements of $\sections{A}$ and $C^\infty(M)$
from the left in $\enveloping{L}$
preserves the filtration~\eqref{Kolkata}.
\end{lemma}

\begin{proof}
For all $a\in\sections{A}$ and $l_1,l_2,\cdots,l_n\in\sections{L}$, we have
\[ a l_1 \cdots l_n = \sum_{i=1}^{n} \big(l_1 \cdots l_{i-1} \lie{a}{l_i} l_{i+1}
\cdots l_n \big) + l_1 \cdots l_n a ,\]
where every term of the right hand side belongs to $\mathcal{U}^{\leqslant n}(L)$
except for the last which belongs to $\enveloping{L}\sections{A}$.
\end{proof}

\begin{theorem}
\label{Montana}
Let $(L,A)$ be a Lie pair. The following assertions are equivalent.
\begin{enumerate}
\newcounter{compteur}
\item\label{een} The Kapranov dg-manifold structure on
$A[1]\oplus L/A$ of Theorem~\ref{Ohio} is linearizable.
\item\label{twee} The Atiyah class
$\alpha_{L/A}\in H_{\CE}^1(A;S^2(L/A)^\vee\otimes L/A)$
of the Lie pair $(L,A)$ vanishes.
\item\label{drie}
For every splitting $j:L/A\to L$, there exists a torsion-free
$L$-connection $\nabla$ on $L/A$ such that the associated
Poincaré--Birkhoff--Witt map $\pbw^{\nabla,j}$
intertwines the $A$-action $\delta^\dagger$ on $\sections{S(L/A)}$
with the $A$-action $\delta^\ddagger$ on
$\frac{\enveloping{L}}{\enveloping{L}\sections{A}}$.
In other words, we have $\delta^\lightning=\delta^\dagger$.
\setcounter{compteur}{\value{enumi}}
\end{enumerate}
\end{theorem}

\begin{Ab}
In the case of real Lie algebroids, the following two assertions
may be added to the list above:
\begin{enumerate}
\setcounter{enumi}{\value{compteur}}
\item\label{vier} For every splitting $j:L/A\to L$,
there exists a torsion-free $L$-connection $\nabla$ on $L/A$
such that the exponential map $\exp^{\nabla,j}: L/A\to \groupoid{L}/\groupoid{A}$
(or more precisely its fiberwise $\infty$-jet along $M$)
intertwines the $A$-actions by coderivations
$\delta^\dagger$ and $\delta^\ddagger$ described
in Propositions~\ref{Bott_vf} and~\ref{Fedosov_vf}.
\item\label{vijf}
There exists a fiber-preserving diffeomorphism $\varphi$
from a tubular neighborhood of the zero section in $L/A$
to a tubular neighborhood of the unit section
in $\groupoid{L}/\groupoid{A}$, which fixes the submanifold $M$
(i.e. $\varphi(0_m)=1_m$ for all $m\in M$),
and whose fiberwise $\infty$-jet along $M$ intertwines
the $A$-actions by coderivations $\delta^\dagger$ and
$\delta^\ddagger$ described in Propositions~\ref{Bott_vf}
and~\ref{Fedosov_vf}.
\end{enumerate}
\end{Ab}

\begin{remark}
Addendum to Theorem~\ref{Montana} was sharpened
by Laurent-Gengoux and Voglaire \cite{MR4059952}
after an early version of the present paper was posted on arXiv.org.
Nevertheless, since complex Lie algebroids do not arise
from Lie groupoids \cite{MR2285039},
Theorem~\ref{Montana} and its addendum cannot be derived
from the result pertaining to groupoids obtained in \cite{MR4059952}.
\end{remark}

\subsection{Proofs of Theorem~\ref{Ohio}, Proposition~\ref{zebra}, and Theorem~\ref{Montana}}
\label{Olten}

Every choice of an $L$-connection $\nabla$ on $L/A$
and a splitting $j:L/A\to L$ (see~\eqref{Murmansk}) of the short exact sequence of vector bundles~\eqref{Sydney}
determines a Poincaré--Birkhoff--Witt map
$\pbw^{\nabla,j}:\sections{S(L/A)}\to\tfrac{\enveloping{L}}{\enveloping{L}\sections{A}}$,
which is an isomorphism of filtered $R$-coalgebras according to Theorem~\ref{Nairobi}.
In this section, we will often write $\pbw$ instead of $\pbw^{\nabla,j}$
so as to avoid making some equations overly cumbersome.
Being a quotient of the universal enveloping algebra $\enveloping{L}$ by a left ideal,
the $R$-coalgebra $\tfrac{\enveloping{L}}{\enveloping{L}\sections{A}}$
is naturally a left $\enveloping{L}$-module. Hence
$\tfrac{\enveloping{L}}{\enveloping{L}\sections{A}}$
is endowed with a canonical $L$-action (by coderivations).
Pulling back this $L$-action through $\pbw$, we obtain an $L$-action
$\nabla^\lightning$ on $\sections{S(L/A)}$:
\begin{equation}\label{Vermont}
\nabla^\lightning_l(s)=\pbw^{-1}\big(l\cdot\pbw(s)\big)
,\end{equation}
for all $l\in\sections{L}$ and $s\in\sections{S(L/A)}$.

\begin{lemma}\label{Daegu}
For all $a\in\sections{A}$ and $b\in\sections{L/A}$, we have
$\nabla^\lightning_{i(a)}b = \nabla^{\Bott}_a b$.
\end{lemma}

\begin{proof}
We have \[ i(a)\cdot j(b)-j(\nabla^{\Bott}_a b) \in\enveloping{L}\sections{A} \]
since \[ i(a)\cdot j(b)=\lie{i(a)}{j(b)}+j(b)\cdot i(a)
= j(\nabla^{\Bott}_a b)+i\circ p(\lie{i(a)}{j(b)})+j(b)\cdot i(a) \] in $\enveloping{L}$.
It follows from Equation~\eqref{Pittsburgh} that
\[ i(a)\cdot\pbw^{\nabla,j}(b) = \pbw^{\nabla,j}(\nabla^{\Bott}_a b) .\]
The desired result then follows immediately from Equation~\eqref{Vermont}.
\end{proof}

The $L$-connection $\nabla$ on $L/A$ extends by derivations
to an $L$-connection on $S(L/A)$, which we also denote by $\nabla$ by abuse of notation
--- see Equation~\eqref{Edmonton}.
Since the difference $\nabla^\lightning-\nabla$ of the two $L$-connections on $S(L/A)$
is an $R$-bilinear map $\sections{L}\times\sections{S(L/A)}\to\sections{S(L/A)}$,
we may define a bundle map \[ \Theta: L\otimes S(L/A)\to S(L/A) \] by the relation
\begin{equation}\label{California}
\Theta(l;s) = \nabla^\lightning_l s-\nabla_l s-q(l)\odot s,
\quad\forall l\in\sections{L},s\in\sections{S(L/A)}
.\end{equation}

The proofs of Theorems~\ref{Ohio} and~\ref{Montana} rely heavily on the
following

\begin{proposition}\label{prop:recursion}
All expressions of the form $\Theta(l;b^n)$
with $l\in\sections{L}$, $b\in\sections{L/A}$, and $n\in\NN$
can be computed recursively from the splitting $j$, the connection $\nabla$,
its torsion $\torsion$, and its curvature $\curvature$,
through an explicit algebraic algorithm
without resorting to the connection $\nabla^{\lightning}$
or the map $pbw^{\nabla,j}$ and its intractable inverse.
\end{proposition}

In preparation for the proof of Proposition~\ref{prop:recursion},
we proceed to establish several preliminary results.

The next lemma gives the `initial value' of $\Theta$ for the recursion
in Proposition~\ref{prop:recursion}.

\begin{lemma}\label{Oregon}
For all $l\in\sections{L}$, we have $\Theta(l;1)=0$.
\end{lemma}

\begin{proof}
Since $\pbw(1)=\bm{1}$; $\pbw(b)=j(b)\cdot\bm{1}$
for all $b\in\sections{L/A}$; and $l-j\circ q(l)\in\sections{A}$
for all $l\in\sections{L}$,
it follows from Equations~\eqref{California} and~\eqref{Vermont} that
\[ \Theta(l;1)=\nabla^{\lightning}_l 1-\nabla_l 1-q(l)\odot 1
=\pbw^{-1}(l\cdot\bm{1})-0-q(l)
=\pbw^{-1}(j\circ q(l)\cdot\bm{1})-q(l)=0
.\qedhere\]
\end{proof}

\begin{lemma}\label{Grez-Doiceau}
For all $l\in\sections{L}$, the map $s\mapsto\Theta(l;s)$
stabilizes the subspace $\sections{S^{\geqslant 1}(L/A)}$ of $\sections{S(L/A)}$.
\end{lemma}

\begin{proof}
The universal enveloping algebra $\enveloping{L}$ decomposes as the direct sum
of its subalgebra $R$ and its left ideal $\enveloping{L}\sections{L}$.
Quotienting by the left ideal generated by $\sections{A}$,
we obtain the direct sum decomposition
\[ \frac{\enveloping{L}}{\enveloping{L}\sections{A}}\cong
R\oplus\frac{\enveloping{L}\sections{L}}{\enveloping{L}\sections{A}} .\]
It is clear from the relations \eqref{Ottawa}, \eqref{Pittsburgh}, and~\eqref{Rome}
that the isomorphism $\pbw^{\nabla,j}$ identifies $\sections{S^{\geqslant 1}(L/A)}$
to $\frac{\enveloping{L}\sections{L}}{\enveloping{L}\sections{A}}$.
Hence, it follows immediately from Equation~\eqref{Vermont}
that $\nabla^{\lightning}_l$ stabilizes $\sections{S^{\geqslant 1}(L/A)}$.
The desired result then follows from Equation~\eqref{California}.
\end{proof}

\begin{lemma}\label{Djibouti}
For all $a\in\sections{A}$ and $b\in\sections{L/A}$, we have
$\Theta\big(i(a);b\big)=\nabla^{\Bott}_a b-\nabla_{i(a)}b$.
\newline
In particular, if the $L$-connection $\nabla$ on $L/A$ is torsion-free,
then $\Theta\big(i(a);b\big)=0$.
\end{lemma}

\begin{proof}
The first assertion follows immediately from Equation~\eqref{California} and Lemma~\ref{Daegu}.
The second assertion then follows from Lemma~\ref{Mascate}.
\end{proof}

\begin{lemma}\label{Malawi}
For every $l\in\sections{L}$, the map $s\mapsto\Theta(l;s)$ is a coderivation of
the $R$-coalgebra $\sections{S(L/A)}$ which preserves the filtration
\begin{equation}\label{Texas}
R\into S^{\leqslant 1}(L/A) \into \cdots \into S^{\leqslant n-1}(L/A)
\into S^{\leqslant n}(L/A) \into S^{\leqslant n+1}(L/A) \into \cdots
.\end{equation}
\end{lemma}

\begin{proof}
The verification that, for every $l\in\sections{L}$,
the three maps $\nabla^\lightning_l$, $\nabla_l$, and $s\mapsto q(l)\odot s$
are coderivations of the $\KK$-coalgebra $\sections{S(L/A)}$ is straightforward.
Being $R$-linear, the map $s\mapsto\Theta(l;s)$ is
a coderivation of the $R$-coalgebra $\sections{S(L/A)}$
for every $l\in\sections{L}$.
According to Proposition~\ref{shark}, it then follows from Lemma~\ref{Oregon}
that the coderivation $s\mapsto\Theta(l;s)$ preserves the filtration~\eqref{Texas}.
\end{proof}

The next lemma translates Equation~\eqref{Rome},
the recursive relation defining the map $\pbw$, in terms of $\Theta$.

\begin{lemma}\label{Mozambique}
For all $n\in\NN$ and all $b_0,b_1,\dots,b_n\in\sections{L/A}$, we have
\[ \sum_{k=0}^n \Theta\big(j(b_k);b_0\odot\cdots\odot
\widehat{b_k}\odot\cdots\odot b_n\big) =0 .\]
\end{lemma}

\begin{proof}
Set $b^{\{k\}}=b_0\odot\cdots\odot
\widehat{b_k}\odot\cdots\odot b_n$ and rewrite Equation~\eqref{Rome} as
\[ (n+1) \pbw(b_0\odot\cdots\odot b_n)
= \sum_{k=0}^n \Big\{j(b_k)\cdot\pbw(b^{\{k\}})
-\pbw\big(\nabla_{j(b_k)}(b^{\{k\}})
\big)\Big\} .\]
Applying $\pbw^{-1}$ to both sides, we obtain
\[ (n+1)\ b_0\odot\cdots\odot b_n = \sum_{k=0}^n
\Big\{ \nabla^\lightning_{j(b_k)}(b^{\{k\}}) - \nabla_{j(b_k)}(b^{\{k\}}) \Big\} ,\]
which is equivalent to
\[ \sum_{k=0}^n \Big\{ \nabla^\lightning_{j(b_k)}(b^{\{k\}})
- \nabla_{j(b_k)}(b^{\{k\}}) - q\big(j(b_k)\big)\odot b^{\{k\}} \Big\} =0 .\]
The result follows from Equation~\eqref{California}.
\end{proof}

The splitting $j:L/A\to L$ (see~\eqref{Murmansk})
of the short exact sequence~\eqref{Sydney}
and the $L$-connection $\nabla$ on $L/A$ determine a $\KK$-bilinear map
\[ \Delta:\sections{L/A}\times\sections{L}\to\sections{L} \]
through the relation
\begin{equation}\label{Dour}
\Delta_b l=j(\nabla_l b)-\lie{l}{j(b)},\quad\forall b\in\sections{L/A},l\in\sections{L}
.\end{equation}
The map $\Delta$ is $R$-linear in its first argument and satisfies
\[ \Delta_b (f\cdot l)=\anchor_{j(b)}(f) \cdot l +f\cdot\Delta_b l \]
for all $b\in\sections{L/A}$, $l\in\sections{L}$, and $f\in R$.
It is a sort of `pseudo-$L/A$-connection on $L$.'

\begin{lemma}\label{pig}
For all $l\in\sections{L}$ and $b\in\sections{L/A}$, and $n\in\NN$, we have
\begin{multline*}
\Theta\left(\big(\id+\tfrac{1}{n+1}j\circ q\big)(l);b^{n+1}\right) \\
= b\odot\Theta\big(l;b^n\big)
+\nabla_{j(b)}\big(\Theta(l;b^n)\big)
-\Theta(\Delta_b l;b^n)
-\Theta\big(l;\nabla_{j(b)}(b^n)\big) \\
+\Theta\big(j(b);\Theta(l;b^n)\big)
-T^\nabla\big(l,j(b)\big)\odot b^n
-R^\nabla\big(l,j(b)\big)(b^n)
,\end{multline*}
where the operator $\id+\tfrac{1}{n+1}j\circ q$ is invertible
with inverse $\id-\tfrac{1}{n+2}j\circ q$.
\end{lemma}

The proof of Lemma~\ref{pig} relies heavily on Lemma~\ref{Mozambique}.

\begin{proof}
Taking $b_0=b_1=\dots=b_n=b$ in Lemma~\ref{Mozambique}, we obtain
\begin{equation}\label{Togo}
\Theta\big(j(b);b^n\big)=0
.\end{equation}
	
Taking $b_0=\nabla_l b$ and $b_1=\dots=b_n=b$ in Lemma~\ref{Mozambique}, we obtain
\begin{equation}\label{Angola}
-\Theta\big(j(b);\nabla_l(b^n)\big)=\Theta\big(j(\nabla_l b);b^n\big)
.\end{equation}
	
Taking $n=m+1$, $b_0=c$, and $b_1=\dots=b_n=b$ in Lemma~\ref{Mozambique}, we obtain
\begin{equation}\label{Brazil}
-\Theta\big(j(b);c\odot b^{m}\big)=\tfrac{1}{m+1}\Theta\big(j(c);b^{m+1}\big)
.\end{equation}
	
Since $\nabla^\lightning$ is flat and
$\nabla^\lightning_l s=q(l)\odot s+\nabla_l s+\Theta(l;s)$,
we have, for any $l,z\in\sections{L}$,
\[ \begin{split}
0 = & \ \nabla^\lightning_l\nabla^\lightning_z s
-\nabla^\lightning_z\nabla^\lightning_l s
-\nabla^\lightning_{\lie{l}{z}}s \\
= & \ T^\nabla(l,z)\odot s +R^\nabla(l,z) s \\
& +q(l)\odot\Theta(z;s)-q(z)\odot\Theta(l;s)
+\nabla_l\big(\Theta(z;s)\big)-\nabla_z\big(\Theta(l;s)\big) \\
& +\Theta\big(l;q(z)\odot s\big)-\Theta\big(z;q(l)\odot s\big)
+\Theta(l;\nabla_z s)-\Theta(z;\nabla_l s) \\
& +\Theta\big(l;\Theta(z;s)\big)-\Theta\big(z;\Theta(l;s)\big) -\Theta(\lie{l}{z};s)
.\end{split} \]
Substituting $j(b)$ for $z$ and $b^n$ for $s$
and making use of Equations~\eqref{Togo}, \eqref{Angola}, and~\eqref{Brazil},
we obtain the desired equation.
\end{proof}

\begin{proof}[Proof of Proposition~\ref{prop:recursion}]
Lemma~\ref{pig} together with Lemma~\ref{Malawi} provide the recursion relation
while Lemma~\ref{Oregon} gives the initial value.
\end{proof}

\begin{lemma}\label{Louisiana}
Provided the $L$-connection $\nabla$ on $L/A$ is torsion-free,
the map $\Delta$ defined by Equation~\eqref{Dour} satisfies the following properties:
\begin{enumerate}	
\item For all $b\in\sections{L/A}$ and $a\in\sections{A}$,
we have $\Delta_b\big(i(a)\big)=i\big(\Delta^{\Bott}_b a\big)$.
\item For all $b,c\in\sections{L/A}$, we have
$\Delta_b\big(j(c)\big)=\scindement(b,c)+j\big(\nabla_{j(b)}c\big)$.
\end{enumerate}
\end{lemma}

\begin{proof}
\textbf{(1)} For all $b\in\sections{L/A}$ and $a\in\sections{A}$, we have
\[ \begin{aligned}
\Delta_b\big(i(a)\big)=&\ j(\nabla_{i(a)} b)-\lie{i(a)}{j(b)} \\
=&\ j(\nabla_{i(a)} b)-\left\{j\circ q\big(\lie{i(a)}{j(b)}\big)
+i\circ p\big(\lie{i(a)}{j(b)}\big)\right\} \\
=&\ j(\nabla_{i(a)} b - \nabla^{\Bott}_{a} b)-i\circ p(\lie{i(a)}{j(b)}) \\
=&\ i\circ p(\lie{j(b)}{i(a)}) \\
=&\ i(\Delta^{\Bott}_b a)
.\end{aligned} \]

\textbf{(2)} Substituting $j(b)$ and $j(c)$ for $l_1$ and $l_2$
in Equation~\eqref{Tibet}, we obtain
\[ 0=\nabla_{j(b)}c-\nabla_{j(c)}b-q\big(\lie{j(b)}{j(c)}\big) \]
and, applying $j$ to both sides,
\[ 0 = j\big(\nabla_{j(b)}c\big)-j\big(\nabla_{j(c)}b\big)
-\lie{j(b)}{j(c)}+\scindement(b,c) .\]
The result follows from Equation~\eqref{Dour}, the definition of $\Delta$.
\end{proof}

\begin{remark}
If the connection $\nabla$ is torsion-free and $j(L/A)$ is a Lie subalgebroid of $L$,
then $\Delta_b\big(j(c)\big)=j(\nabla_{j(b)}c)$ for all $b,c\in\sections{L/A}$.
\end{remark}

\begin{lemma}\label{Gouvy}
Provided the $L$-connection $\nabla$ on $L/A$ is torsion-free,
the following assertions hold.
\begin{enumerate}
\item
For all $a\in\sections{A}$ and $b\in\sections{L/A}$,
we have $\Theta\big(i(a);b^{2}\big)=-\cocycle^{\nabla}(a;b^2)$.
\item
The Atiyah cocycle $\cocycle^{\nabla}$ vanishes if and only if
$\Theta\big(i(a);s\big)=0$ for all $a\in\sections{A}$
and $s\in\sections{S(L/A)}$.
\item
If $j(L/A)$ is a Lie subalgebroid of $L$ (i.e.\ $\scindement=0$)
and $\curvature\big(j(b_1),j(b_2)\big)=0$
for all $b_1,b_2\in\sections{L/A}$, then
\begin{equation}\label{worm}
\begin{aligned}
\Theta\big(i(a);b^{n+1}\big) =&\ b\odot\Theta\big(i(a);b^n\big)
-n \cocycle^{\nabla}(a;b^2)\odot b^{n-1} \\
& + \nabla_{j(b)}\big(\Theta(i(a);b^n)\big)
-\Theta\big(i(\Delta^{\Bott}_b a);b^n\big) -\Theta\big(i(a);\nabla_{j(b)}(b^n)\big)
\end{aligned}\end{equation}
for all $a\in\sections{A}$, $b\in\sections{L/A}$, and $n\in\NN$.
\end{enumerate}
\end{lemma}

\begin{proof}
\textbf{(1)} Since $\nabla$ is torsion-free, Lemma~\ref{pig};
Lemma~\ref{Louisiana}; and Equation~\eqref{Boussu} imply that
\begin{equation}\label{monkey}
\begin{aligned}
\Theta\big(i(a);b^{n+1}\big) =&\ b\odot\Theta\big(i(a);b^n\big) \\
& +\nabla_{j(b)}\big(\Theta(i(a);b^n)\big)-\Theta(i(\Delta^{\Bott}_b a);b^n)
-\Theta\big(i(a);\nabla_{j(b)}(b^n)\big) \\
& +\Theta\big(j(b);\Theta(i(a);b^n)\big)-n \cocycle^{\nabla}(a;b^2)\odot b^{n-1}
,\end{aligned}\end{equation}
for all $a\in\sections{A}$, $b\in\sections{L/A}$, and $n\in\NN$ with $n\geqslant 1$.
It follows from Lemma~\ref{Malawi} that, for all $a\in\sections{A}$,
the map $s\mapsto\Theta\big(i(a);s\big)$ is an $R$-linear coderivation of $\sections{S(L/A)}$
which preserves the filtration~\eqref{Texas}.
Furthermore, according to Lemma~\ref{Oregon} and Lemma~\ref{Djibouti},
we have $\Theta\big(i(a);1\big)=0$
and $\Theta\big(i(a);c\big)=0$ for all $a\in\sections{A}$ and $c\in\sections{L/A}$
since $\nabla$ is assumed to be torsion-free.
Therefore, setting $n=1$ in Equation~\eqref{monkey}, we obtain
$\Theta\big(i(a);b^{2}\big) = -\cocycle^{\nabla}(a;b^2)$.

\textbf{(2)}
This can be proved by a simple induction based on Equation~\eqref{monkey}.

\textbf{(3)} Finally, to prove the last assertion, we consider the bundle map
\[ H:L/A\otimes S(L/A)\to S(L/A)\] defined by
$H(c;s)=\Theta\big(j(c);s\big)$ for all $c\in\sections{L/A}$ and $s\in\sections{S(L/A)}$.
According to Lemma~\ref{Oregon}, we have $H(c;1)=0$ for all $c\in\sections{L/A}$.
It follows from Lemma~\ref{pig}, the vanishing assumptions on $T^\nabla$; $\scindement$; and $\curvature$,
and Lemma~\ref{Louisiana} that
\begin{equation}\label{warts}\begin{aligned}
\tfrac{n+2}{n+1}H(c;b^{n+1}) =&\ b\odot H(c;b^n)+H\big(b;H(c;b^n)\big) \\
& +\nabla_{j(b)}\big(H(c;b^n)\big)-H(\nabla_{j(b)} c;b^n)-H\big(c;\nabla_{j(b)}(b^n)\big)
.\end{aligned}\end{equation}
Since the map $s\mapsto H(c;s)$ is, according to Lemma~\ref{Malawi},
a coderivation of the $R$-coalgebra $\sections{S(L/A)}$ preserving the filtration~\eqref{Texas},
a simple induction on $n$ based on Equation~\eqref{warts} shows that
$H(c;b^n)=0$ for all $b,c\in\sections{L/A}$ and $n\in\NN$.
Therefore $H=0$. The result then follows from Equation~\eqref{monkey}.
\end{proof}

\begin{lemma}\label{horse}
Given a splitting $j:L/A\to L$ of the short exact sequence \eqref{Sydney}
and a torsion-free $L$-connection $\nabla$ on $L/A$,
the following assertions are equivalent:
\begin{enumerate}
\item\label{eins} $\cocycle^\nabla=0$;
\item\label{zwei} $\Theta\big(i(a);s\big)=0$
for all $a\in\sections{A}$ and $s\in\sections{S(L/A)}$;
\item\label{drei} $\pbw^{\nabla,j}(\nabla^{\Bott}_a s)=i(a)\cdot\pbw^{\nabla,j}(s)$
for all $a\in\sections{A}$ and $s\in\sections{S(L/A)}$;
\item\label{funf} $\pbw^{\nabla,j}:\sections{S(L/A)}
\to\frac{\enveloping{L}}{\enveloping{L}\sections{A}}$
is a morphism of left $A$-modules.
\end{enumerate}
\end{lemma}

\begin{proof}
The equivalence of (\ref{eins}) and (\ref{zwei}) follows immediately from Lemma~\ref{Gouvy}.
It follows immediately from the definitions of $\nabla^{\lightning}$ and $\Theta$,
the assumption that $T^\nabla=0$, and Lemma~\ref{Mascate} that
\begin{multline*}
i(a)\cdot\pbw(s)-\pbw\big(\nabla^{\Bott}_a s\big)
=\pbw\big(\nabla^{\lightning}_{i(a)}s-\nabla^{\Bott}_a s\big) \\
=\pbw\big(\nabla_{i(a)}s+\Theta\big(i(a);s\big)-\nabla^{\Bott}_a s\big)
=\pbw\big(\Theta(i(a);s)\big)
.\end{multline*}
Therefore, (\ref{zwei}) is equivalent to (\ref{drei}).
The equivalence of (\ref{drei}) and (\ref{funf}) is a tautology.
\end{proof}

We return now to the central issue of Theorem~\ref{Ohio}:
how to make the graded manifold $A[1]\oplus L/A$ a Kapranov dg-manifold.

\pagebreak[2]
\begin{proposition}\label{dolphin}
\strut
\begin{enumerate}
\item
Every isomorphism of filtered $R$-coalgebras
\[ \psi:\sections{S(L/A)}\to\tfrac{\enveloping{L}}{\enveloping{L}\sections{A}} \]
gives rise to a homological vector field $D_\psi$ on $A[1]\oplus L/A$
making it a Kapranov dg-manifold
with the Chevalley--Eilenberg differential $d_A$
as homological vector field on the dg-submanifold $A[1]$.
\item
Given any two such isomorphisms of filtered $R$-coalgebras $\phi$ and $\psi$,
there exists an isomorphism of Kapranov dg-manifolds
between $(A[1]\oplus L/A,D_{\phi})$ and $(A[1]\oplus L/A,D_{\psi})$
fixing the dg-submanifold $(A[1],d_A)$.
\end{enumerate}
\end{proposition}

\begin{proof}
\textbf{(1)}
For each $a\in\sections{A}$, we define a coderivation $\delta_a$ of $\sections{S(L/A)}$ by
\[ \delta_a(s)=\psi^{-1}(i(a)\cdot\psi(s)),
\quad\forall a\in\sections{A},s\in\sections{S(L/A)} .\]
According to Lemma~\ref{Nebraska}, the $A$-action
\[ \delta:\sections{A}\times\sections{S(L/A)}\to\sections{S(L/A)} \]
preserves the filtration \eqref{Texas} since $\psi$ respects the filtrations.
Hence, we have $\delta_a(1)=0$ for all $a\in\sections{A}$ by Proposition~\ref{shark}.
The conclusion follows from Proposition~\ref{Tasmania}.
\newline
\textbf{(2)}
For every $a\in\sections{A}$, the automorphism $\phi^{-1}\circ\psi$
of the filtered $R$-coalgebra $\sections{S(L/A)}$
intertwines the coderivations $\phi^{-1}\big(i(a)\cdot\phi(\argument)\big)$
and $\psi^{-1}\big(i(a)\cdot\psi(\argument)\big)$.
Thus, according to Proposition~\ref{Tagalog}, the automorphism $\phi^{-1}\circ\psi$
determines in a natural way an automorphism of the graded manifold $A[1]\oplus L/A$
that fixes the submanifold $A[1]$ and maps the homological vector field $D_\phi$ to $D_\psi$.
\end{proof}

Forearmed with Proposition~\ref{dolphin}, we can now prove Theorem~\ref{Ohio}.
\begin{proof}[Proof of Theorem~\ref{Ohio}]
Taking $\psi=\pbw^{\nabla,j}$ in Proposition~\ref{dolphin},
we obtain the Kapranov dg-manifold $(A[1]\oplus L/A, D)$
associated with the $A$-action by coderivations $\delta^\lightning$
on $\sections{S(L/A)}$ defined by the relation
\begin{equation}\label{Gibraltar}
\delta^\lightning_a(s)=\pbw^{-1}\big(i(a)\cdot\pbw(s)\big)
.\end{equation}
The $A$-action $\delta^\lightning$ on $\sections{S(L/A)}$ determines
an $A$-representation $\nabla^\natural$ on $L/A$ through
the commutative diagram
\[ \begin{tikzcd}
\sections{S(L/A)} \arrow[r, "\delta^\lightning_a"] &
\sections{S(L/A)} \arrow[d, two heads, "\varpi"] \\
\sections{L/A} \arrow[u, hook] \arrow[r, "\nabla^\natural_a"']
& \sections{L/A}
\end{tikzcd} ,\]
where $\varpi$ denotes the canonical projection of $S(L/A)$
onto $S^1(L/A)\cong L/A$.
According to Proposition~\ref{Tasmania}
and Equation~\eqref{Scandinavia} in particular,
we have $D=d_A^{\nabla^\natural}+\cR$, where $d_A^{\nabla^\natural}$
denotes the covariant differential of the $A$-representation
on $\hat{S}\big((L/A)^\vee\big)$ induced by the $A$-representation
$\nabla^\natural$ on $L/A$
and $\cR=\sum_{k\geqslant 2} \cR_k$
is a solution of the Maurer--Cartan Equation~\eqref{eq:MC}.

It follows from Lemma~\ref{Daegu} that,
for all $a\in\sections{A}$ and $b\in\sections{L/A}$,
\[ \nabla^\natural_a b = \varpi(\delta^\lightning_a b)
= \varpi(\nabla^{\lightning}_{i(a)} b)
= \varpi\big(\nabla^{\Bott}_a b\big)=\nabla^{\Bott}_a b .\]
It follows from Equations~\eqref{Scandinavia}, \eqref{Mali},
\eqref{Poland}, \eqref{Gibraltar}, and~\eqref{California} that
\[ \begin{split} \duality{i_a\cR(\varepsilon)}{s}
=& \duality{i_a(D-d^{\nabla^\natural}_A)(\varepsilon)}{s} \\
=& \duality{i_a(D-d^{\nabla^{\Bott}}_A)(\varepsilon)}{s} \\
=& \duality{\varepsilon}{\nabla^{\Bott}_a s-\delta^\lightning_a s} \\
=& \duality{\varepsilon}{\nabla^{\Bott}_a s-\nabla^\lightning_{i(a)} s} \\
=& \duality{\varepsilon}{\nabla^{\Bott}_a s-\nabla_{i(a)}s-\Theta(i(a);s)}
\end{split} \]
and thus
\begin{equation}\label{Quaregnon}
\cR(a;s)=\varpi\big(\nabla^{\Bott}_a s-\nabla^\lightning_{i(a)} s\big)
=\varpi\big(\nabla^{\Bott}_a s-\nabla_{i(a)} s-\Theta(i(a);s)\big)
,\end{equation}
for all $a\in\sections{A}$, $\varepsilon\in\sections{(L/A)^\vee}$, and $s\in\sections{S(L/A)}$.
Therefore, we have
\[ D = d_A^{\nabla^{\Bott}} + \sum_{k\geqslant 2} \cR_k \]
with
\begin{equation}\label{whale}
\cR_k(a;b^k)=-\varpi\big(\Theta(i(a);b^k)\big),\quad\forall k\geqslant 2
,\end{equation}
for all $a\in\sections{A}$ and $b\in\sections{L/A}$.
Provided the $L$-connection $\nabla$ on $L/A$ is torsion-free, we also have
\[ \cR_2(a;b^2)=-\varpi\big(\Theta(i(a);b^2)\big)=\cocycle^{\nabla}(a;b^2) \]
by Lemma~\ref{Gouvy}.
It follows from Proposition~\ref{prop:recursion} that, 
for any $k\geqslant 3$, $\cR_k$ can be computed explicitly 
as an algebraic function
of $\scindement$, $\cocycle^{\nabla}$, $\curvature$, their higher covariant
derivatives, and compositions thereof.

Finally, let $D'$ denote the homological vector field
on $A[1]\oplus L/A$ arising from another choice of connection
$\nabla'$ and splitting $j'$.
Then, Proposition~\ref{dolphin} warrants the existence
of an isomorphism of Kapranov dg-manifolds
between $(A[1]\oplus L/A,D)$ and $(A[1]\oplus L/A,D')$
fixing the dg-submanifold $(A[1],d_A)$.
This completes the proof of Theorem~\ref{Ohio}.
\end{proof}

\begin{proof}[Proof of Addendum to Theorem~\ref{Ohio}]
This follows from Theorem~\ref{Madagascar}
and Proposition~\ref{schist}.
\end{proof}

\begin{proof}[Proof of Proposition~\ref{zebra}]
Since the $L$-connection $\nabla=\nabla^{\Bott}+\ddot{\nabla}$ on $B$ is torsion-free
and the splitting $j$ identifies $L/A$ to $B$,
there exists, according to Theorem~\ref{Ohio}, a Kapranov dg-manifold structure on $A[1]\oplus B$
with homological vector field $D=d_A^{\Bott}+\sum_{k=2}^\infty\cR_k$
satisfying $\cR_2=\cocycle^\nabla$ and Equation~\eqref{whale}.
Since $B$ is a Lie subalgebroid of $L$
and the $B$-connection $\ddot{\nabla}$ on $B$ is flat,
Equation~\eqref{worm} holds.
Applying the projection $\varpi:S(B)\onto S^1(B)$ to both sides of Equation~\eqref{worm}
and making use of Equation~\eqref{whale}, we deduce that
\[ \cR_{k+1}(a;b^{k+1})=\ddot{\nabla}_b\big(\cR_k(a;b^k)\big)
-\cR_k\big(\Delta^{\Bott}_b a;b^k\big)-\cR_k\big(a;\ddot{\nabla}_b(b^k)\big) \]
for all $a\in\sections{A}$, $b\in\sections{B}$, and $k\in\NN$ with $k\geqslant 2$.
Indeed, for $k\geqslant 2$, we have
\[ \varpi\big(b\odot\Theta(i(a);b^k)-k\cocycle^\nabla(a;b^2)\odot b^{k-1}\big)=0 \]
since $\cocycle^\nabla(a;b^2)\odot b^{k-1}\in\sections{S^{\geqslant 2}(B)}$
and $\Theta(i(a);b^k)\in\sections{S^{\geqslant 1}(B)}$ by Lemma~\ref{Grez-Doiceau}.
Hence, we have
\[ \begin{split} \cR_{k+1}(a;b_0\odot b_1\odot\cdots\odot b_k)
= \sum_{j=0}^k \Big\{ & \ddot{\nabla}_{b_j}
\big(\cR_k(a;b_0\odot\cdots\odot\widehat{b_{j}}
\odot\cdots\odot b_k)\big) \\
& -\cR_k\big(\Delta^{\Bott}_{b_j}a;b_0\odot\cdots\odot\widehat{b_{j}}
\odot\cdots\odot b_k\big) \\
& -\cR_k\big(a;\ddot{\nabla}_{b_j}(b_0\odot\cdots\odot\widehat{b_{j}}
\odot\cdots\odot b_k)\big) \Big\}
\end{split} \]
for all $a\in\sections{A}$, $b_0,b_1,\dots,b_k\in\sections{B}$,
and $k\geqslant 2$.
The dual of the vector bundle morphism
\[ S^{k+1}(B)\ni b_0\odot\cdots\odot b_k
\longmapsto \sum_{j=0}^k b_j\otimes(b_0\odot\cdots\odot
\widehat{b_j}\odot\cdots\odot b_k)\in B\otimes S^k(B) \]
is the symmetrization map
\[ \sym:B^\vee\otimes S^k(B^\vee)\to S^{k+1}(B^\vee) \]
sending $b\otimes(b_1\odot b_2\odot\cdots\odot b_k)$
to $b\odot b_1\odot b_2\odot\cdots\odot b_k$.
Therefore, we conclude that, for $k\geqslant 2$,
\[ \cR_{k+1}=(\id_{A^\vee\otimes B}\otimes\sym)
\big(d_B^{\widehat{\nabla}}\cR_k\big) ,\]
where $\widehat{\nabla}$ denotes the $B$-connection on
$A^\vee\otimes B\otimes S^k(B^\vee)$ induced by the
$B$-connections $\Delta^{\Bott}$ on $A$ and $\ddot{\nabla}$ on $B$.
\end{proof}

\begin{proof}[Proof of Theorem~\ref{Montana}]
\quad\textbf{(\ref{een})$\Rightarrow$(\ref{twee})}\quad According to Corollary~\ref{Cantonese},
the characteristic class of a linearizable Kapranov dg-manifold vanishes.
Moreover, according to Theorem~\ref{Ohio}, the characteristic class
of the Kapranov dg-manifold arising from a Lie pair is the Atiyah class
of the Lie pair.
Therefore, if the Kapranov dg-manifold arising from a Lie pair $(L,A)$ is linearizable,
the Atiyah class of the Lie pair must vanish.
\newline
\quad\textbf{(\ref{twee})$\Rightarrow$(\ref{drie})}\quad Fix a splitting $j$.
Since the Atiyah class of the Lie pair $(L,A)$ is assumed to be zero,
there exists, according to \cite[Theorem~2.5]{MR3439229}, an $L$-connection $\nabla$ on $L/A$
extending the Bott representation and such that the associated $1$-cocycle $\cocycle^\nabla$ vanishes.
Without loss of generality, we may assume that $\nabla$ is torsion-free
--- see the proof of Proposition~\ref{Vilnius}.
It follows from Lemma~\ref{horse} that the resulting map $\pbw^{\nabla,j}$
is an isomorphism of $A$-modules.
\newline
\quad\textbf{(\ref{drie})$\Rightarrow$(\ref{een})}\quad
According to Lemma~\ref{horse},
the map $\pbw^{\nabla,j}$ is a morphism of $A$-modules
if and only if $\Theta\big(i(a);s\big)=0$ (i.e.\ $\nabla^{\lightning}_{i(a)}s=\nabla^{\Bott}_a s$)
for all $a\in\sections{A}$ and $s\in\sections{S(L/A)}$.
It then follows from Equations~\eqref{Quaregnon}
and~\eqref{whale} that $D = d_A^{\nabla^{\Bott}}$, which shows that the
Kapranov dg-manifold $(A[1]\oplus L/A,D)$ is linear.
\end{proof}

\begin{proof}[Proof of Addendum to Theorem~\ref{Montana}]
In the real case, the equivalence of \textbf{(\ref{drie})}
and \textbf{(\ref{vier})} follows immediately
from Theorem~\ref{Madagascar} and the description
of the actions at play given in Propositions~\ref{Bott_vf}
and~\ref{Fedosov_vf}.
It is obvious that \textbf{(\ref{vier})}
implies \textbf{(\ref{vijf})}.
Finally, it follows from the last assertion of the Addendum to Theorem~\ref{Ohio}
that \textbf{(\ref{vijf})} implies \textbf{(\ref{een})}.
\end{proof}

\subsection{Example from complex geometry}
\label{Amiens}

Every complex manifold $X$ determines a Lie pair $(L=T_X\otimes\CC,A=T\zo_X)$
for which $L/A\cong T\oz_X$, the Bott connection
\[ \nabla^{\Bott}:\sections{T\zo_X}\times\sections{T\oz_X}\to\sections{T\oz_X} \]
is precisely the Dolbeault operator (i.e.\ $d_A^{\nabla^{\Bott}}=\overline{\partial}$),
and the Atiyah class $\alpha_{L/A}$ is the classical Atiyah class
of the holomorphic tangent bundle $T_X$ --- see~\cite{MR3439229}.
Therefore, as a consequence of Theorems~\ref{Ohio} and~\ref{Montana}, we obtain

\begin{theorem}
Let $X$ be a complex manifold.
\begin{enumerate}
\item Every torsion-free $T\oz_X$-connection $\nabla\oz$ on $T\oz_X$ determines
a Kapranov dg-manifold structure on $T\zo_X[1]\oplus T\oz_X$
with homological vector field $D=\overline{\partial}+\sum_{k\geqslant 2}\cR_k$,
where $\cR_2=\cocycle^{\nabla^{\Bott}+\nabla\oz}$ is the Atiyah $1$-cocycle
associated with the $T_X\otimes \CC$-connection $\nabla^{\Bott}+\nabla\oz$
on $T\oz_X$ and all \[ \cR_k \in \Omega^{0, 1}\big(S^k{(T\oz_X)}^\vee\otimes T\oz_X\big) \]
with $k\geqslant 3$ 
can be computed explicitly as algebraic functions of $\cR_2$, 
the curvature of $\nabla\oz$, their higher covariant derivatives, 
and compositions thereof.
\item The Dolbeault complex $\OO^{0,\bullet}(T\oz_X)$ admits an $L_\infty[1]$ algebra structure.
For $n\geqslant 2$, the $n$-th multibracket $\lambda_n$
is the composition of the wedge product
\[ \OO^{0,j_1}(T\oz_X)\otimes\cdots\otimes\OO^{0,j_n}(T\oz_X)
\to\OO^{0,j_1+\cdots+j_n}(\otimes^n T\oz_X) \]
with the map
\[ \OO^{0,j_1+\cdots+j_n}(\otimes^n T\oz_X)\to\OO^{0,j_1+\cdots+j_n+1}(T\oz_X) \]
induced by $\cR_n\in\OO\zo\big(\Hom(\otimes^n T\oz_X,T\oz_X)\big)$,
while the unary bracket $\lambda_1$ is the Dolbeault operator
$\overline{\partial}:\Omega^{0,j}(T\oz_X)\to\Omega^{0,j+1}(T\oz_X)$.
\item
The Kapranov dg-manifold structures on $T\zo_X[1]\oplus T\oz_X$
resulting from different choices of torsion-free connections $\nabla\oz$ are all isomorphic.
Hence the $L_\infty[1]$ algebra structures on the Dolbeault complex $\OO^{0,\bullet}(T\oz_X)$
resulting from different choices of $\nabla\oz$ are isomorphic.
\item The Kapranov dg-manifold $T\zo_X[1]\oplus T\oz_X$ is linearizable
if and only if the Atiyah class of the holomorphic tangent bundle $T_X$ vanishes.
\end{enumerate}
\end{theorem}

When $X$ is a Kähler manifold, the Levi-Civita connection $\nabla^{\LC}$
induces a $T\oz_X$-connection $\nabla\oz$ on $T\oz_X$ as follows.
First, extend the Levi-Civita connection $\CC$-linearly
to a $T_X\otimes\CC$-connection $\nabla^\CC$ on $T_X\otimes\CC$.
Since $X$ is Kähler, the almost complex structure $J$ on $X$
is parallel (i.e.\ $\nabla^{\LC}J=0$), and therefore $\nabla^\CC$
restricts to a $T_X\otimes\CC$-connection on $T\oz_X$.
It is easy to check that the induced $T\zo_X$-connection on $T\oz_X$
is the canonical representation of the Lie algebroid $T\zo_X$
on the vector bundle $T\oz_X$ --- a local section of $T\oz_X$ is $T\zo_X$-horizontal
if and only if it is holomorphic --- while the induced $T\oz_X$-connection
$\nabla\oz$ on $T\oz_X$ is flat and torsion-free.
Thus according to Proposition~\ref{zebra},
we have\footnote{Kapranov writes $\cR_{n+1}=d^{\nabla^{1,0}}\cR_n$ instead.}
\[ \cR_{n+1}=(\id_{(T\zo_X)^\vee\otimes T\oz_X}\otimes\sym)
\big(d_{T\oz_X}^{\widehat{\nabla}}\cR_n\big)
,\quad\text{for}\ n\geqslant 2 .\]

Hence, we recover a classical theorem of Kapranov~\cite[Theorem~2.6]{MR1671737}.

\appendix

\section{Coderivations of the symmetric coalgebra}

Given a module $V$ over a commutative ring $R$ with identity $1$,
its symmetric tensor algebra $S(V)$ admits a natural $R$-linear comultiplication:
the deconcatenation map $\Delta:S(V)\to S(V)\otimes_R S(V)$ defined by Equation~\eqref{London}.
For every $v\in V$ and every non-negative integer $n$, we have
\[ \Delta(v^n)=\sum_{k=0}^{n}\binom{n}{k} v^k\otimes v^{n-k} ,\]
where $v^k$ denotes the symmetric product of $k$ copies of $v$
while $v^0$ denotes the unit element $1$ of the ring $R$.
Thus $\big(S(V),\Delta\big)$ is a $R$-coalgebra.

The purpose of this appendix is to establish Proposition~\ref{shark} below.

The following result is immediate.
\begin{lemma}\label{squirrel}
$\ker\big(\Delta-(1\otimes\id+\id\otimes 1)\big)=S^1(V)$
\end{lemma}

\begin{lemma}\label{chicken}
If $\delta\in\End_R\big(S(V)\big)$ is a coderivation
of the symmetric $R$-coalgebra $\big(S(V),\Delta\big)$,
then $\delta(1)\in S^1(V)$.
\end{lemma}

\begin{proof}
We have $\delta(1)=\varepsilon+v+H$ for some $\varepsilon\in R$,
$v\in V$, and $H\in S^{\geqslant 2}(V)$.
Since $\delta$ is a coderivation, we have
\[ \Delta\circ\delta(1)=(\delta\otimes\id+\id\otimes\delta)\circ\Delta(1) .\]
Since $\Delta (v)=v\otimes 1+1 \otimes v$ and $\Delta (1)=1 \otimes 1$,
it follows that
\[ \varepsilon\cdot 1\otimes 1+\Delta(H)=\varepsilon\cdot 1\otimes 1+
H\otimes 1+\varepsilon\cdot 1\otimes 1+1\otimes H ,\]
which implies that
\[ \Delta(H)-(1\otimes H+H\otimes 1)=\varepsilon\cdot 1\otimes 1 .\]
In this last equation, the right hand side belongs to $S^0(V)\otimes S^0(V)$ while the left hand side
belongs to $S^{\geqslant 1}(V)\otimes S^{\geqslant 1}(V)$.
Therefore, we obtain:
\[ \left\{
\begin{array}{l}
\Delta(H)-(1\otimes H+H\otimes 1)=0 \\
\varepsilon\cdot 1\otimes 1 =0
\end{array}
\right. .\]
From Lemma~\ref{squirrel}, it follows that $\varepsilon=0$ and $H=0$.
Hence $\delta(1)=v\in S^1(V)$.
\end{proof}

\begin{proposition}
\label{shark}
A coderivation $\delta$ of the symmetric $R$-coalgebra
$\big(S(V),\Delta\big)$ preserves the filtration
\[ R \into S^{\leqslant 1}(V) \into S^{\leqslant 2}(V)
\into S^{\leqslant 3}(V) \into \cdots \]
if and only if $\delta(1)=0$.
\end{proposition}

\begin{proof}
Assume $\delta\big(S^{\leqslant n}(V)\big)\subset S^{\leqslant n}(V)$
for all non-negative integers $n$.
Then $\delta(1)\in S^0(V)$ as $1\in S^0(V)$.
Since $\delta$ is a coderivation, we have $\delta(1)\in S^1(V)$ by Lemma~\ref{chicken}.
Therefore $\delta(1)=0$.

Conversely, assume that $\delta(1)=0$.
Then, obviously, $\delta\big(S^0(V)\big)=\{0\}\subset S^0(V)$.
Given $w\in V$, we have $\delta(w)=\varepsilon+v+H$ for some
$\varepsilon\in R$, $v\in V$, and $H\in S^{\geqslant 2}(V)$.
Since $\delta$ is a coderivation, we have
$\Delta\circ\delta(v)=(\delta\otimes\id+\id\otimes\delta)\circ\Delta(v)$.
By the same argument as in the proof of Lemma~\ref{chicken}, we have $H=0$ and $\epsilon =0$.
Hence $\delta(w)=v\in V$. This shows that $\delta\big(S^1(V)\big)\subset S^1(V)$.
For $n\geqslant 2$, we proceed by induction.
Suppose that, for any $v\in V$, $\delta(v^k)\in S^{\leqslant k}(V)$ for $0\leqslant k\leqslant n-1$.
We write $\delta(v^n)=x+y$, where $x\in S^{\leqslant n}(V)$ and $y\in S^{>n}(V)$.
Since $\delta$ is a coderivation:
\begin{align*}
\Delta(x+y)=&(\delta\otimes\id+\id\otimes\delta)
\big(1\otimes v^n+\sum_{k=1}^{n-1}\binom{n}{k} v^k\otimes v^{n-k}+v^n\otimes 1 \big).
\end{align*}
This implies that
\[ \Delta(y)-(1\otimes y+y\otimes 1)=1\otimes x+x\otimes 1-\Delta(x)
+\sum_{k=1}^{n-1}\binom{n}{k}\big(\delta(v^k)\otimes v^{n-k}+v^k\otimes\delta(v^{n-k})\big) .\]
In this last equation, while the left hand side belongs
to $\bigoplus_{p+q>n}\big(S^p(V)\otimes S^q(V)\big)$,
the right hand side belongs to $\bigoplus_{p+q\leqslant n}\big(S^p(V)\otimes S^q(V)\big)$
by the induction hypothesis.
Therefore, $y\in S^{>n}(V)\cap\ker\big(\Delta-(1\otimes\id+\id\otimes 1)\big)$.
It follows from Lemma~\ref{squirrel} that $y=0$.
Hence $\delta(v^n)=x\in S^{\leqslant n}(V)$.
\end{proof}

\printbibliography

\end{document}